\documentclass[12pt]{article}

\usepackage{fullpage}
\usepackage{amsmath}
\usepackage{amssymb}
\usepackage{amsthm}
\usepackage{enumerate}
\usepackage{paralist}
\usepackage{xcolor}
\usepackage{changes}
\usepackage{authblk}
\usepackage{lineno}
\linenumbers

\usepackage{comment}
\usepackage{authblk}
\definechangesauthor[color=blue]{mk}
\definechangesauthor[color=red]{dm}
\definechangesauthor[color=magenta]{al}

\usepackage{biblatex} 
\usepackage{hyperref}

\newcommand{\EE}{\mathbb{E}}
\newcommand{\PP}{\mathbb{P}}
\newcommand{\Var}{\mathbb{V}}

\newcommand{\NN}{\mathbb{N}}
\newcommand{\RR}{\mathbb{R}}

\renewcommand{\P}{\mathbb{P}}
\renewcommand{\ss}{\mathfrak{z}}
\newcommand{\dt}{\mathcal{D}^{(\ss)}}

\newcommand{\ndt}{\overline{\mathcal{D}}^{(\ss)}}

\newcommand{\mkomit}[1]{{\color{orange}#1}}
\newcommand{\mkor}[2]{#2}
\newcommand{\nnu}{n_{\mkomit{\nu}}}
\newcommand{\mint}[1]{I^r(s)}

\newcommand{\rabs}{\mathfrak{r}}

\newcommand{\auxy}{y}
\newcommand{\auxY}{Y}
\newcommand{\yabs}{{\bf y}}

\newcommand{\trad}{T^{rad}_0}
\newcommand{\tang}{T^{ang}_0}
\newcommand{\phis}{\phi^{(\ss)}}
\newcommand{\hatr}{\widehat{r}_0}
\newcommand{\hatt}{\widehat{\theta}_0}
\newcommand{\JJ}{J}
\newcommand{\II}[1]{I(#1)}
\newcommand{\tII}[1]{\widetilde{I}(#1)}
\newcommand{\tI}{\widetilde{I}}

\newcommand{\KA}{\kappa_A}
\newcommand{\KR}{\kappa_R}
\newcommand{\auxc}{\mathfrak{C}}
\newcommand{\auxl}{\mathfrak{L}}

\newtheorem{theorem}{Theorem}
\newtheorem{lemma}[theorem]{Lemma}
\newtheorem{observation}[theorem]{Observation}
\newtheorem{fact}[theorem]{Fact}
\newtheorem{proposition}[theorem]{Proposition}
\newtheorem{corollary}[theorem]{Corollary}
\newtheorem{remark}[theorem]{Remark}

\DeclareMathOperator{\cosech}{cosech}
\DeclareMathOperator{\cotanh}{coth}

\usepackage{tikz}
  \usetikzlibrary{calc}
  \usetikzlibrary{patterns}
  \usetikzlibrary{arrows,decorations.markings}
\usepackage{pgfplots}
  \pgfplotsset{compat=1.9}
  \usepgfplotslibrary{polar}

\makeatletter
\usepackage{hyperref}
\newcommand{\mycom}[2]{\hypertarget{#1}{#2}\global\@namedef{mycom@#1}{#2}}
\newcommand{\linktomycom}[1]{%
\@ifundefined{mycom@#1}{\textbf{??}\@latex@warning{Reference `#1' on page \thepage \space undefined}}%
{\hyperlink{#1}{\@nameuse{mycom@#1}}}%
}
\makeatother

\begin{document}
\nolinenumbers
\title{Tail bounds for detection times in mobile hyperbolic graphs}

\author[1]{Marcos Kiwi\thanks{Depto.~de Ingenier\'ia Matem\'atica and Centro de Modelamiento Matem\'atico (CNRS IRL2807), Univ.~Chile, (\texttt{mkiwi@dim.uchile.cl}). Gratefully acknowledges support by ACE210010 and FB210005, BASAL funds for centers of excellence from ANID-Chile, and by GrHyDy ANR-20-CE40-0002.}}
\author[2]{Amitai Linker\thanks{Depto.~de Matem\'aticas, Facultad de Ciencias Exactas, Univ.\ Andr\'es Bello, (\texttt{amitai.linker@unab.cl}). Gratefully acknowledges support by IDEXLYON of Univ.\ de Lyon (Programme Investissements d'Avenir ANR16-IDEX-0005), and by DFG project number 425842117.}} 
\author[3]{Dieter Mitsche\thanks{Institut Camille Jordan, Univ.\ Jean Monnet, Univ.\ de Lyon and IMC, Pontif\'{i}cia Univ.\ Cat\'{o}lica de Chile, (\texttt{dmitsche@gmail.com}). Gratefully acknowledges support by grant GrHyDy ANR-20-CE40-0002 and by IDEXLYON of Univ.\} de Lyon (Programme Investissements d'Avenir ANR16-IDEX-0005).}}
\affil[1]{Univ.\ Chile}
\affil[2]{Univ.\ Andres Bello}
\affil[3]{Univ.\ de Lyon and PUC Chile}
\maketitle 

\begin{abstract}
Motivated by Krioukov et al.'s model of random hyperbolic graphs~\cite{KPKVB10} for real-world networks, and inspired by the analysis of a dynamic model of graphs in Euclidean space by Peres et al.~\cite{Peres2010}, we introduce a dynamic model of hyperbolic graphs in which vertices are allowed to move according to a Brownian motion maintaining the distribution of vertices in hyperbolic space invariant. For different parameters of the speed of angular and radial motion, we analyze tail bounds for detection times of a fixed target and obtain a complete picture, for very different regimes, of how and when the target is detected: as a function of the time passed, we characterize the subset of the hyperbolic space where particles typically detecting the target are initially located. Our analysis shows that our dynamic model exhibits a phase transition as a function of the relation of angular and radial speed.

We overcome several substantial technical difficulties not present in Euclidean space, and provide a complete picture on tail bounds. On the way, moreover, we obtain results for a class of one-dimensional continuous processes with drift and reflecting barrier, concerning the time they spend within a certain interval. We also derive improved bounds for the tail of independent sums of Pareto random variables.
\end{abstract}

\section{Introduction}\label{intro}
\emph{Random Geometric Graphs (RGGs)} are a family of spatial networks that have been intensively studied as models of communication networks, in particular sensor networks, see for example~\cite{Akyildiz}.
In this model, an almost surely finite number of vertices is distributed in a metric space according to some 
probability distribution and two vertices are joined by an edge if the distance between them is at most a given parameter called
radius.
Typically, the metric space is the $d$-dimensional unit cube or torus (often with $d=2$) and points are chosen according to a Poisson point process of given intensity. 
While simple, RGGs do capture relevant characteristics of some real world networks, for instance, non-negligible clustering coefficient.
However, RGGs fail to exhibit other important features such as scale-freeness and non-homogeneous vertex degree distribution, both of which are staple features
of a large class of networks loosely referred to as "social networks" that are meant to encompass networks such as the Internet, citation networks, friendship relation among individuals, etc. 

A network model that naturally exhibits clustering and scale-freeness is the \emph{Random Hypebolic Graph (RHG)} model introduced by Krioukov et al.~\cite{KPKVB10} where vertices of the network are points in a bounded region of the hyperbolic plane, and connections exist if their hyperbolic distance is small. In~\cite{BPK10}, a surprisingly good maximum likelihood fit of the hyperbolic model was shown for the embedding of the network corresponding to the autonomous systems of the Internet,
drawing much attention, interest, and follow-up work on the model (see the section on related work below). 

It has been recognized that in many applications of geometric network models the entities represented by vertices are not fixed in space but mobile. One way in which this has been addressed is to assume that the vertices of the network move according to independent Brownian motions, giving rise to what Peres et al. in~\cite{Peres2010}) call the \emph{mobile geometric graph} model.
Mobility, and more generally dynamic models, are even more relevant in the context of social networks. Thus, it is natural to adapt the mobile geometric graph setting to the hyperbolic graph context and assume the vertices of the latter graphs again move according to independent Brownian motions but in hyperbolic space. This gives rise, paraphrasing Peres et al., to the \emph{mobile hyperbolic graph} model. 
We initiate the study of this new model by focusing on the fundamental
problem of \emph{detection}, that is, the time until a fixed (non-mobile) added target vertex becomes non-isolated in the (evolving) hyperbolic graph.
We will do this; but in fact do much more. In order to discuss our contributions in detail we first need to precisely describe the model we introduce and formalize the main problem we address in our study. 


\subsection{The mobile hyperbolic graph model}\label{sec:model}
We first introduce the model of Krioukov et al.~\cite{KPKVB10} in its Poissonized version (see also~\cite{GPP12} for the same description in the so called uniform model): for each $n \in \mathbb{N}^+$, consider a Poisson point process $\mathcal{P}$ on the hyperbolic disk of radius $R :=2 \log (n/\nu)$ for some positive constant~$\nu \in \mathbb{R}^+$ ($\log$ denotes here and throughout the paper the natural logarithm).
The intensity function $\mu$ at polar coordinates $(r,\theta)$ for 
  $0\leq r\leq R$ and $-\pi \leq \theta < \pi$ is equal to $n f(r,\theta)$, where $f(r,\theta)$ is given by
\begin{align*}
f(r,\theta) & := \begin{cases}
  \displaystyle
 \frac{\alpha \sinh(\alpha r)}{2\pi(\cosh(\alpha R)-1)}, 
  &\text{if $0\leq r\leq R$}, \\[2ex]
  0, & \text{otherwise.}
  \end{cases}
\end{align*}
In other words, the angle and radius are chosen independently; the former uniformly at random in $(-\pi,\pi]$ and the latter with density proportional to $\sinh(\alpha r)$. 
Next, identify the points of the Poisson process with vertices
and define the following graph $G_n:=(V_n,E_n)$ 
where $V_n:=\mathcal{P}$. For $P, P'\in V_n$, $P \neq P'$, with polar coordinates $(r,\theta)$ and $(r',\theta')$ respectively, there is an edge in $E_n$ with endpoints 
  $P$ and $P'$ provided the hyperbolic distance $d_H$ between $P$ and $P'$ is such that  $d_H\leq R$, where $d_H$ is obtained by solving 
\begin{equation}\label{eqn:coshLaw}
\cosh d_H := \cosh r\cosh r'-
  \sinh r\sinh r'\cos( \theta{-}\theta').
\end{equation}
In particular, note that $\EE(|V_n|)=n$.


Henceforth, we denote the point whose radius is $0$ by $O$ and refer to it as the \emph{origin}. For a point $P$ and $r\geq 0$ we let $B_P(r)$ denote the ball centered at $P$ of radius $r$, that is, the set of all points at hyperbolic distance less than $r$ from $P$. 
Also, we henceforth denote the boundary of $B_P(r)$ by $\partial B_P(r)$.

\medskip
To define a dynamic version of Krioukov et al.'s model we consider an initial configuration~$\mathcal{P}$ of the Poissonized model in $B_O(R)$ and then associate to each $x_0:=(r_0,\theta_0)\in\mathcal{P}$ a particle that evolves independently of other particles following a trajectory given in radial coordinates by $x_t:=(r_t,\theta_t)$ at time $t$. At a microscopic level a natural choice for the movement of a particle is that of a random walk in $B_O(R)$ with $\partial B_O(R)$ acting as a reflecting boundary, and where at each step the particle can move either in the radial or angular direction. Assuming that the movement does not depend on the angle 
and that there is no angular drift, we conclude that at a macroscopic level particles should move according to a generator of the form
\[\Delta_{h} = \frac{1}{2}\frac{\partial^2}{\partial r^2}+\frac{\alpha}{2}\frac{1}{\tanh(\alpha r)}\frac{\partial}{\partial r}+\frac12\sigma^2_\theta(r)\frac{\partial^2}{\partial\theta^2},\]
where the drift term in the radial component is chosen so that $f(r)drd\theta$ remains the stationary distribution of the process (this can be checked using the Fokker-Planck equation). The function $\sigma^2_{\theta}(\cdot)$ is unrestricted and relates to the displacement given by the angular movement at a given position $(r,\theta)$. In this sense a natural choice is to take $\sigma^2_\theta(r)$ proportional to $\sinh^{-2}(r)$ which follows by taking the displacement proportional to the hyperbolic perimeter at that radius. An alternative, however, is to take $\sigma^2_\theta(r)$ proportional to $\sinh^{-2}(\alpha r)$ 
corresponding to the (re-scaled) Brownian motion in hyperbolic space.
In order to capture both settings, we introduce an additional parameter 
$\beta$ and work throughout this paper with the following generalized generator:
\begin{equation}\label{truegenerator}
	\Delta_{h} := \frac{1}{2}\frac{\partial^2}{\partial r^2}+\frac{\alpha}{2}\frac{1}{\tanh(\alpha r)}\frac{\partial}{\partial r}+\frac{1}{2\sinh^2(\beta r)}\frac{\partial^2}{\partial\theta^2}
	\end{equation}
where $\beta>0$ is a new parameter related to the velocity of the angular movement which can alternatively be understood as a deformation of the underlying space.
We shall see that thus enhancing our model yields a broader range of behavior and exhibits phase transition phenomena  
(for detection times this is explained in detail in our next section where we summarize our main results).

Fix now an initial configuration of particles $\mathcal{P}$ located at points in $B_O(R)$. 
We denote by $\PP_{x_0}$ the law of a particle initially placed at a given point $x_0:=(r_0, \theta_0)\in B_O(R)$. 
We have one more fixed target $Q$, located at the boundary of $B_O(R)$ (that is, $r_Q=R$), and at angular coordinate $\theta_Q:=0$.
For any $s>0$, let  $\mathcal{P}_s\subseteq\mathcal{P}$ be the set of points that have \emph{detected} the target $Q$ by time $s$: that is, for each $P \in \mathcal{P}_s$ there exists a time instant $0 \le t \le s$, so that $x_t \in B_{Q}(R)$. Note that if $t=0$, then $P$ might be in the interior of $B_Q(R)$, whereas if~$t > 0$, the first instant at which $P$ detects $Q$ is when $P$ is at the boundary of $B_Q(R)$, that is $x_t\in\partial B_Q(R)$. This instant $t$ is also called the \emph{hitting time} of $B_Q(R)$ by particle $P$. 
The \emph{detection time} of $Q$, denoted by $T_{det}$, is then defined as the minimum hitting time of  $B_{Q}(R)$ among all initially placed particles. We are particularly interested in the tail probability
$
\PP_{x_0}(T_{det}\ge\ss)
$
for different values of $\ss$ (note that $\ss$ is a function on $n$). In words, we are interested in the tail probability that no particle initially located at position $x_0$ detects the target $Q$ by time $\ss$.

Observe that any given particle $P$ evolving according to the generator $\Delta_h$ specified in~\eqref{truegenerator} will eventually detect the target at some point, so that $\PP(T_{det}\ge\ss)\to 0$ when $\ss\to\infty$ as soon as there is at least one particle. Following what was done in~\cite{Peres2010} for a similar model on Euclidean space, our main result determines the speed at which $\PP(T_{det}>\ss)$ tends to zero as a function of $\ss$.
We consider the same setting of~\cite{Peres2010}, that is $\ss/\EE(T_{det})\to \infty$ but we have to deal with several additional, both qualitatively and quantitatively different, new issues that arise due to the dependency 
of $\ss$ on $n$.

\subsection{Main results}\label{sec:results}
In this section, we present the main results we obtain regarding tail probabilities for detection time, both for the mobile hyperbolic graph model we introduced in the previous section and for two restricted instances: one where the radial coordinate of particles does not change over time and another where the angular coordinate  does not change. We also discuss the relation between the main results and delve into the insights they provide concerning the dominant mechanism (either angular movement, radial movement or a combination of both) that explain the different asymptotic regimes, depending on the relation between $\alpha$ and~$\beta$. We point out that we do not present in this section a complete list of other significant results, in particular the ones that provide a detailed idea of the initial location of particles that typically detect the target. These last results will be discussed at the start of Sections~\ref{sec:angular}, \ref{sec:radial}, and~\ref{sec:mix}
where, in fact, we prove slightly stronger results than the ones stated in this section. Neither do we delve here into those results concerning one-dimensional processes with non-constant drift and a reflecting barrier that might be useful in other settings and thus of independent interest (these results are found in Section~\ref{sec:radial}).


\smallskip
We begin with the statement describing the (precise) behavior of the detection time tail probability depending on how the model parameters $\alpha$ and $\beta$ relate to each other:\footnote{We use the standard Bachmann--Landau asymptotic (in $n$) notation of $O(\cdot)$, $o(\cdot)$, $\Omega(\cdot)$, $\omega(\cdot)$, $\Theta(\cdot)$,  with all terms inside asymptotic expressions being positive.}
\begin{theorem}\label{thm:intro-mixed}
Let $\alpha\in (\frac12,1]$, $\beta >0$,  $\ss:=\ss(n)$, and assume that particles move according to the generator 
$\Delta_h$ in~\eqref{truegenerator}.
Then, the following hold:
\begin{enumerate}[(i)]
    \item\label{thm:mixed-itm-ssmall} 
    For $\beta\leq\frac12$, if $\ss=\Omega((e^{\beta R}/n)^2)\cap O(1)$, then
$\displaystyle
\PP (T_{det} \ge \ss)=\exp\big({-}\Theta(ne^{-\beta R}\sqrt{\ss})\big).\label{ssmall}
$
\smallskip
\item
For $\beta\leq\frac12$ and $\ss=\Omega(1)$ the tail exponent depends on the relation between $\alpha$ and $2\beta$ as follows: 
\begin{enumerate}
    \item\label{thm:mixed-itm1}
    For $\alpha<2\beta$, if $\ss=O(e^{\alpha R})$, then 
    $\displaystyle \PP (T_{det} \ge \ss)=\exp\big({-}\Theta(n e^{-\beta R}\ss^{\frac{\beta}{\alpha}})\big)$.
    
    \smallskip    
    \item\label{thm:mixed-itm2} 
    For $\alpha=2\beta$, if $\ss=O(e^{\alpha R}/(\alpha R))$, then $\displaystyle\PP (T_{det} \ge \ss)=\exp\big({-}\Theta(ne^{-\beta R}\sqrt{\ss\log\ss})\big)$.

    \smallskip
    \item\label{thm:mixed-itm3}
    For $\alpha>2\beta$, if $\ss= O(e^{2\beta R})$, then
    $\displaystyle
    \PP (T_{det} \ge \ss)=\exp\big({-}\Theta(ne^{-\beta R}\sqrt{\ss})\big).$
\end{enumerate}
\item\label{thm:mixed-itm4}
For $\beta>\tfrac{1}{2}$, if $\ss= \Omega(1)\cap O(e^{\alpha R})$, then
    $\displaystyle
    \PP (T_{det} \ge \ss)=\exp\big({-}\Theta(\ss^{\frac{1}{2\alpha}})\big)$.
\end{enumerate}
\end{theorem}

\begin{remark}
Since we are working with the Poissonized model, the probability of not having any particle to begin with is of order $e^{-\Theta(n)}$, and on this event the detection time is infinite. The reader may check that replacing the upper bound of $\ss$ in each of the previous theorem cases gives a probability of this order, which explains the asymptotic upper bounds on $\ss$. 
\end{remark}

%
Observe that by Theorem~\ref{thm:intro-mixed}, for $\ss=\Theta((e^{\beta R}/n)^{2})$  in the case $\beta\le \frac{1}{2}$, and for $\ss=\Theta(1)$ in the case $\beta>\frac{1}{2}$  we recover a tail exponent of order $\Theta(1)$, showing that the expected detection time occurs at values of $\ss$ of said order. It follows that at $\beta=\frac{1}{2}$ there is a phase transition in the qualitative behavior of the  model; for $\beta<\frac{1}{2}$, since $(e^{\beta R}/n)^2=o(1)$, the detection becomes asymptotically ``immediate" whereas for $\beta>\frac{1}{2}$ the target can remain undetected for an amount of time of order $\Theta(1)$. To explain this change in behavior notice that even though for any $\beta>0$ the value of $\sigma^{2}_{\theta}(r):=\sinh^{-2}(\beta r)$ is minuscule near the boundary of $B_O(R)$ (where particles spend most of the time due to the radial drift), decreasing~$\beta$ does increase $\sigma_{\theta}(\cdot)$ dramatically, allowing for a number of particles tending to infinity, to detect the target immediately. Since for small values of $\ss$ all but a few particles remain ``radially still" near the boundary, we deduce that there must be some purely angular movement responsible for the detection of the target, to which we can associate the tail exponent $ne^{-\beta R}\sqrt{\ss}$ seen in~\eqref{ssmall} of Theorem~\ref{thm:intro-mixed}. The same tail exponent appears in Case~\eqref{thm:mixed-itm3} of Theorem~\ref{thm:intro-mixed} for large $\ss$ when $\beta$ is again sufficiently small (smaller than $\frac{\alpha}{2}$), and again the purely angular movement is responsible for the detection of the target. 
To make the understanding of the observed distinct behaviors even more explicit we study two simplified models where particles are restricted to evolve solely through their angular or radial movement, respectively.
\begin{theorem}[angular movement]\label{thm:angularMain}
Let $\alpha\in (\frac12,1]$, $\beta>0$, $\ss:=\ss(n)$, and assume that particles move according to the generator
  \begin{equation*}\label{anggenerator}
	\Delta_{ang} := \frac{1}{2\sinh^2(\beta r)}\frac{\partial^2}{\partial\theta^2}.
	\end{equation*}
If $\sqrt{\ss}\leq\frac{\pi}{2}(1-o(1))e^{\beta R}$, then 
\[
\P(T_{det}\geq \ss) = 
\begin{cases}
\exp\Big({-}\Theta\Big(\mkor{\nu}{1}+n\Big(\frac{\sqrt{\ss}}{e^{\beta R}}\Big)^{1\wedge \frac{\alpha}{\beta}}\Big)\Big), & \text{if $\alpha\neq\beta$,} \\[8pt]
\exp\Big({-}\Theta\Big(\mkor{\nu}{1}+n\Big(\frac{\sqrt{\ss}}{e^{\beta R}}\Big)^{1\wedge \frac{\alpha}{\beta}}\log\big(\frac{e^{\beta R}}{1+\sqrt{\ss}}\big)\Big)\Big), & \text{if $\alpha=\beta$.}
\end{cases}
\]
\end{theorem}

Observe that the tail exponent appearing in the case $\alpha>\beta$ of the previous theorem is the same as the one appearing in Case~\eqref{thm:mixed-itm3} (where $\beta\leq \frac{1}{2}<\alpha$) and Case~\eqref{thm:mixed-itm-ssmall} (where $\beta<\frac{\alpha}{2}<\alpha$) of Theorem~\ref{thm:intro-mixed}: this is no coincidence. Our proof shows that a purely angular movement is responsible for detection in those cases. Note also that the other exponents contemplated in Theorem~\ref{thm:intro-mixed} are not present in Theorem~\ref{thm:angularMain}. 

\smallskip
We consider next the result of our analysis of the second simplified model where particles move only radially:
\begin{theorem}[radial movement]{\label{thm:radialmain}}
Let $\alpha\in(\tfrac{1}{2},1]$, $\beta>0$, $\ss:=\ss(n)$,
and assume that particles move according to the generator 
\begin{equation*}\label{radialgenerator}
	\Delta_{rad} := \frac{1}{2}\frac{\partial^2}{\partial r^2}+\frac{\alpha}{2}\frac{1}{\tanh(\alpha r)}\frac{\partial}{\partial r}.
\end{equation*}
For every $0<c<\frac{\pi}{2}$, there is a $C>0$ such that if  $\ss\ge C$ and $\ss^{\frac{1}{2\alpha}}\le (\frac{\pi}{2}-c)e^{\frac{R}{2}}$, then
\[
\P(T_{det}\geq \ss) = \exp\big({-}\Theta(\ss^{\frac{1}{2\alpha}})\big).
\]
\end{theorem}

Once more observe that the tail exponent in the $\beta>\frac12$ case of Theorem~\ref{thm:intro-mixed} is the same as the one in Theorem~\ref{thm:radialmain}, and once more this is not a coincidence: our proof shows that when $\beta>\frac{1}{2}$ the detection of the target is the result of the radial motion of particles alone.
In contrast, Cases~\eqref{thm:mixed-itm1} and~\eqref{thm:mixed-itm2} in Theorem~\ref{thm:intro-mixed}, when $\frac{1}{2}<\alpha\leq2\beta\leq1$ 
yield new tail exponents not observed neither in  Theorem~\ref{thm:angularMain} nor in Theorem~\ref{thm:radialmain}, and which is larger than those given in said results. Since a larger tail exponent means a larger probability of an early detection, in this case, neither the angular nor the radial component of the movement dominates the other, but they rather work together to detect the target quicker: in Case~\eqref{thm:mixed-itm1}, the proof reveals that the radial movement is responsible for pulling particles sufficiently far away from the boundary to a radial value, where $\sigma_{\theta}(r):=\sinh^{-2}(\beta r)$ becomes sufficiently large (although still small) so that with a constant probability at least one particle has a chance of detecting at this radial value through its angular motion. Finally, for Case~\eqref{thm:mixed-itm2}  in Theorem~\ref{thm:intro-mixed} we see a tail exponent of the form $ne^{-\beta R}\sqrt{\ss}$ (as expected when taking $\alpha\nearrow 2\beta$ in Case~\eqref{thm:mixed-itm1} or $\alpha\searrow 2\beta$ in Case~\eqref{thm:mixed-itm3}) but accompanied by a logarithmic correction.  This correction becomes clear when inspecting the proof: it is a result of the balance between the effect of the radial drift and the contribution of the angular movement of particles at distinct radii whose effect is that the contributions to the total angular movement coming from the time spent by a particle within a narrow band centered at a given radius are all about the same, adding up and giving the additional logarithmic factor.
 
\medskip 
Our main theorem, that is Theorem~\ref{thm:intro-mixed}, provides insight on \textit{how unlikely} it is for the target~$Q$ to remain undetected until time $\ss$, and as discussed before, the proof strategy as well as our remarks provide additional information about the \textit{dominant detection mechanisms} (that is, either by angular movement only, by radial movement only, or by a combination of the two). Moreover,  through our techniques we can obtain more precise information showing in each case \textit{where} particles must be initially located in order to detect~$Q$ before time $\ss$ with probability bounded away from zero. Specifically, given a parameter $\kappa>0$ (independent of $n$) 
we can construct a parametric family of sets $\dt(\kappa)$ depending on $\ss:=\ss(n)$ and the parameters of the model, such that for every $x_0\in\dt(\kappa)$ and sufficiently large $n$, the probability $\PP_{x_0}(T_{det}\ge\ss)$ is at least a constant that depends on the parameter $\kappa$. Even further, we will show that $\mu(\dt(\kappa))$ is of the same order as the tail exponents, which implies that asymptotically the best chance for $Q$ to remain undetected by time $\ss$ is to find these sets initially empty of points.
To be precise, we show the following meta-result:
\begin{theorem}{\label{thm:intro-Dt}}
Let $\alpha\in (\frac12,1]$, $\beta > 0$ and $\ss:=\ss(n)$. For every case considered in Theorems~\ref{thm:intro-mixed},~\ref{thm:angularMain} and~\ref{thm:radialmain} with the corresponding hypotheses, and under the additional assumption of $\ss=\omega(1)$ in the case of Theorem~\ref{thm:intro-mixed}, there is an explicit parametric family of sets $\dt:=\dt(\kappa)$ which is increasing in $\kappa$, and which satisfies:
\begin{enumerate}[(i)]
    \item\label{lowerDt} (Uniform lower bound) For every $\kappa$ sufficiently large and $n$ sufficiently large, 
    \[
    \inf_{x_0\in\dt\!(\kappa)}\PP_{x_0}(T_{det}\ge\ss)
    \]
    is bounded from below by a positive expression that depends only on $\kappa$ and the parameters of the model.
    
    \item\label{upperDt} (Uniform upper bound) For every $\kappa$ sufficiently large and $n$ 
    sufficiently large
    \[
    \sup_{x_0\in\ndt\!(\kappa)}\PP_{x_0}(T_{det}\ge\ss)
    \]
    is bounded from above by a positive expression that depends only on $\kappa$ and the parameters of the model, and that tends to $0$ as $\kappa \to \infty$.
    
    \item\label{integralDt} (Integral bound) Fix $\kappa=C$ for some sufficiently large constant $C>0$. For $n$ sufficiently large, we have 
    \[
    \int_{x_0\in\ndt\!(C)}\PP_{x_0}(T_{det}\le\ss)d\mu(x_0) = O(\mu(\dt(C))).
    \]
\end{enumerate}
\end{theorem}
The added value of this last theorem is that it reveals that the set~$\dt(\kappa)$  
can roughly be interpreted as the regions where particles \emph{typically detecting} the target are initially located. Moreover, each of the 
Theorems~\ref{thm:intro-mixed}, \ref{thm:angularMain}, and~\ref{thm:radialmain} follow from the instances of Theorem~\ref{thm:intro-Dt} regarding angular, radial and unrestricted movement, respectively, thus
revealing the unified approach that we take throughout the paper as explained in detail in Subsection~\ref{sec:strategy}. 

Establishing the uniform lower bounds for the different cases of Theorem~\ref{thm:intro-Dt} requires, especially in Case~\eqref{thm:mixed-itm2} of Theorem~\ref{thm:intro-mixed}, a delicate coupling with a discretized auxiliary process. For the integral upper bounds a careful analysis has to be performed: we thereby need to calculate hitting times of Brownian motions in the hyperbolic plane with a reflecting barrier where we make use of results on tails of independent sums of Pareto random variables (that differ greatly according to the exponent).

\smallskip
Finally, let us point out that from a technical point of view, the additional difficulties we encounter compared to Euclidean space are substantial: in~\cite{Peres2010} the authors address the Euclidean analogue of the problem studied in this paper by relating the detection probability to the expected volume of the $d$-dimensional Wiener sausage (that is, the neighborhood of the trace of Brownian motion up to a certain time moment -- see also the related work section below). To the best of our knowledge, however, the volume of the Wiener sausage of Brownian motion with a reflecting boundary in hyperbolic space is not known. Even further, our proof techniques give significantly more information: with our proof we are able to characterize the subregions of the hyperbolic space such that particles initially located therein are likely to detect the target up to a certain time moment, whereas particles located elsewhere are not - with the information of the pure volume of the Wiener sausage we would not be able to do so. 
The radial drift towards the boundary of $B_O(R)$ entails many additional technical difficulties that do not arise in Euclidean space (see also the related work section below for work being done there). Moreover, our setup contains a reflecting barrier at the boundary, thereby adding further difficulty to the analysis.

\mbox{\ }

\subsection{Structure and strategy of proof}\label{sec:strategy}
Recall that in our mobile hyperbolic graph model
we consider an initial configuration $\mathcal{P}$ of the Poissonized model in $B_O(R)$ and then associate to each $x_0:=(r_0,\theta_0)\in\mathcal{P}$ a particle that evolves independently of other particles according to the generator $\Delta_h$ with a reflecting barrier at $\partial B_0(R)$.
Denote by $\mathcal{P}_{\ss}$ the Poisson point process obtained from $\mathcal{P}$ by retaining only particles having detected the target by time $\ss$.
Since points move independently of each other, each particle initially placed at $x_0\in B_O(R)$ belongs to $\mathcal{P}_{\ss}$ independently with probability $\P_{x_0}(T_{det}\leq \ss)$ and hence $\mathcal{P}_{\ss}$ is a thinned Poisson point process on $B_O(R)$ with intensity measure
\[d\mu_{\ss}(x)\;=\;\P_{x_0}(T_{det}\leq \ss)d\mu(x).\]
Noticing that the event $\{T_{det}\geq \ss\}$ is equivalent to $\{\mathcal{P}_{\ss}=\emptyset\}$, we have
\begin{equation}\label{eqn:main}
    \P(T_{det}\geq \ss)=\exp\Big(-\int_{B_O(R)}\P_{x_0}(T_{det}\leq \ss)d\mu(x_0)\Big)
\end{equation}
and hence in order to obtain lower and upper bounds on $\P(T_{det}\geq \ss)$ we will compute the dependence of this integral as a function of $\ss$.
In fact, we can say a bit more on the way we go about determining the value of the said integral. Specifically, we partition the range of integration $B_O(R)$ into $\dt(\kappa)$ and 
$\ndt(\kappa)$ and observe that the three parts of Theorem~\ref{thm:intro-Dt} together immediately imply that
\[
\int_{B_O(R)}\P_{x_0}(T_{det}\leq \ss)d\mu(x_0) = \Theta(\mu(\dt(\kappa)))
\]
thus reducing, via~\eqref{eqn:main}, the computation of tail bounds for detection times to fixing the parameter $\kappa$ equal to a constant and determining the expected number of particles initially in $\dt(\kappa)$, that is, computing $\mu(\dt(\kappa))$. 
So, the structure of the proof of our main results is the following: we determine a suitable candidate set $\dt(\kappa)$, compute $\mu(\dt(\kappa))$, and establish an adequate version of Theorem~\ref{thm:intro-Dt} for the specific type of movement of particles considered (angular, radial or a combination of both).

\subsection{Related work}\label{sec:relatedWork}
After the introduction of random hyperbolic graphs by Krioukov et al.~\cite{KPKVB10}, the model was then first analyzed in a mathematically rigorous way by Gugelmann et al.~\cite{GPP12}: they proved that for the case $\alpha > \frac12$ the distribution of the degrees follows a power law, they showed that the clustering coefficient is bounded away from $0$, and they also obtained the average degree and maximum degree in such networks. The restriction $\alpha>\frac12$ guarantees that the resulting graph has bounded average degree (depending on $\alpha$ and $\nu$ only): if $\alpha<\frac12$, then the degree sequence is so heavy tailed that this is impossible (the graph is with high probability connected in this case, as shown in~\cite{BFM13b}). 
The power-law degree distribution of random hyperbolic graphs is equal to $2\alpha+1$ (see~\cite[Theorem 2.2]{GPP12}) which, for $\alpha\in (\frac12, 1]$, belongs to the interval $(2,3]$, and this is the interval where the best fit power-law exponent of social networks typically falls into (see~\cite[p.~69]{barabasiLinked}).
If $\alpha>1$, then as the number of vertices grows, the largest component of a random hyperbolic graph has sublinear size (more precisely, its order is $n^{1/(2\alpha)+o(1)}$, see~\cite[Theorem~1.4]{BFM13} and~\cite{Diel}). On the other hand, it is known that for $\frac12 < \alpha < 1$, with high probability 
a hyperbolic random graph of expected order~$n$ has a connected component whose order is also linear in $n$~\cite[Theorem~1.4]{BFM13} and the second largest component has size $\Theta(\log^{\frac{1}{1-\alpha}} n)$~\cite{km19},  which justifies referring to the linear size component as \emph{the giant component}. 
More precise results including a law of large numbers for the largest component in these networks were established in~\cite{FMLaw}. Further results on the static version of this model include results on the diameter~\cite{km15, fk15, MS19}, on the spectral gap~\cite{KM18}, on typical distances~\cite{ABF}, on the clustering coefficient~\cite{CF16, Clustering}, on bootstrap percolation~\cite{KL} and on the contact process~\cite{Contact}.

No dynamic model of random hyperbolic graphs has been proposed so far. In the \emph{Boolean model} (that is similar to random geometric graphs but now the underlying metric space is~$\RR^d$ and, almost surely, there is an infinite number of vertices)
the same question  of the tail probability of detection time, as well as the questions of coverage time (the time to detect a certain region of $\mathbb{R}^d$), percolation time (the time a certain particle needs to join particle of the infinite component) and broadcast time (the time it takes to broadcast a message to all particles) were studied by Peres et al.~\cite{Peres2010}: the authors therein show that for a Poisson process of intensity $\lambda$, and a target performing independent Brownian motion, as $t \to \infty$, 
$$
\P{(T_{det} \ge t)} = \exp\Big(- 2\pi \lambda\frac{t}{\log t}(1+o(1))\Big),
$$
where a target is detected if a particle is at Euclidean distance at most $r$, for some arbitrary but fixed $r$ (in the three other mentioned contexts the interpretations are similar: we say that an interval is covered if each point of the interval has been at distance at most $r$ from some particle at some time instant before $t$; we say that a particle joins a component if it is at distance at most $r$ from another particle of the component, and we say that a message can be sent between two particles if they are at distance $r$). Note that the probability holds only as $t \to \infty$, and in this case the main order term of the tail probability does not depend on $r$, as shown in~\cite{Peres2010}. Stauffer~\cite{Stauffer} generalized detection of a mobile target moving according to any continuous function and showed that for a Poisson process of sufficiently high intensity $\lambda$ over $\mathbb{R}^d$ (with the same detection setup) the target will eventually be detected almost surely, whereas for small enough $\lambda$, with positive probability the target can avoid detection forever. The somewhat inverse question of isolation of the target, that is, the time it takes for a (possibly) moving target (again for a Poisson process of intensity $\lambda$ in $\mathbb{R}^d$) until no other vertex is at distance $r$ anymore and the target becomes isolated, was then studied by Peres et al.~\cite{Isolation}: it was shown therein that the best strategy for the target to stay isolated as long as possible in fact is to stay put (again with the same setup for being isolated).

Also for the Boolean model, the question of detection time was already addressed by Liu et al.~\cite{Liu} when each particle moves continuously in a fixed randomly chosen direction: they showed that the time it takes for the set of particles to detect a target is exponentially distributed with expectation depending on the intensity $\lambda$ (where detection again means entering the sensing area of the target). For the case of a stationary target as discussed here, as observed in Kesidis et al.~\cite{Kesidis} and in Konstantopoulos~\cite{Konstant} the detection time can be deduced
from classical results on continuum percolation: namely, in this case it follows from Stoyan et al.~\cite{Stoyan} that
$$
\P{(T_{det} \ge t)} = e^{-\lambda \EE(vol(W_r(t)))},
$$
where $vol(W_r(t))$ is the volume of the Wiener sausage of radius $r$ up to time $t$ (equivalent to being able to detect at distance at most $r$), and which in the case of Euclidean space is known quite precisely~\cite{Spitzer, Bere}. 

A study of dynamic random geometric graphs prior to the paper of Peres et al.~was undertaken by D\'{i}az et al.~\cite{DMP09}: the authors therein consider a random geometric graph whose radius is close to the threshold of connectivity and analyze in this setup the lengths of time intervals of connectivity and disconnectivity. Even earlier, the model of Brownian motion in infinite random geometric graphs (in the "dynamic Boolean model") was studied by van den Berg et al.~\cite{vdBerg} who showed that for a Poisson intensity above the critical intensity for the existence of an infinite component, almost surely an infinite component exists at all times (here the radii of particles are not fixed to $r$, but rather i.i.d.~random variables following a certain distribution).
More generally, the question of detecting an immobile target (or escaping from a set of immobile traps) when performing random walks on a square lattice is well studied (here, in contrast to the previous papers, detecting means to be exactly at the same position in the lattice); escaping from a set of mobile traps was recently analyzed as well, for both questions we refer to Athreya et al.~\cite{Athreya} and the references therein.

\subsection{Organization}
Section~\ref{preliminaries} gathers facts and observations that will be used throughout the paper. Section~\ref{sec:angular} then analyzes the simplified model where the particles' movement is restricted to angular movement only, at the same time illustrating in the simplest possible setting our general proof strategy. Section~\ref{sec:radial} then considers the model where the particles' movement is restricted to radial movement only. In Section~\ref{sec:mix} we finally address the mixed case, that is, the combined case of both radial and angular movement of particles. Specifically, in Section~\ref{sec:Lower}, we look at quick detection mechanisms and, in Section~\ref{sec:Upper}, we show that the detection mechanisms found  are asymptotically optimal in all cases, that is, no other strategy can detect the target asymptotically faster.
In Section~\ref{sec:conclusion} we briefly comment on future work.

\subsection{Global conventions}
Throughout all of this paper $\alpha$ is a fixed constant such that $\frac12<\alpha\le 1$
(the reason to consider only this range is that only therein hyperbolic random graphs are deemed to be reasonable models of real-world networks -- see the discussion in Section~\ref{sec:relatedWork}). 
Furthermore, in this article both $\beta$ and $\nu$ are strictly positive constants. In order to avoid tedious repetitions and lengthy statements, we will not reiterate these facts. 
Also, wherever appropriate, we will hide $\alpha$, $\beta$, and $\nu$ within asymptotic notation.
Moreover, throughout the paper, we let~$\ss:=\ss(n)$ be a non-negative function depending on $n$.

In order to avoid using ceiling and floors, we will always assume $R:=2\ln(n/\nu)$ is an integer. Since all our claims hold asymptotically, doing so has no impact on the stated results. 
All asymptotic expressions are with respect to $n$, and wherever it makes sense, inequalities should be understood as valid asymptotically.

\section{Preliminaries}\label{preliminaries}
In this section we collect basic facts and a few useful observations together with a bit of additional notation.
First, for a region $\Omega$ of $B_O(R)$, that is, $\Omega\subseteq B_O(R)$, let $\mu(\Omega)$ denote the expected number of points $\mathcal{P}$ in $\Omega$. The first basic fact we state here gives the expected number of points $\mathcal{P}$ at distance at most $r$ from the origin, 
as well as the expected number of points at distance at most $R$ from a given point $Q \in B_{O}(R)$:
\begin{lemma}[{\cite[Lemma~3.2]{GPP12}}]\label{lem:muBall}
If $0\leq r\leq R$, then
  $\mu(B_{O}(r)) = n e^{-\alpha(R-r)}(1+o(1))$.
Moreover, if $Q \in B_O(R)$ is such that $r_Q:=r$, then 
for $C_{\alpha}:=2\alpha/(\pi (\alpha-\frac12))$,
\[
\mu(B_Q(R) \cap B_O(R)) = n C_{\alpha} e^{-\frac{r}{2}}(1 + O(e^{-(\alpha-\frac12)r}+e^{-r})).
\]
\end{lemma}

We also use the following Chernoff bounds for Poisson random variables:
\begin{theorem}[Theorem A.1.15 of~\cite{AlonSpencer}]
Let $P$ have Poisson distribution with mean $\mu$. Then, for every $\varepsilon > 0$,
$$
\P(P \le \mu(1-\varepsilon)) \le e^{-\varepsilon^2 \mu/2}
$$
and
$$
\P(P \ge \mu(1+\varepsilon)) \le \left(e^{\varepsilon} (1+\varepsilon)^{-(1+\varepsilon)} \right)^{\mu}.
$$
\end{theorem}

\medskip
Throughout this paper, we denote by $\theta_R(r,r')$ the maximal angle between two points at (hyperbolic) distance at most $R$ and radial coordinates  $0\leq r,r'\leq R$. By~\eqref{eqn:coshLaw}, for $r+r'\geq R$, we have
\begin{equation}\label{eqn:angle}
\cosh R = \cosh r\cosh r'-\sinh r\sinh r'\cos \theta_R(r,r').
\end{equation}
Henceforth, consider the mapping $\phi:[0,R]\to [\theta_R(R,R),\pi/2]$ such that $\phi(r):=\theta_R(R,r)$.
For the sake of future reference, we next collect some simple as well as some known facts concerning~$\phi(\cdot)$.
\begin{lemma}\label{lem:phi}
The following hold:
\begin{enumerate}[(i)]
\item\label{itm:phi1} $\cos(\phi(r))=\coth R\cdot\tanh(\frac{r}{2})$.
\item\label{itm:phi2} $\phi(\cdot)$ is differentiable (hence, also continuous) and decreasing.
\item\label{itm:phi3} $\phi(0)=\frac{\pi}{2}$.
\item\label{itm:phiInv} $\phi(\cdot)$ is invertible and its inverse is differentiable (hence, continuous) and decreasing.
\item\label{itm:phi4} $\phi(r)=2e^{-\frac{r}{2}}(1\pm O(e^{-r}))$. 
\end{enumerate}
\end{lemma}
\begin{proof}
The first part follows from~\eqref{eqn:coshLaw} taking $d_H:=R$, $r':=R$, and by the hyperbolic tangent half angle formula. 
The second part follows because the derivative with respect to $r$ of $\arccos(\coth R\cdot\tanh(\frac{r}{2}))$
exists and is negative except when $r=R$ (for details see~\cite[Remark 2.1]{km19}).
The third part follows directly from the first part, and the fourth part follows immediately from
the preceding two parts.
The last part is a particular case of a more general result~\cite[Lemma 3.1]{GPP12}.
\end{proof}

\medskip
To conclude this section, we introduce some notation that we will use throughout the article. For a subset $\Omega$ of $B_O(R)$ we let~$\overline{\Omega}$ denote its complement with respect to $B_O(R)$, i.e., $\overline{\Omega}:=B_{O}(R)\setminus\Omega$.
Thus, $\overline{B}_O(r)$ denotes $B_{O}(R)\setminus B_O(r)$. 

\section{Angular movement}\label{sec:angular}
In this section we consider angular movement only.
As it will become evident, restricting the motion to angular
movement alone simplifies the analysis considerably, 
compared to the general case, and in fact also compared
to restricting movement radially.
The simpler setting considered in this section is ideal for
illustrating our general strategy for
deriving tail bounds on detection times.
It is also handy for introducing some notation and approximations we shall 
use throughout the rest of the paper.
In later sections we will follow the same general strategy
argument but encounter technically more challenging obstacles.
In contrast, the calculations involved in the study of
the angular movement only case can be mostly reduced to
ones involving known facts about one dimensional Brownian motion.

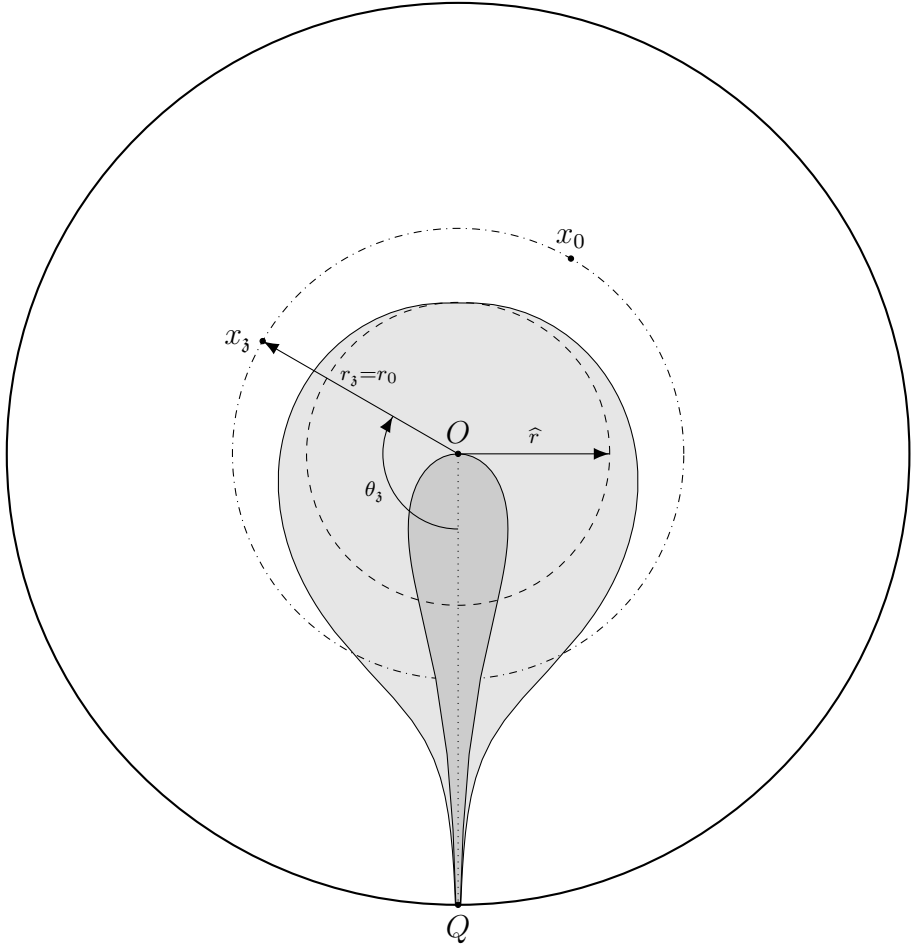
\begin{figure}
  \begin{center}
    \begin{tikzpicture}[x=1cm,y=1cm,scale=0.5,
     decoration={
       markings,
       mark=at position 1 with {\arrow[scale=1.5,black]{latex}};
     }
   ]
      \def\c{12}
      \def\barr{4}
      \def\rrho{6}
      \def\radp{6}
      \def\phip{150}
      \def\phi0{60}
      \node[inner sep=0] (O) at (\c,\c) {};
      \node[inner sep=0] (P) at (2,4) {};
      \node[inner sep=0] (Q) at (\c,0) {};
      

      \draw[fill=black] (O) circle (0.035);

      \node[above] at (O) {$O$};
      \draw[fill=black] (O) ++(\phi0-360:\radp) circle (0.07) node[above] {$x_0$};

      \draw[fill=gray!20] (11.929cm,0.000cm) -- (11.909cm,0.480cm) -- (11.884cm,0.961cm) -- (11.849cm,1.441cm) -- (11.803cm,1.922cm) -- (11.740cm,2.404cm) -- (11.653cm,2.887cm) -- (11.531cm,3.373cm) -- (11.359cm,3.865cm) -- (11.115cm,4.371cm) -- (10.769cm,4.906cm) -- (10.281cm,5.504cm) -- (9.604cm,6.239cm) -- (8.717cm,7.267cm) -- (8.385cm,7.723cm) -- (8.218cm,7.979cm) -- (8.051cm,8.258cm) -- (7.888cm,8.562cm) -- (7.732cm,8.892cm) -- (7.586cm,9.251cm) -- (7.456cm,9.640cm) -- (7.347cm,10.062cm) -- (7.267cm,10.518cm) -- (7.222cm,11.006cm) -- (7.223cm,11.527cm) -- (7.244cm,11.798cm) -- (7.281cm,12.076cm) -- (7.334cm,12.360cm) -- (7.406cm,12.648cm) -- (7.497cm,12.941cm) -- (7.610cm,13.235cm) -- (7.747cm,13.530cm) -- (7.908cm,13.823cm) -- (8.095cm,14.113cm) -- (8.309cm,14.396cm) -- (8.552cm,14.669cm) -- (8.825cm,14.929cm) -- (9.127cm,15.172cm) -- (9.459cm,15.394cm) -- (9.820cm,15.590cm) -- (10.209cm,15.755cm) -- (10.625cm,15.884cm) -- (11.065cm,15.971cm) -- (11.525cm,16.012cm) -- (12.000cm,16.015cm) -- (12.475cm,16.012cm) -- (12.935cm,15.971cm) -- (13.375cm,15.884cm) -- (13.791cm,15.755cm) -- (14.180cm,15.590cm) -- (14.541cm,15.394cm) -- (14.873cm,15.172cm) -- (15.175cm,14.929cm) -- (15.448cm,14.669cm) -- (15.691cm,14.396cm) -- (15.905cm,14.113cm) -- (16.092cm,13.823cm) -- (16.253cm,13.530cm) -- (16.390cm,13.235cm) -- (16.503cm,12.941cm) -- (16.594cm,12.648cm) -- (16.666cm,12.360cm) -- (16.719cm,12.076cm) -- (16.756cm,11.798cm) -- (16.777cm,11.527cm) -- (16.778cm,11.006cm) -- (16.733cm,10.518cm) -- (16.653cm,10.062cm) -- (16.544cm,9.640cm) -- (16.414cm,9.251cm) -- (16.268cm,8.892cm) -- (16.112cm,8.562cm) -- (15.949cm,8.258cm) -- (15.782cm,7.979cm) -- (15.615cm,7.723cm) -- (15.283cm,7.267cm) -- (14.396cm,6.239cm) -- (13.719cm,5.504cm) -- (13.231cm,4.906cm) -- (12.885cm,4.371cm) -- (12.641cm,3.865cm) -- (12.469cm,3.373cm) -- (12.347cm,2.887cm) -- (12.260cm,2.404cm) -- (12.197cm,1.922cm) -- (12.151cm,1.441cm) -- (12.116cm,0.961cm) -- (12.091cm,0.480cm) -- (12.071cm,0.000cm) -- (12.060cm,0.000cm) -- (12.045cm,0.000cm) -- (12.030cm,0.000cm) -- (12.015cm,0.000cm) -- (12.000cm,0.000cm) -- (11.985cm,0.000cm) -- (11.970cm,0.000cm) -- (11.955cm,0.000cm) -- (11.940cm,0.000cm) -- cycle;

      \draw[fill=black] (O) ++(\phip:\radp) circle (0.07);
      \draw[postaction={decorate}] (O) -- node[midway,above] {$\ \ _{r_{\ss}=r_0}$} ++(\phip:\radp) node[left] {$x_{\ss}$};
      \draw[dash dot] (O) circle (\rrho);
            
      \draw[fill=gray!40] (11.941cm,0.000cm) -- (11.865cm,2.001cm) -- (11.707cm,4.005cm) -- (11.404cm,6.030cm) -- (10.937cm,8.144cm) -- (10.842cm,8.591cm) -- (10.798cm,8.820cm) -- (10.758cm,9.051cm) -- (10.725cm,9.285cm) -- (10.698cm,9.521cm) -- (10.681cm,9.760cm) -- (10.675cm,9.999cm) -- (10.681cm,10.239cm) -- (10.704cm,10.477cm) -- (10.744cm,10.711cm) -- (10.804cm,10.938cm) -- (10.885cm,11.154cm) -- (10.988cm,11.356cm) -- (11.113cm,11.538cm) -- (11.260cm,11.696cm) -- (11.426cm,11.825cm) -- (11.608cm,11.921cm) -- (11.801cm,11.980cm) -- (12.000cm,12.000cm) --(12.199cm,11.980cm) -- (12.392cm,11.921cm) -- (12.574cm,11.825cm) -- (12.740cm,11.696cm) -- (12.887cm,11.538cm) -- (13.012cm,11.356cm) -- (13.115cm,11.154cm) -- (13.196cm,10.938cm) -- (13.256cm,10.711cm) -- (13.296cm,10.477cm) -- (13.319cm,10.239cm) -- (13.325cm,9.999cm) -- (13.319cm,9.760cm) -- (13.302cm,9.521cm) -- (13.275cm,9.285cm) -- (13.242cm,9.051cm) -- (13.202cm,8.820cm) -- (13.158cm,8.591cm) -- (13.063cm,8.144cm) -- (12.596cm,6.030cm) -- (12.293cm,4.005cm) -- (12.135cm,2.001cm) -- (12.059cm,0.000cm) -- (12.059cm,0.000cm) -- (12.045cm,0.000cm) -- (12.030cm,0.000cm) -- (12.015cm,0.000cm) -- (12.000cm,0.000cm) -- (11.985cm,0.000cm) -- (11.970cm,0.000cm) -- (11.955cm,0.000cm) -- (11.941cm,0.000cm) -- cycle;
      \draw[dotted] (O) -- (Q);
      \draw[dashed] (O) circle (4.03);
      \draw[postaction={decorate}] (O) -- node[midway,above] {$_{\widehat{r}}$} ++(0:4.03);
      \draw[postaction={decorate}] (\c,\c-2) arc (-90:\phip-360:2);
      \node at (\c-2.2,\c-1) {$_{\theta_{\ss}}$};
      \draw[thick] (O) circle (\c);
      \draw[fill=black] (Q) circle (0.07);
      \draw[fill=black] (O) circle (0.07);
      \node[below] at (Q) {$Q$};      
      \node[above] at (O) {$O$}; 
    \end{tikzpicture}
  \end{center}
  \caption{The position $x_{\ss}=(r_{\ss},\theta_{\ss})$ of particle $P$ at time
    $\ss$ is shown. The lightly shaded area corresponds to
    $\dt(\kappa)\setminus B_{Q}(R)$ and the strongly shaded
    region represents $B_Q(R)$. The dashed line corresponds to the arc
    segment where a particle $P$ at initial radial distance $r_0$ is,
    conditional under not having detected the target~$Q$ up to time
    $\ss$. The depicted dash-dot arc segment has length $C_{r_0}$.
    Each of the two dash-dot arc lengths inside the lightly shaded 
    region have (hyperbolic) length $\kappa\sqrt{\ss}$ corresponding to an angle $\theta_0$ such that
    $\phi(r_0)\leq |\theta_0| \le \phi(r_0)+\kappa \phi^{(\ss)}$.
    The dashed circle is the boundary of the largest ball centered at the origin which is contained in $\dt(\kappa)$.}\label{fig:angular}
\end{figure}

Following the proof strategy described in Section~\ref{sec:strategy}, we first define $\dt(\kappa)$. 
Recall that, roughly speaking, $\dt(\kappa)$ is the set of starting points from which a particle has a "good" chance to detect the target by time $\ss$, where "good" depends on a parameter $\kappa > 1$ (independent of $n$).
Let $\phi^{(\ss)}$ be roughly proportional to the angular distance travelled by a particle at radial coordinate $r_0$ up to time $\ss$, more precisely, let $\phi^{(\ss)}:=\sqrt{\ss}e^{-\beta r_0}$ (that is,~$\phi^{(\ss)}$ corresponds to the standard deviation of a standard Brownian motion during an interval of time of length $\ss$ normalized by a term which is proportional to the perimeter of $B_O(r_0)$ for~$r_0$ bounded away from $0$). We can then define the following region (the shaded region in Figure~\ref{fig:angular}): 
\[
\dt = \dt(\kappa) := \big\{ x_0\in B_O(R) : |\theta_0| \le \phi(r_0)+\kappa \phi^{(\ss)}\big\}.
\]

In words, $\dt(\kappa)$ is the collection of points $x_0=(r_0,\theta_0)\in B_O(R)$ which are either contained in $B_Q(R)$ or form an angle at the origin of at most $\kappa \phi^{(\ss)}$ with a point that belongs to the boundary of $B_Q(R)$ and has radial coordinate exactly $r_0$.
Since $x_0\in B_{Q}(R)$ if and only if~$|\theta_0|\leq\phi(r_0)$, it is clear that $B_Q(R)$
is contained in $\dt(\kappa)$. The value $\kappa$ is a parameter that tells us at which scaling up to time $\ss$ we consider our region $\dt(\kappa)$. 

The main goal of this section is to prove the result stated next
which is an instance of Theorem~\ref{thm:intro-Dt} for the case of angular movement only.
Thus, Theorem~\ref{thm:angularMain} will immediately follow  (by the proof strategy discussed in Section~\ref{sec:strategy}) once we show that~$\mu(\dt(\kappa))$ is of the right order:
\begin{theorem}\label{thm:mainangular}
Denote by $\Phi$ the standard normal distribution.
If $\kappa>0$ and 
  $\kappa\sqrt{\ss}\leq\frac{\pi}{2}(1-o(1))e^{\beta R}$, then
\[
\inf_{x_0\in\dt(\kappa)}\P_{x_0}(T_{det}\leq \ss)=\Omega(\Phi(-\kappa))
\qquad\text{and}\qquad
\sup_{x_0\in\ndt(\kappa)}\P_{x_0}(T_{det}\leq \ss)  = O(\Phi(-\kappa)).
\]
Furthermore,
\[
\int_{\ndt\!(\kappa)} \PP_{x_0}(T_{det}\le\ss)d\mu(x_0) = O(\mu(\dt(\kappa))).
\]
\end{theorem}

Before proving the previous result, we make a few observations and introduce some definitions. 
First, note that the perimeter of $B_{O}(r)$ that is outside $B_Q(R)$ has length (see Figure~\ref{fig:angular}) 
\[
C_{r}:=2(\pi-\phi(r))\sinh(\beta r).
\]

Since $\phi(\cdot)$ is decreasing and continuous 
(by Part~\eqref{itm:phi2} of Lemma~\ref{lem:phi}), 
we obtain the following:
\begin{fact}\label{fct:angular-monot}
The mapping $r\mapsto \frac12 C_r=(\pi-\phi(r))\sinh(\beta r)$ is increasing
and continuous for arguments in $[0,R]$ and takes values in $[0,\frac12 C_R]$.
\end{fact}
In particular, for  $\kappa\sqrt{\ss}\in [0,\frac12C_R]$ there is a unique value 
$\widehat{r}$ such that
\begin{equation}\label{eqn:angular-defrhat}
  \kappa\sqrt{\ss} = \tfrac12 C_{\widehat{r}}
    = (\pi-\phi(\widehat{r}))\sinh(\beta\widehat{r}).
\end{equation}
One may think of $\widehat{r}$ as chosen so that up to time $\kappa^2\ss$ a point at distance $\widehat{r}$ from the origin has a reasonable chance to detect the target due to their angular movement. Using Fact~\ref{fct:angular-monot}, we immediately obtain the following:
\begin{fact}\label{fct:angular-inclus}
For $r\in [0,R]$, the following holds:
$r\geq\widehat{r}$ if and only if $\kappa\sqrt{\ss}\leq\frac12 C_r$.
In particular, $B_O(\widehat{r})$ is the largest ball centered at the origin
contained in $\dt(\kappa)$.
\end{fact}
(See Figure~\ref{fig:angular} for an illustration of $B_O(\widehat{r})$.)
\begin{fact}\label{fct:angular-aproxhat} 
If $\widehat{r}>1$, then $\widehat{r}=\frac{1}{\beta}\log(\kappa\sqrt{\ss})+\Theta(1)$.
Moreover, if $\widehat{r}\leq 1$, then $\kappa\sqrt{\ss}=O(1)$.
\end{fact}


%
We will need the following claim whose proof is simple although the calculations involved are a bit tedious, mostly consisting in 
computing integrals and case analysis. 
\begin{lemma}\label{lem:angular-muDt}
If $\kappa>0$ and $\kappa\sqrt{\ss}\leq \frac{\pi}{2}(1-o(1))e^{\beta R}$, then
\[
\mu(\dt(\kappa))=
\begin{cases}
\Theta\Big(\mkor{\nu}{1}+n\Big(\frac{\kappa\sqrt{\ss}}{e^{\beta R}}\Big)^{1\wedge \frac{\alpha}{\beta}}\Big),
& \text{if $\alpha\neq\beta$,} \\[8pt]
\Theta\Big(\mkor{\nu}{1}+n\frac{\kappa\sqrt{\ss}}{e^{\beta R}}\Big(1+\log\big(\frac{e^{\beta R}}{\kappa\sqrt{\ss}}\big)\Big)\Big),
& \text{if $\alpha=\beta$.}
\end{cases}
\]
\end{lemma}
\begin{proof}
If $\widehat{r}> R-\Theta(1)$, by~\eqref{eqn:angular-defrhat} and Part~\eqref{itm:phi3} of Lemma~\ref{lem:phi}, we see that $\kappa\sqrt{\ss}=\Theta(e^{\beta R})$ which always gives an expression of order $\Theta(n)$ on the right-hand side of the equality in the lemma's statement. On the other hand, by Fact~\ref{fct:angular-inclus}, we know that $B_O(\widehat{r})\subseteq\dt(\kappa)\subseteq B_O(R)$, so
from Lemma~\ref{lem:muBall} we get that $\mu(\dt(\kappa))=\Theta(\mu(B_O(\widehat{r})))=\Theta(n)$, and hence 
the claim holds for said large values of $\widehat{r}$.

Assume henceforth that $\widehat{r} \leq R-\Theta(1)$ and define $A:=\dt(\kappa)\setminus (B_Q(R)\cup B_O(\widehat{r}))$.
Clearly, 
\begin{equation}\label{eqn:angular-muSum}
\mu(\dt(\kappa)) = \mu(B_Q(R)\cup B_O(\widehat{r}))+\mu(A).
\end{equation}
Now, observe that $B_Q(R)$ is completely contained in half of the disk $B_O(R)$ (see Figure~\ref{fig:angular}), so
$\mu(B_Q(R)\cup B_O(\widehat{r}))=\mu(B_Q(R))+\Theta(\mu(B_O(\widehat{r})))$, and thus by Lemma~\ref{lem:muBall} and Fact~\ref{fct:angular-aproxhat}, for~$\widehat{r}>1$, 
\begin{equation}\label{eqn:angular-muNotA}
\mu(B_Q(R)\cup B_O(\widehat{r})))=\Theta\Big(n\Big(e^{-\frac{R}{2}}+\Big(\frac{\kappa\sqrt{\ss}}{e^{\beta R}}\Big)^{\frac{\alpha}{\beta}}\Big)\Big)
=\Theta\Big(\mkor{\nu}{1}+n\Big(\frac{\kappa\sqrt{\ss}}{e^{\beta R}}\Big)^{\frac{\alpha}{\beta}}\Big).
\end{equation}
Moreover, the identity also holds when $\widehat{r}\leq 1$, since $\mu(B_Q(R))=\Theta(\mkor{\nu}{1})$ and $\mu(B_O(\widehat{r}))=O(ne^{-\alpha R})=o(1)$ (by Lemma~\ref{lem:muBall}, definition of $R$ and the fact that $\alpha>\frac12$), and $n(\kappa\sqrt{\ss}/e^{\beta R})^{\frac{\alpha}{\beta}}
=O(ne^{-\alpha R})=o(1)$ (by Fact~\ref{fct:angular-aproxhat}, definition of $R$ and since $\alpha>\frac12)$.

On the other hand, by Fact~\ref{fct:angular-inclus} and our choice of $\dt(\kappa)$, we get 
\begin{equation}\label{eqn:angular-muA}
\mu(A)= \Theta(ne^{-\alpha R}\kappa\sqrt{\ss})\int_{\widehat{r}}^R e^{-\beta r_0}\sinh(\alpha r_0)dr_0
\end{equation}
The next claim together with~\eqref{eqn:angular-muSum},  \eqref{eqn:angular-muNotA} and~\eqref{eqn:angular-muA} yield the lemma:
\[
ne^{-\alpha R}\kappa\sqrt{\ss}\int_{\widehat{r}}^R e^{-\beta r_0}\sinh(\alpha r_0)dr_0 =
\begin{cases}
O\Big(n\Big(\frac{\kappa\sqrt{\ss}}{e^{\beta R}}\Big)^{1\wedge \frac{\alpha}{\beta}}\Big),
& \text{if $\alpha\neq\beta$,} \\[8pt]
\Theta\Big(n\frac{\kappa\sqrt{\ss}}{e^{\beta R}}\log\big(\frac{e^{\beta R}}{\kappa\sqrt{\ss}}\big)\Big),
& \text{if $\alpha=\beta$.}
\end{cases}
\]
To prove the claim, not that when $\alpha=\beta$, the last integral equals $\Theta(R-\widehat{r})$.
If $\widehat{r}\leq 1$, by Fact~\ref{fct:angular-aproxhat} we have that $\kappa\sqrt{\ss}=O(1)$,
  so by definition of $R$ we get that $R-\widehat{r}=\Theta(R)=\Theta(\log(e^{\beta R}/(\kappa\sqrt{\ss})))$.
If on the other hand $\widehat{r}>1$, again by Fact~\ref{fct:angular-aproxhat}, we have that $\widehat{r}=\frac{1}{\beta}\log(\kappa\sqrt{\ss})+\Theta(1)$,
  so analogously to the previous calculations we obtain
  $R-\widehat{r}=\Theta(\log(e^{\beta R}/(\kappa\sqrt{\ss})))$.
Plugging back  establishes the claim when $\alpha=\beta$.  

For $\alpha > \beta$, since $\widehat{r}\leq R-\Theta(1)$, the claim follows 
because the integral therein
equals $\Theta(e^{(\alpha-\beta)R})$.

Finally, when $\alpha<\beta$, the integral in the claim is
  $\Theta(e^{-(\beta-\alpha)\widehat{r}})$.
If $\widehat{r}>1$,  by Fact~\ref{fct:angular-aproxhat},
$e^{-(\beta-\alpha)\widehat{r}}
= \Theta((\kappa\sqrt{\ss})^{-(1-\frac{\alpha}{\beta})})$.
If $\widehat{r}\leq 1$, then $e^{-(\beta-\alpha)\widehat{r}}=O(1)$ and, again by Fact~\ref{fct:angular-aproxhat}, also $\kappa\sqrt{\ss}=O(1)$.
So, no matter the value of $\widehat{r}$ 
the claim holds when $\alpha<\beta$ which concludes the proof of the claim for all cases.
\end{proof}

\medskip
We now have all the required ingredients to prove this section's main result.

\begin{proof}{(of Theorem~\ref{thm:mainangular})}
Throughout the ensuing discussion we let $\dt:=\dt(\kappa)$.
We begin by showing the uniform upper and lower bounds on $\P_{x_0}(T_{det}\leq \ss)$. To do so observe first that if $x_0\in B_Q(R)\subseteq \dt$, clearly $\P_{x_0} (T_{det} \le \ss)=1$ for any $\ss\ge0$ and hence the uniform lower bound follows directly for said $x_0$. Assume henceforth that $x_0\not\in B_Q(R)$ and observe that since there is only angular movement, a particle initially located at $(r_0,\theta_0)$ detects $Q$ if and only if it reaches $(r_0,\phi(r_0))$ or $(r_0,-\phi(r_0))$. Now, recall that the angular movement's
law is that of a variance $1$ Brownian motion $B_{\II{\ss}}$ with $\II{\ss}:=\int_{0}^{\ss} \cosech^2(\beta r_\ss)d\ss = (\phi^{(\ss)})^{2}$, so 
\begin{equation}\label{eqn:angular-exitt0}
\P_{x_0}(T_{det}\leq \ss)=\P(H_{[-a,b]}\leq \II{\ss})
\end{equation}
where we have used (with a slight abuse of notation) $\P$ for the law of a standard Brownian motion, and where $H_{[-a,b]}$ is its exit time from the interval $[-a,b]$ where in this case $a:=\phi(r_0)-|\theta_0|$ and $b:=2\pi-\phi(r_0)-|\theta_0|$. This last probability is a well studied function of~$a,b$ and $\II{\ss}$ (see~\cite{Borodin2002}, formula 3.0.2),
which can be bounded using the reflection principle and the fact that~$a\leq b$, giving 
\begin{equation}\label{eqn:angular-exitt}
\P(H_{[-a,b]}\leq \II{\ss})=\Theta\big(\P(B_{\II{\ss}}\leq -a)\big)=\Theta\big(\Phi\big((\phi(r_0)-|\theta_0|)/\phi^{(\ss)}\big)\big).
\end{equation}
From our assumption $x_0\not\in B_Q(R)$ we deduce that $|\theta_0|>\phi(r_0)$, and hence the argument within~$\Phi$ above is always negative. It follows that for $\theta_0>0$ the mapping $\theta_0\mapsto\Phi\big((\phi(r_0)-\theta_0)/\phi^{(\ss)}\big)$ is decreasing, and so both uniform bounds on $\P_{x_0}(T_{det}\leq\ss)$ follow from the definition of $\dt(\kappa)$.

We next establish the integral bound. Let $\dt:=\dt(\kappa)$.
From~\eqref{eqn:angular-exitt0} and~\eqref{eqn:angular-exitt} we observe that
\begin{align*}
    \int_{\ndt}\P_{x_0}(T_{det}\leq\ss)d\mu(x_0)
    = \Theta(ne^{-\alpha R})\int_{\widehat{r}}^R\int_{\phi(r_0)+ \kappa\phi^{(\ss)}}^{\pi}\Phi\big((\phi(r_0)-\theta_0)/\phi^{(\ss)}\big)\sinh(\alpha r_0)d\theta_0dr_0.
\end{align*}
Applying the change of variables $y_0:=(\theta_0-\phi(r_0))/\phi^{(\ss)}$ and bounding $\pi$ by $\infty$ in the upper limit of the integral we obtain 
\begin{align*}
    \int_{\ndt}\P_{x_0}(T_{det}\leq\ss)d\mu(x_0)
    & = O(ne^{-\alpha R})\int_{\widehat{r}}^R\int_{\kappa}^{\infty}\Phi({-}y_0)\sinh(\alpha r_0)\phi^{(\ss)} dy_0dr_0 \\
    & =  O(ne^{-\alpha R}\sqrt{\ss})\int_{\widehat{r}}^R e^{-\beta r_0}\sinh(\alpha r_0)dr_0.
\end{align*}
The last expression is the same as one encountered in the proof of 
Lemma~\ref{lem:angular-muDt}. Substituting by the values obtained therein one gets a term which, by Lemma~\ref{lem:angular-muDt}, is $O(\mu(\dt(\kappa))$, thus completing the proof of the claimed integral upper bound.
\end{proof}

\section{Radial movement}\label{sec:radial}
The basic structure of this section is similar to Section~\ref{sec:angular}.
However, we now consider radial movement only. We define the relevant set $\dt:=\dt(\kappa)$ depending on a parameter $\kappa\geq 1$ independent of $n$, then we compute $\mu(\dt)$ and afterwards 
separately establish the stated upper and lower bounds. Since the radial movement contains a drift towards the boundary that makes the calculations more involved, we first need to prove basic results on such diffusion processes before actually defining $\dt$.
Let us thus start with the definition of the radial movement of a given particle, initially located at $x_0=(r_0, \theta_0)$. Recall that
a particle which at time~$\ss$ is in position
$x_s=(r_s,\theta_s)$ will stay at an angle $\theta_s=\theta_0$ while its
radial distance from the origin $r_s$ evolves according to a diffusion
process with a reflecting barrier at $R$ and generator 
\[
\Delta_{rad} := \frac{1}{2}\frac{\partial^2}{\partial r^2}+\frac{\alpha}{2}\frac{1}{\tanh(\alpha r)}\frac{\partial}{\partial r}.
\]
We are only concerned with the evolution of the process up until the detection time of the target, which occurs when the particle reaches $B_Q(R)$, and since the particles can only move radially, for any $x_0\not\in B_Q(R)$ we can also impose an absorbing barrier for $r_{s}$ at the radius $\rabs_0$ corresponding to the point in $\partial B_Q(R)$ with angle $\theta_0$. Recall that by Part~\eqref{itm:phiInv} of Lemma~\ref{lem:phi}, the function 
$\phi:[0,R]\to [\phi(R),\frac{\pi}{2}]$ has an inverse $\phi^{-1}$ 
which is also decreasing and continuous, so the absorbing barrier is given by $\rabs_0=\phi^{-1}(|\theta_0|)$.
This means that for values of $|\theta_0|>\frac{\pi}{2}$ we choose as absorbing barrier 
the origin $O$, that is,  $\rabs_0=0$.
Also recall that, whatever the value of $\theta_0$, we 
have that $(\theta_0,\rabs_0)$ always belongs to the boundary of $B_Q(R)$.
Since near the origin~$O$ the drift towards the boundary grows to infinity, for a point 
$x_0$ such that $|\theta_0|>\frac{\pi}{2}$ we have $\P_{x_0}(T_{det}\leq t)=0$ (in other words, at these angles the only way to detect $Q$ would be by reaching the origin, but since the drift there is $-\infty$ this is impossible). For the case where the absorbing barrier is distinct from the origin  we use the following result, which we state as a standalone result since it will be of use in later sections as well (by abuse of notation, since $\Delta_{rad}$ defined above depends only on the radial (one-dimensional) movement we use in the following lemma the operator $\Delta_{rad}$ also to denote the specific operator acting over sufficiently smooth one variable functions over the reals):
\begin{lemma}\label{lemmaradial}
	Let $\{\auxy_s\}_{s\geq 0}$ be a diffusion process on $(0,\auxY]$ with generator $\Delta_{rad}$, and with a reflecting barrier at $\auxY$, and let $\P_{\auxy}$ denote the law of $\auxy_s$ with initial position $\auxy$. Define also $T_{\yabs_0}$, $T_\auxY$ the hitting times of~$\yabs_0$ and $\auxY$ respectively. We have:
	\begin{enumerate}[(i)]
	\item\label{radial:itm:phi1} For any $\lambda>0$ and any $\auxy\in[\yabs_0,\auxY]$,
	\begin{equation*}
	\EE_{\auxy}(e^{-\lambda T_{\yabs_0}})\,\leq\,\frac{\lambda_1 e^{-\lambda_2 (\auxY-\auxy)}+\lambda_2 e^{\lambda_1(\auxY-\auxy)}}{{\lambda_1 e^{-\lambda_2 (\auxY-\yabs_0)}+\lambda_2 e^{\lambda_1(\auxY-\yabs_0)}}}\end{equation*}
	where $\lambda_1=\sqrt{\frac{\alpha^2}{4}+2\lambda}+\frac{\alpha}{2}$ and $\lambda_2=\sqrt{\frac{\alpha^2}{4}+2\lambda}-\frac{\alpha}{2}$.
	\item\label{radial:itm:phi2} If $\auxy\in[\yabs_0,\auxY]$, then
	\[\EE_{\auxy}(T_{\yabs_0})\,\leq\,\EE_{\auxY}(T_{\yabs_0})\,\leq\,\frac{e^{\alpha \auxY}}{\alpha^2}\log(\cotanh(\tfrac12\alpha \yabs_0)).\]
	In particular, if $\yabs_0>\tfrac{1}{\alpha}\log 2$, then 
	$\displaystyle
	\EE_{\auxy}(T_{\yabs_0})\,\leq\,\frac{4}{\alpha^2} e^{\alpha(\auxY-\yabs_0)}$.
	\item\label{radial:itm:phi3} If $\auxy\in[\yabs_0,\auxY]$, then
	\[G_{\yabs_0}(\auxy):=\PP_{\auxy}(T_{\yabs_0}<T_\auxY) = \frac{\log\big(\tfrac{\tanh(\alpha \auxY/2)}{\tanh(\alpha \auxy/2)}\big)}{\log\big(\tfrac{\tanh(\alpha \auxY/2)}{\tanh(\alpha \yabs_0/2)}\big)}.
	\]
	\item\label{radial:itm:phi4} If  $\auxy\in[\yabs_0,\auxY)$, then $\displaystyle
	\EE_{\auxy}(T_{\yabs_0}\,\big|\,T_{\yabs_0}<T_\auxY)\leq \frac{2}{\alpha}(\auxy-\yabs_0)+\frac{2}{\alpha^2}(1-G_{\yabs_0}(\auxy))$.
	\end{enumerate}
\end{lemma}
\begin{proof}
We begin with the proof of~\eqref{radial:itm:phi1} by observing that $\tanh(x)\leq 1$ for all positive values of~$x$ and hence we can couple the trajectory of a particle~$P$ with that of an auxiliary particle $\widetilde{P}$ starting with the same initial position as $P$, but whose radius $\widetilde{\auxy}_\ss$ evolves according to the diffusion with generator
\[
\widetilde{\Delta}_{rad}(f) := \frac{1}{2}f''+\frac{\alpha}{2}f',
\]
in such a way that $\widetilde{\auxy}_\ss\leq y_\ss$ for all $\ss$. It follows that the detection time $\widetilde{T}_{\yabs_0}$ of this auxiliary particle is smaller than the one of $P$ so in particular $\EE_{\auxy}(e^{-\lambda T_{\yabs_0}})\leq\EE_{\auxy}(e^{-\lambda \widetilde{T}_{\yabs_0}})$, and it suffices to prove the inequality for the auxiliary process. Let now $g$ be the solution of the following O.D.E.,
\begin{equation}\label{ODE1}
\frac{1}{2}g''(\auxy)+\frac{\alpha}{2}g'(\auxy)-\lambda g(\auxy)=0
\end{equation}
on $[\yabs_0,\auxY]$ with boundary conditions $g(\yabs_0)=1$ and $g'(\auxY)=0$, which is equal to
\[g(\auxy) = \frac{\lambda_1 e^{-\lambda_2 (\auxY-\auxy)}+\lambda_2 e^{\lambda_1(\auxY-\auxy)}}{{\lambda_1 e^{-\lambda_2 (\auxY-\yabs_0)}+\lambda_2 e^{\lambda_1(\auxy-\yabs_0)}}}\]
where $\lambda_1$ and $\lambda_2$ are as in the statement of the lemma. It follows from Itô's lemma that $\{e^{-\lambda s}g(\widetilde{\auxy}_s)\}_{s \ge 0}$ is a bounded martingale, and hence we can apply Doob's optional stopping theorem to deduce $g(\auxy)=\EE_{\auxy}(e^{-\lambda\cdot 0}g(\widetilde{\auxy}_0))=\EE_{\auxy}(e^{-\lambda\widetilde{T}_{\yabs_0}})$, giving the result. To obtain the bound in~\eqref{radial:itm:phi2}, we go back to the original process $\{\auxy_s\}_{s \ge 0}$ which evolves according to $\Delta_{rad}$, and define $F(\auxy)$ as the solution of the O.D.E. on $[\yabs_0,\auxY]$
\begin{equation}\label{eq(ii)}
-1 = \frac{1}{2}F''(\auxy)\,+\,\frac{\alpha}{2}\cotanh(\alpha \auxy)F'(\auxy)
\end{equation}
with boundary conditions $F'(\auxY)=0$ and $F(\yabs_0)=0$. We advance that the solution is smooth and bounded and deduce from Itô's lemma that $\{F(\auxy_s)+s\}_{s\ge 0}$ is a martingale, so applying Doob's optional stopping theorem we deduce $F(\auxy)=\EE_{\auxy}(F(\auxy_0)+0)=\EE_{\auxy}(F(\auxy_{t\wedge T_{\yabs_0}})+t\wedge T_{\yabs_0})$ for every $t>0$. Choosing any $c>0$, a simple argument obtained by restarting the process at $\auxY$ every $c$ units of time gives
\[\PP_\auxy(T_{\yabs_0}>t)\leq \PP_\auxY(T_{\yabs_0}>t)\leq(\PP_{\auxY}(T_{\yabs_0}>c))^{\lfloor\frac{t}{c}\rfloor},\]
and hence $\lim_{t\to\infty}t\PP_\auxy(T_{\yabs_0}>t)=0$. We deduce then that $\lim_{t\to\infty}\EE_\auxy(t\wedge T_{\yabs_0})=\EE_\auxy(T_{\yabs_0})$ and since $F$ is bounded, $\lim_{t\to\infty}\EE_{\auxy}(F(\auxy_{t\wedge T_{\yabs_0}})= \EE_{\auxy}(F(\auxy_{T_{\yabs_0}}))=0$. Thus, $F(\auxy)=\EE_\auxy(T_{\yabs_0})$, and it remains to solve the O.D.E. To do so, we  multiply~\eqref{eq(ii)} by $2\sinh(\alpha \auxy)$ to obtain
\[-2\sinh(\alpha \auxy) = \sinh(\alpha \auxy)F''(\auxy)+\alpha\cosh(\alpha \auxy)F'(\auxy) = (\sinh(\alpha \auxy)F'(\auxy))'.\]
Thus, integrating from $\auxy$ to $\auxY$ and using that $F'(\auxY)=0$ we have
\[\frac{2}{\alpha}(\cosh(\alpha \auxY)-\cosh(\alpha \auxy)) = \sinh(\alpha \auxy)F'(\auxy),\]
which in particular proves directly that $F(\auxy)$ is an increasing function, so that $\EE_{\auxy}(T_{\yabs_0})\leq \EE_\auxY(T_{\yabs_0})$. Integrating from $\yabs_0$ to $\auxY$, together with the condition $F'(\auxY)=0$ gives
\[
\EE_\auxY(T_{\yabs_0}) = F(\auxY) = \frac{2}{\alpha^2}\Big(\log\Big(\frac{\sinh(\alpha \yabs_0)}{\sinh(\alpha \auxY)}\Big)-\cosh(\alpha \auxY)\log\Big(\frac{\tanh(\tfrac12\alpha\yabs_0)}{\tanh(\tfrac12\alpha \auxY)}\Big)\Big)
\]
and hence the general bound appearing in~\eqref{radial:itm:phi2} follows by noticing that the first term is negative, and by bounding $\cosh(\alpha \auxY)$ by $\frac{1}{2}e^{\alpha \auxY}$. To obtain $\EE_{\auxy}(T_{\yabs_0})\leq \frac{4}{\alpha^2} e^{\alpha(\auxY-\yabs_0)}$ observe that if we assume $\yabs_0>\tfrac{1}{\alpha}\log 2$ then $\cotanh(\tfrac12\alpha\yabs_0)\leq 1+4e^{-\alpha\yabs_0}$ so the result follows from bounding $\log(1+4e^{-\alpha\yabs_0})\leq4e^{-\alpha\yabs_0}$. 

To establish~\eqref{radial:itm:phi3} and~\eqref{radial:itm:phi4} it will be enough to work with a diffusion evolving on $(0,\infty)$ according to the original generator $\Delta_{rad}$, but without any barriers. Abusing notation we still call the process $\{\auxy_s\}_{s \ge 0}$. To deduce~\eqref{radial:itm:phi3}, observe that the solution of the O.D.E.
\[
0 = \frac{1}{2}G_{\rabs_0}''(\auxy) + \frac{\alpha}{2}\cotanh(\alpha \auxy)G_{\yabs_0}'(\auxy)
\]
with conditions $G_{\yabs_0}(\yabs_0)=1$ and $G_{\yabs_0}(\auxY)=0$ is given by the closed expression given in~\eqref{radial:itm:phi3}. It follows from Itô's lemma that $\{G_{\yabs_0}(\auxy_s)\}_{s \ge 0}$ is a bounded martingale, so applying Doob's optional stopping theorem we deduce $G_{\yabs_0}(y)=\EE_\auxy(G_{\yabs_0}(\auxy_0))=\EE_\auxy(G_{\yabs_0}(\auxy_{T_{\yabs_0}\wedge T_\auxY}))=\PP_{\auxy}(T_{\yabs_0}<T_\auxY)$.

Finally, to prove~\eqref{radial:itm:phi4} define the function $H(\auxy)$ as the solution of the ordinary differential equation
\[-G_{\yabs_0}(\auxy) = \frac{1}{2}H''(\auxy)\,+\,\frac{\alpha}{2}\cotanh(\alpha \auxy)H'(\auxy)\]
with boundary conditions $H(\yabs_0)=H(\auxY)=0$. It can be checked directly that the last equation is satisfied by
{\footnotesize
\begin{equation}\label{eq(iv)}
H(\auxy)=\frac{2}{\alpha}G_{\yabs_0}(r)\int_{\yabs_0}^{\auxy}\!\sinh(\alpha l)G_{\yabs_0}(l)\log\Big(\frac{\tanh(\frac{\alpha l}{2})}{\tanh(\frac{\alpha \yabs_0}{2})}\Big)dl+\frac{2}{\alpha}(1{-}G_{\yabs_0}(\auxy))\int_{\auxy}^\auxY\!\sinh(\alpha l)G_{\yabs_0}(l)\log\Big(\frac{\tanh(\frac{\alpha \auxY}{2})}{\tanh(\frac{\alpha l}{2})}\Big)dl,
\end{equation}}
which is smooth. It follows once again from Itô's lemma that $\{H(\auxy_s)+\int_0^s G_{\yabs_0}(\auxy_u)du\}_{s \ge 0}$ is a martingale. Since in the proof of~\eqref{radial:itm:phi3} we already showed that $\{G_{\yabs_0}(\auxy_s)\}_{s \ge 0}$ is a martingale, it follows that $\{\int_0^s G_{\yabs_0}(\auxy_u)du-s G_{\yabs_0}(\auxy_s)\}_{s \ge 0}$ is also a martingale. We conclude that $\{H(\auxy_s)+s G_{\yabs_0}(\auxy_s)\}_{s \ge 0}$ is a martingale, so applying Doob's optional stopping theorem we deduce
\[H(\auxy)=\EE_{\auxy}(H(\auxy_0)+0\cdot G_{\yabs_0}(y_0))=\EE_{\auxy}(H(\auxy_{t\wedge T_{\yabs_0}\wedge T_{\auxY}})+(t\wedge T_{\yabs_0}\wedge T_{\auxY})\cdot G_{\yabs_0}(\auxy_{t\wedge T_{\yabs_0}\wedge T_{\auxY}}))\]
for every $t>0$. Reasoning as in the proof of~\eqref{radial:itm:phi2} we can take the limit as $t\to\infty$ to obtain
\[H(\auxy)=\EE_{\auxy}(H(\auxy_{T_{\yabs_0}\wedge T_{\auxY}})+(T_{\yabs_0}\wedge T_{\auxY})\cdot G_{\yabs_0}(\auxy_{T_{\yabs_0}\wedge T_{\auxY}}))=\EE_{\auxy}(T_{\yabs_0}{\bf1}_{\{T_{\yabs_0}<T_{\auxY}\}})\]
where we used that $H(\yabs_0)=H(\auxY)=G_{\yabs_0}(\auxY)=0$ and $G_{\yabs_0}(\yabs_0)=1$. Observing that $\EE_{\auxy}(T_{\yabs_0}\,|\,T_{\yabs_0}<T_\auxY)=\frac{H(\auxy)}{G_{\yabs_0}(\auxy)}$, to obtain the inequality in~\eqref{radial:itm:phi4} we only need to bound $H(\auxy)$. To do so notice that for any $\yabs_0\leq l$ gives $\frac{\tanh(\alpha l/2)}{\tanh(\alpha \yabs_0/2)}\leq \frac{1}{\tanh(\alpha l/2)}$, which together with $0\leq G_{\yabs_0}(l)\leq 1$ allows us to bound from above the first integral of~\eqref{eq(iv)} by $\sinh(\alpha l)\log(\cotanh(\frac12\alpha l))=\frac{2\cotanh(\frac12\alpha l)}{\cotanh^2(\frac12\alpha l)-1}\log (\cotanh(\frac12\alpha l))=f(\cotanh(\frac12\alpha l))$, where $f(z)=\frac{2z}{z^2-1}\log z$ is bounded from above by $1$ on $[1,\infty)$. Using the same argument we can bound the term within the second integral of~\eqref{eq(iv)} by $G_{\yabs_0}(l)$, so that
\[\frac{H(\auxy)}{G_{\yabs_0}(\auxy)}\leq\frac{2}{\alpha}(\auxy-\yabs_0)+\frac{2}{\alpha}(1-G_{\yabs_0}(\auxy))\int_{\auxy}^\auxY \frac{G_{\yabs_0}(l)}{G_{\yabs_0}(\auxy)}dl.\]
Using the fact that $\frac{G_{\yabs_0}(l)}{G_{\yabs_0}(\auxy)}=G_{\auxy}(l)$, the second integral becomes $\int_{\auxy}^\auxY G_{\auxy}(l)dl$ which we control by studying the function $\auxy\mapsto \int_{\auxy}^\auxY G_{\auxy}(l)dl$ on $(0,\auxY)$. Notice first that any critical point $\auxy'$ of said function satisfies
\[\int_{\auxy'}^\auxY G_{\auxy'}(l)dl=\frac{1}{\alpha}\log\Big(\frac{\tanh(\frac12\alpha \auxY)}{\tanh(\frac12\alpha \auxy')}\Big)\sinh(\alpha \auxy')\leq\frac{1}{\alpha},\]
so it will be sufficient to control the integral when either $\auxy=\auxY$ or $\auxy=0$. For the first case we have $\int_{\auxY}^\auxY G_{\auxY}(l)dl=0$, and for the second one, by definition we have $\lim_{\auxy\to 0}G_{\auxy}(l)=0$ for any $l>0$. 
Since $G_\auxy(l)$ is monotone increasing in $\auxy$, by the monotone convergence theorem, $\lim_{\auxy\to 0}\int_\auxy^\auxY G_{\auxy}(l)dl=\int \lim_{\auxy \to 0} G_\auxy(l) dl=0$, so putting all these cases together we conclude $\int_\auxy^\auxY G_{\auxy}(l)dl\leq \frac{1}{\alpha}$.
\end{proof}

\medskip
Before we define $\dt(\kappa)$, let
\[
\phi^{(\ss)}:=\Big(\frac{\ss^{\frac{1}{\alpha}}}{e^{R}}\Big)^{\frac{1}{2}}.
\] 
Intuitively, one may think of $\phi^{(\ss)}$ as those points that are so close in terms of angle to the target, so that - even though possibly initially at the boundary of $B_O(R)$ - they have a reasonable chance to detect the target by time $\ss$ through the radial movement. 
We define~$\dt(\kappa)$ as the collection of points where a particle initially located can detect the target before time $\ss$ with a not too small probability (depending on $\kappa$). From our discussion preceding Lemma~\ref{lemmaradial}, we will always assume that any point $x_0\in\dt$ satisfies $|\theta_0|\leq\frac{\pi}{2}$. Since $\PP_{r}(T_{\rabs_0}<T_R)=G_{\rabs_0}(r)$, with $G_{\rabs_0}(r)$ as defined in Lemma~\ref{lemmaradial} with $y:=r$, $\yabs_0:=\rabs_0$ and $Y:=R$, it follows that $G_{\rabs_0}(r)$ is decreasing as a function of $r$, continuous and takes values in $[0,1]$. In particular, for $\kappa>1$ 
there is a unique $\widetilde{\rabs}_0\in [\rabs_0,R]$ such that
\begin{equation}\label{eqn:radial-defrtilde}
\PP_{\widetilde{\rabs}_0}(T_{\rabs_0}<T_R)=
\delta(\kappa,\ss), 
\quad\text{ where }\quad
\delta=\delta(\kappa,\ss):= \frac{1+\ss}{(1+\kappa\ss^{\frac{1}{2\alpha}})^{2\alpha}}.
\end{equation}
Note that $0\leq\delta(\kappa,\ss)< 1$ (the latter inequality holds because we assume $\kappa\geq 1$). Also observe that $\delta(\kappa,0)=1$ and $\delta(\kappa,\ss)$ tends to $\kappa^{-2\alpha}$ as $\ss$ tends to infinity.
Furthermore, $\delta(\kappa,\ss)=\Theta(\kappa^{-2\alpha})$ if $\ss=\Omega(1)$.
Now, define (see Figure~\ref{fig:radial}) 
\[
\dt = \dt(\kappa) := \big\{ x_0\in B_O(R) : |\theta_0|\leq\tfrac{\pi}{2} \wedge \big[|\theta_0|\leq\phi(R)+\kappa\phi^{(\ss)} \vee  r_0\leq\widetilde{\rabs}_0\big]\big\}.
\]
which contains $B_Q(R)$ since every $x_0\in B_Q(R)$ satisfies $r_0<\rabs_0\leq\widetilde{\rabs}_0$.
To better understand the motivation for defining $\widetilde{\rabs}_0$ and $\delta$ as above, we consider the most interesting regime, i.e.,~$\ss=\Omega(1)$. Under this condition the effect of the drift has enough time to move the particle far away from its initial position, and towards the boundary, so that the event $\{T_{\rabs_0}<\ss,\,T_{\rabs_0}<T_R\}$ is mostly explained by the particles' initial trajectory. In particular, by Part~\eqref{radial:itm:phi3} of Lemma~\ref{lemmaradial}, we have that $\PP_{r_0}(T_{\rabs_0}<\ss,\,T_{\rabs_0}<T_R)\approx \PP_{r_0}(T_{\rabs_0}<T_R)=G_{\rabs_0}(r_0)$,  so the condition $r_0\leq\widetilde{\rabs_0}$ aims to include in $\dt$ all points whose probability of detecting the target before reaching the boundary of $B_O(R)$ is not too small. To exhaust all possibilities, we must also include in $\dt$ all points which have a sufficiently large probability of detecting the target even after reaching the boundary of $B_O(R)$. Said points are considered through the condition $|\theta_0|\leq\phi(R)+\kappa\phi^{(\ss)}$, which gives a lower bound of order $\delta(\kappa,\ss)$ for the detection probability, thus explaining our choice of said function.

\medskip
Before moving to the main theorem of this section, we spend some time building some intuition about the geometric shape of $\dt$. 
Since $\rabs_0$ goes to $0$ when $\theta_0$ tends to $\frac{\pi}{2}$, from~\eqref{eqn:radial-defrtilde} which is used to define $\widetilde{\rabs}_0$, it is not hard to see that 
$\widetilde{\rabs}_0$ also goes to $0$
(and hence $\widetilde{\rabs}_0-\rabs_0$ as well) when $\theta_0$ tends to $\frac{\pi}{2}$.
It requires a fair amount of additional work to show that $\widetilde{\rabs}_0-\rabs_0$, as a function of $\theta_0>\phi(R)+\kappa\phi^{(\ss)}$, first increases very slowly, then reaches a maximum value of roughly $\frac{1}{\alpha}\log\frac{1}{\delta}$
and finally decreases rapidly (we omit the details since we will not rely on this observation). In fact, for all practical purposes, one might think of $\widetilde{\rabs}_0-\rabs_0$ as being essentially constant up to the point when $\rabs_0$ is smaller than a constant (equivalently, $\theta_0$ is at least a constant).

The main goal of this section is to prove the following result from which Theorem~\ref{thm:radialmain} immediately follows (by the proof strategy discussed in Section~\ref{sec:strategy}) once we show that~$\mu(\dt(\kappa))$ is of the right order:
\begin{theorem}\label{thm:rad} 
The following hold:
\begin{enumerate}[(i)]
\item If $\kappa\geq 1$ and $\phi(R)+\kappa\phi^{(\ss)}\leq \frac{\pi}{2}$, then $\sup_{x_0 \in \ndt}\PP_{x_0}(T_{det}\leq\ss)=O(\delta(\kappa,\ss))$ and
\[
\int_{\ndt\!(\kappa)}\P(T_{det}\leq\ss)d\mu(x_0) = O(\mu(\dt(\kappa))). 
\]
\item\label{thm:rad:itm2} For every $c>0$ there is a $\kappa_0 >0$ such that if $\kappa\ge\kappa_0$ and $\ss= \frac{16}{\alpha}\log \kappa+\Omega(1)$ satisfy $\phi(R)+\kappa\phi^{(\ss)}\leq \frac{\pi}{2}-c$, then $\inf_{x_0 \in \dt}\PP_{x_0}(T_{det}\leq\ss)=\Omega(\delta(\kappa,\ss))$.
\end{enumerate}
\end{theorem}

\begin{remark}
The extra hypotheses on $\kappa$ and $\ss$ needed for the lower bounds of Part~\eqref{thm:rad:itm2} of Theorem~\ref{thm:rad} 
are key for our methods to work, 
since for small times $\ss= o(1)$ the detection probability for a point always tends to $0$ unless it is already in $B_Q(R)$ or very close to it (but the latter set is of smaller measure than $B_Q(R)$). Similarly, if one is very close to the origin (that is, for angles very close to $\frac{\pi}{2}$) the explosion of the drift towards the boundary at the origin also implies an extra penalization for the probability of detection. Observe nevertheless that the extra hypothesis in Part~\eqref{thm:rad:itm2} are automatically satisfied if $\ss=\omega(1)\cap o(e^{\alpha R})$.
Furthermore, the hypothesis $\ss =\Omega(1)$ is natural since in the case of radial movement only we will show that $\EE(T_{det})=\Theta(1)$, and we are interested in tail behaviors of the detection time.
\end{remark}



Among all facts regarding the intuition of $\dt$, we will only need to prove rigorously (see the next result) that if $\phi(R)+\kappa\phi^{(\ss)}\leq |\theta_0|\leq\frac{\pi}{2}$, then $\widetilde{\rabs}_0-\rabs_0\leq\frac{1}{\alpha}\log\frac{1}{\delta}+O(1)$.
\begin{fact}\label{fct:radial-varphi2}
For $\ss> 0$ and $\kappa\geq 1$, 
if $\theta_0\in [\phi(R)+\kappa\phi^{(\ss)},\frac{\pi}{2}]$, then
  $e^{\widetilde{\rabs}_0-\rabs_0}=O(1/\delta^{\frac{1}{\alpha}})$
  where $\delta:=\delta(\kappa,\ss)$.
\end{fact}
\begin{proof}
We first handle some simple cases.
If $\widetilde{\rabs}_0\leq\frac{1}{\alpha}\log\frac{2}{\delta}$, then $e^{\widetilde{r}_0-\rabs_0}\leq e^{\widetilde{r}_0}\leq (2/\delta)^{\frac{1}{\alpha}}$ and we are done.
Similarly, if $\rabs_0\geq R-\frac{1}{\alpha}\log\frac{3}{\delta}$, since $\widetilde{\rabs}_0\leq R$, we have that $e^{\widetilde{r}_0-\rabs_0}\leq e^{R-\rabs_0}\leq (3/\delta)^{\frac{1}{\alpha}}$ and we obtain again the claimed bound. Henceforth, assume that $\widetilde{\rabs}_0>\frac{1}{\alpha}\log\frac{2}{\delta}$ and that $\rabs_0< R-\frac{1}{\alpha}\log\frac{3}{\delta}$.

Let $g(r):=\log(\tanh(\frac12\alpha R)/\tanh(\frac12\alpha r))$ and note that since $\tanh(x)\leq 1$, $\cotanh(x)=1+\frac{2e^{-2x}}{1-e^{-2x}}$ and $1+y\leq e^{y}$, 
\[g(\widetilde{\rabs}_0) \leq \log(\cotanh(\tfrac12\alpha\widetilde{\rabs}_0))
\leq\frac{2e^{-\alpha\widetilde{\rabs}_0}}{1-e^{-\alpha\widetilde{\rabs}_0}}\leq4e^{-\alpha\widetilde{\rabs}_0},\]
where in the last inequality we have used that $\delta\leq 1$. Moreover, since $\tanh(x)=1-\frac{2e^{-2x}}{1+e^{-2x}}$, 
$\cotanh(x)=1+\frac{2e^{-2x}}{1-e^{-2x}}$ and
using twice that $\log(1+y)\geq \frac{y}{1+y}$ for $y>-1$, we have
\[
g(\rabs_0) = \log(\coth(\tfrac12\alpha\rabs_0))+\log(\tanh(\tfrac12\alpha R)) \geq \frac{2e^{-\alpha\rabs_0}}{1+e^{-\alpha\rabs_0}} - 
\frac{2e^{-\alpha R}}{1-e^{-\alpha R}}
\geq e^{-\alpha\rabs_0}-(2+o(1))e^{-\alpha R}.
\]
Since $\delta\leq 1$, by our assumption on $\rabs_0$, we conclude that
  $g(\rabs_0)=\Omega(e^{-\alpha\rabs_0})$. Observing that $g(\widetilde{\rabs}_0)/g(\rabs_0)=G_{\rabs_0}(\widetilde{\rabs}_0)=\delta$, it follows from the 
previously derived bounds on $g(\widetilde{\rabs}_0)$ and $g(\rabs_0)$ 
that $e^{\alpha(\widetilde{\rabs}_0-\rabs_0)}=O(\delta^{-1})$ 
and thus $e^{\widetilde{\rabs}_0-\rabs_0}=O(1/\delta^{\frac{1}{\alpha}})$ as desired.
\end{proof}
For practical purposes one may view $\dt$ as the collection of points of $B_O(R)$
which belong either to a sector of angle $2(\phi(R)+\kappa\phi^{(\ss)})$ 
whose bisector contains $Q$, 
or to the ball $B_Q(R)$, or to those points with angular coordinate $|\theta_0|\geq \phi(R)+\kappa\phi^{(\ss)}$ which are 
within radial distance roughly $\frac{1}{\alpha}\log\frac{1}{\delta}$ of $B_Q(R)$. 
This picture would be accurate except for the fact that it places into $\dt$ points with angular coordinate close to $\frac{\pi}{2}$ and fairly close to the origin $O$, say at distance 
$\frac{1}{\alpha}\log\frac{1}{\delta}-\Omega(1)$. 
However, particles initially located at such points are extremely unlikely to reach $B_Q(R)$ and detect $Q$, since the drift towards the boundary tends to infinity close to the origin.
This partly justifies why $\dt$ is defined so that in the case where $\theta_0$ goes to $\frac{\pi}{2}$ the expression $\widetilde{\rabs}_0-\rabs_0$ tends to $0$, thus 
leaving out of $\dt$ the previously mentioned problematic points.

\begin{figure}
\begin{center}
\begin{tikzpicture}[x=1cm,y=1cm,scale=0.65,
     decoration={markings,
       mark=at position 1 with {\arrow[scale=1.5,black]{latex}};
      }]
      
      \def\c{12}
      \def\barr{4}
      \def\radp{9.5}
      \def\phip{240}
      \def\angabs{-87.94}
      \node[inner sep=0] (O) at (\c,\c) {};
      \node[inner sep=0] (P) at (2,4) {};
      \node[inner sep=0] (Q) at (\c,0) {};
  \draw[fill=gray!20] (11.561cm,0.008cm) --(11.686cm,3.423cm) -- (11.656cm,3.656cm) -- (11.624cm,3.891cm) -- (11.589cm,4.126cm) -- (11.552cm,4.362cm) -- (11.511cm,4.599cm) -- (11.468cm,4.837cm) -- (11.421cm,5.076cm) -- (11.371cm,5.316cm) -- (11.318cm,5.558cm) -- (11.261cm,5.800cm) -- (11.201cm,6.045cm) -- (11.137cm,6.291cm) -- (11.070cm,6.540cm) -- (11.000cm,6.792cm) -- (10.928cm,7.047cm) -- (10.854cm,7.305cm) -- (10.778cm,7.568cm) -- (10.702cm,7.835cm) -- (10.627cm,8.108cm) -- (10.555cm,8.387cm) -- (10.487cm,8.672cm) -- (10.426cm,8.963cm) -- (10.376cm,9.261cm) -- (10.344cm,9.513cm) -- (10.339cm,9.563cm) -- (10.335cm,9.614cm) -- (10.331cm,9.665cm) -- (10.328cm,9.717cm) -- (10.325cm,9.768cm) -- (10.322cm,9.819cm) -- (10.321cm,9.870cm) -- (10.320cm,9.922cm) -- (10.319cm,9.973cm) -- (10.319cm,10.025cm) -- (10.320cm,10.076cm) -- (10.322cm,10.128cm) -- (10.324cm,10.179cm) -- (10.327cm,10.231cm) -- (10.331cm,10.282cm) -- (10.335cm,10.333cm) -- (10.341cm,10.384cm) -- (10.347cm,10.435cm) -- (10.354cm,10.486cm) -- (10.362cm,10.537cm) -- (10.371cm,10.587cm) -- (10.381cm,10.638cm) -- (10.392cm,10.688cm) -- (10.403cm,10.737cm) -- (10.416cm,10.787cm) -- (10.430cm,10.836cm) -- (10.444cm,10.884cm) -- (10.460cm,10.932cm) -- (10.477cm,10.980cm) -- (10.494cm,11.027cm) -- (10.513cm,11.074cm) -- (10.533cm,11.120cm) -- (10.554cm,11.166cm) -- (10.577cm,11.211cm) -- (10.600cm,11.255cm) -- (10.625cm,11.299cm) -- (10.651cm,11.341cm) -- (10.678cm,11.383cm) -- (10.706cm,11.424cm) -- (10.736cm,11.465cm) -- (10.767cm,11.504cm) -- (10.799cm,11.542cm) -- (10.832cm,11.579cm) -- (10.867cm,11.615cm) -- (10.904cm,11.650cm) -- (10.941cm,11.684cm) -- (10.981cm,11.717cm) -- (11.022cm,11.748cm) -- (11.065cm,11.778cm) -- (11.110cm,11.806cm) -- (11.156cm,11.834cm) -- (11.205cm,11.859cm) -- (11.257cm,11.883cm) -- (11.311cm,11.905cm) -- (11.369cm,11.926cm) -- (11.431cm,11.944cm) -- (11.499cm,11.961cm) -- (11.574cm,11.975cm) -- (11.660cm,11.987cm) -- (11.769cm,11.995cm) -- (12.000cm,12.000cm) --(12.231cm,11.995cm) -- (12.340cm,11.987cm) -- (12.426cm,11.975cm) -- (12.501cm,11.961cm) -- (12.569cm,11.944cm) -- (12.631cm,11.926cm) -- (12.689cm,11.905cm) -- (12.743cm,11.883cm) -- (12.795cm,11.859cm) -- (12.844cm,11.834cm) -- (12.890cm,11.806cm) -- (12.935cm,11.778cm) -- (12.978cm,11.748cm) -- (13.019cm,11.717cm) -- (13.059cm,11.684cm) -- (13.096cm,11.650cm) -- (13.133cm,11.615cm) -- (13.168cm,11.579cm) -- (13.201cm,11.542cm) -- (13.233cm,11.504cm) -- (13.264cm,11.465cm) -- (13.294cm,11.424cm) -- (13.322cm,11.383cm) -- (13.349cm,11.341cm) -- (13.375cm,11.299cm) -- (13.400cm,11.255cm) -- (13.423cm,11.211cm) -- (13.446cm,11.166cm) -- (13.467cm,11.120cm) -- (13.487cm,11.074cm) -- (13.506cm,11.027cm) -- (13.523cm,10.980cm) -- (13.540cm,10.932cm) -- (13.556cm,10.884cm) -- (13.570cm,10.836cm) -- (13.584cm,10.787cm) -- (13.597cm,10.737cm) -- (13.608cm,10.688cm) -- (13.619cm,10.638cm) -- (13.629cm,10.587cm) -- (13.638cm,10.537cm) -- (13.646cm,10.486cm) -- (13.653cm,10.435cm) -- (13.659cm,10.384cm) -- (13.665cm,10.333cm) -- (13.669cm,10.282cm) -- (13.673cm,10.231cm) -- (13.676cm,10.179cm) -- (13.678cm,10.128cm) -- (13.680cm,10.076cm) -- (13.681cm,10.025cm) -- (13.681cm,9.973cm) -- (13.680cm,9.922cm) -- (13.679cm,9.870cm) -- (13.678cm,9.819cm) -- (13.675cm,9.768cm) -- (13.672cm,9.717cm) -- (13.669cm,9.665cm) -- (13.665cm,9.614cm) -- (13.661cm,9.563cm) -- (13.656cm,9.513cm) -- (13.650cm,9.462cm) -- (13.644cm,9.411cm) -- (13.638cm,9.361cm) -- (13.631cm,9.311cm) -- (13.624cm,9.261cm) -- (13.617cm,9.211cm) -- (13.609cm,9.161cm) -- (13.600cm,9.111cm) -- (13.592cm,9.062cm) -- (13.583cm,9.012cm) -- (13.574cm,8.963cm) -- (13.564cm,8.914cm) -- (13.555cm,8.865cm) -- (13.545cm,8.817cm) -- (13.534cm,8.768cm) -- (13.524cm,8.720cm) -- (13.513cm,8.672cm) -- (13.502cm,8.624cm) -- (13.491cm,8.576cm) -- (13.480cm,8.529cm) -- (13.469cm,8.481cm) -- (13.457cm,8.434cm) -- (13.445cm,8.387cm) -- (13.434cm,8.340cm) -- (13.422cm,8.293cm) -- (13.410cm,8.247cm) -- (13.398cm,8.200cm) -- (13.385cm,8.154cm) -- (13.373cm,8.108cm) -- (13.361cm,8.062cm) -- (13.348cm,8.017cm) -- (13.336cm,7.971cm) -- (13.323cm,7.926cm) -- (13.311cm,7.880cm) -- (13.298cm,7.835cm) -- (13.285cm,7.790cm) -- (13.273cm,7.746cm) -- (13.260cm,7.701cm) -- (13.247cm,7.657cm) -- (13.235cm,7.612cm) -- (13.222cm,7.568cm) -- (13.209cm,7.524cm) -- (13.197cm,7.480cm) -- (13.184cm,7.436cm) -- (13.171cm,7.392cm) -- (13.159cm,7.349cm) -- (13.146cm,7.305cm) -- (13.134cm,7.262cm) -- (13.121cm,7.219cm) -- (13.109cm,7.175cm) -- (13.096cm,7.132cm) -- (13.084cm,7.090cm) -- (13.072cm,7.047cm) -- (13.060cm,7.004cm) -- (13.047cm,6.961cm) -- (13.035cm,6.919cm) -- (13.023cm,6.876cm) -- (13.011cm,6.834cm) -- (13.000cm,6.792cm) -- (12.988cm,6.750cm) -- (12.976cm,6.708cm) -- (12.964cm,6.666cm) -- (12.953cm,6.624cm) -- (12.941cm,6.582cm) -- (12.930cm,6.540cm) -- (12.918cm,6.499cm) -- (12.907cm,6.457cm) -- (12.896cm,6.416cm) -- (12.885cm,6.374cm) -- (12.874cm,6.333cm) -- (12.863cm,6.291cm) -- (12.852cm,6.250cm) -- (12.841cm,6.209cm) -- (12.831cm,6.168cm) -- (12.820cm,6.127cm) -- (12.810cm,6.086cm) -- (12.799cm,6.045cm) -- (12.789cm,6.004cm) -- (12.779cm,5.963cm) -- (12.769cm,5.922cm) -- (12.759cm,5.882cm) -- (12.739cm,5.800cm) -- (12.682cm,5.558cm) -- (12.629cm,5.316cm) -- (12.579cm,5.076cm) -- (12.532cm,4.837cm) -- (12.489cm,4.599cm) -- (12.448cm,4.362cm) -- (12.411cm,4.126cm) -- (12.376cm,3.891cm) -- (12.344cm,3.656cm) -- (12.314cm,3.423cm) -- (12.439cm,0.008cm) --(12.352cm,0.005cm) --(12.264cm,0.003cm) --(12.176cm,0.001cm) --(12.088cm,0.000cm) --(12.000cm,0.000cm) --(11.912cm,0.000cm) --(11.824cm,0.001cm) --(11.736cm,0.003cm) --(11.648cm,0.005cm) --(11.561cm,0.008cm) --(12.439cm,0.008cm) --cycle;

\draw[fill=gray!40] (11.941cm,0.000cm) -- (11.902cm,1.200cm) -- (11.842cm,2.401cm) -- (11.748cm,3.604cm) -- (11.607cm,4.811cm) -- (11.404cm,6.030cm) -- (11.136cm,7.278cm) -- (10.842cm,8.591cm) -- (10.798cm,8.820cm) -- (10.758cm,9.051cm) -- (10.725cm,9.285cm) -- (10.698cm,9.521cm) -- (10.681cm,9.760cm) -- (10.675cm,9.999cm) -- (10.681cm,10.239cm) -- (10.704cm,10.477cm) -- (10.744cm,10.711cm) -- (10.804cm,10.938cm) -- (10.885cm,11.154cm) -- (10.988cm,11.356cm) -- (11.113cm,11.538cm) -- (11.260cm,11.696cm) -- (11.426cm,11.825cm) -- (11.608cm,11.921cm) -- (11.801cm,11.980cm) -- (12.000cm,12.000cm) --(12.199cm,11.980cm) -- (12.392cm,11.921cm) -- (12.574cm,11.825cm) -- (12.740cm,11.696cm) -- (12.887cm,11.538cm) -- (13.012cm,11.356cm) -- (13.115cm,11.154cm) -- (13.196cm,10.938cm) -- (13.256cm,10.711cm) -- (13.296cm,10.477cm) -- (13.319cm,10.239cm) -- (13.325cm,9.999cm) -- (13.319cm,9.760cm) -- (13.302cm,9.521cm) -- (13.275cm,9.285cm) -- (13.242cm,9.051cm) -- (13.202cm,8.820cm) -- (13.158cm,8.591cm) -- (13.112cm,8.366cm) -- (13.063cm,8.144cm) -- (13.013cm,7.924cm) -- (12.963cm,7.707cm) -- (12.913cm,7.492cm) -- (12.864cm,7.278cm) -- (12.815cm,7.067cm) -- (12.768cm,6.857cm) -- (12.723cm,6.649cm) -- (12.679cm,6.441cm) -- (12.636cm,6.235cm) -- (12.596cm,6.030cm) -- (12.557cm,5.825cm) -- (12.521cm,5.621cm) -- (12.486cm,5.418cm) -- (12.453cm,5.215cm) -- (12.422cm,5.013cm) -- (12.393cm,4.811cm) -- (12.366cm,4.609cm) -- (12.340cm,4.408cm) -- (12.316cm,4.206cm) -- (12.293cm,4.005cm) -- (12.272cm,3.805cm) -- (12.252cm,3.604cm) -- (12.158cm,2.401cm) -- (12.098cm,1.200cm) -- (12.059cm,0.000cm) -- (12.059cm,0.000cm) -- (12.045cm,0.000cm) -- (12.030cm,0.000cm) -- (12.015cm,0.000cm) -- (12.000cm,0.000cm) -- (11.985cm,0.000cm) -- (11.970cm,0.000cm) -- (11.955cm,0.000cm) -- (11.941cm,0.000cm) -- cycle;
\draw[fill=black] (O) ++(\phip:\radp) circle (0.07) node[left] {$x_{\ss}$};
\draw[fill=black] (O) ++(\phip:8.0) circle (0.07) node[left] {$x_0$};
\draw[fill=black] (O) ++(\phip:2.62) circle (0.07) node[right] {$_{A_0}$};
\draw[dotted] (O) -- ++(\phip:2.62);
\draw[fill=black] (O) ++(\phip:3.22) circle (0.07) node[left] {$_{B_0}$};
\draw[dashed] (O)  ++(\phip:2.62) -- ++(\phip:9.38);

      \draw[dotted] (0,\c) -- (2*\c,\c);
      \draw[dotted] (O) -- (Q);
      \draw[dotted] (O) -- ++(\angabs:\c+2);
      \draw[dotted] (\c,-0.75) -- (\c,-2);
      \draw (\c,-1.5) arc (-90:\angabs:\c+1.5);
      \draw[postaction={decorate}] ({(\c+1.5)*cos(-90-5)+\c},{(\c+1.5)*sin(-90-5)+\c}) arc (-90-5:-90:\c+1.5);
      \draw[postaction={decorate}] ({(\c+1.5)*cos(-83.1)+\c},{(\c+1.5)*sin(-83.1)+\c}) arc (\angabs+5:\angabs:\c+1.5);
      \draw[postaction={decorate}] (\c,\c-\radp) arc (-90:\phip-360:\radp);
      \node at (\c+0.9,-1.1) {$_{\phi(R)+\kappa\phi^{(\ss)}}$};
      \node at (\c-2.5,\c-8.7) {$_{\theta_{\ss}=\theta_0}$};
      \draw[thick,black] (2*\c,\c) arc (0:-180:\c);
      \draw[fill=black] (Q) circle (0.07);
      \draw[fill=black] (O) circle (0.07);
      \node[below] at (Q) {$_Q$};      
      \node[above] at (O) {$_O$}; 
	\end{tikzpicture}
\end{center}
  \caption{Half of disk $B_O(R)$.  The position $x_{\ss}:=(r_{\ss},\theta_{\ss})$ of particle $P$ at time
    $\ss$ is shown. The lightly shaded area corresponds to
    $\dt(\kappa)\setminus B_{Q}(R)$ and the strongly shaded
    region represents~$B_Q(R)$. The dashed line corresponds to the radial
    segment where a particle $P$ at initial angle~$\theta_0$ is,
    conditional under not having detected the target~$Q$ at time
    $\ss$. The radial coordinates of $A_0$ and $B_0$ are $\rabs_0:=\rabs_0(\theta_0)$
    and $\widetilde{\rabs}_0:=\widetilde{\rabs}_0(\theta_0)$,
    respectively.}\label{fig:radial}
\end{figure}
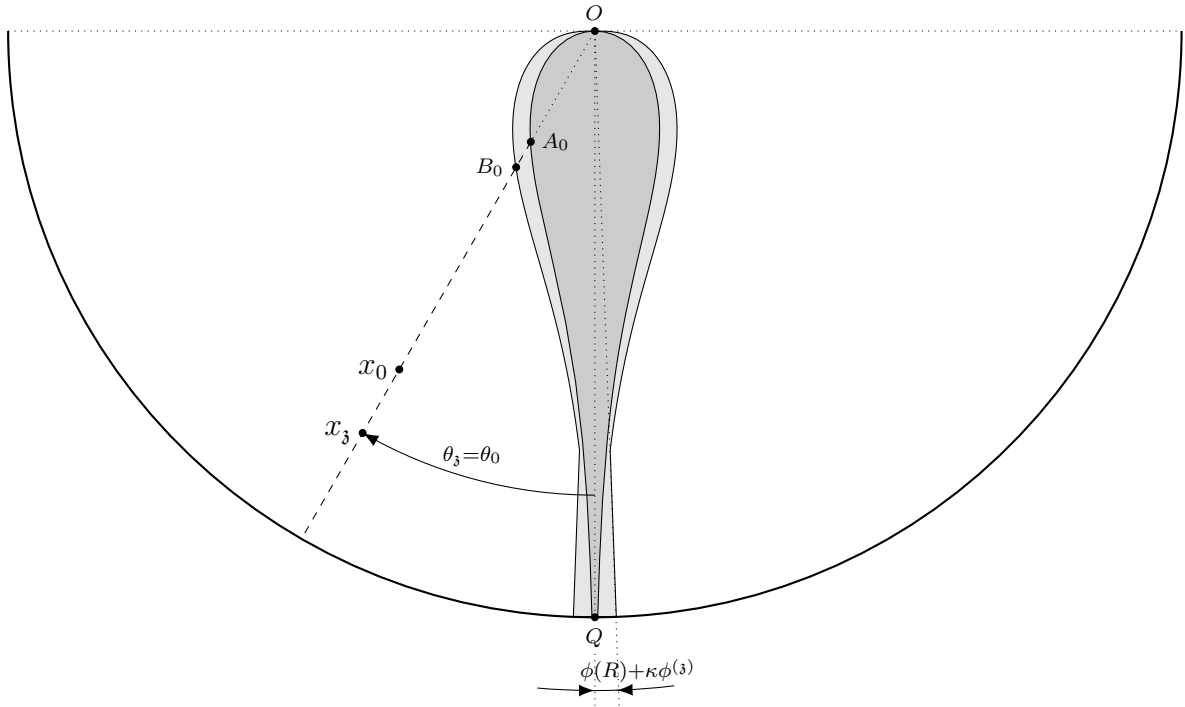

\medskip
We next determine $\mu(\dt)$, which simply amounts to performing an integral. 
\begin{lemma}\label{lem:radialmeasure}
If  $\kappa\ge 1$ are such that $\phi(R)+\kappa\phis\le\frac{\pi}{2}$, then  
\[
\mu(\dt(\kappa)) = \Theta\big(n(\phi(R)+\kappa\phi^{(\ss)})\big) = \Theta(\mkor{\nu(1+\kappa\ss^{\frac{1}{2\alpha}})}{1+\kappa\ss^{\frac{1}{2\alpha}}}).
\]
\end{lemma}
\begin{proof}
For the lower bound, simply observe that $\dt:=\dt(\kappa)$ contains a sector $\Upsilon_Q$ 
  of angle $2(\phi(R)+\kappa\phi^{(\ss)})$ on whose bisector lies the target $Q$,
  so  $\mu(\dt)\geq\mu(\Upsilon_Q)=\Theta(n(\phi(R)+\kappa\phi^{(\ss)}))$.
Now, we address the upper bound. 
It suffices to show that $\mu(\dt\setminus\Upsilon_Q)=O(\mu(\Upsilon_Q))$.
Indeed, observe that
\[
\mu(\dt\setminus\Upsilon_Q) 
= O(ne^{-\alpha R})\int_{\phi(R)+\kappa\phi^{(\ss)}}^{\frac{\pi}{2}}\int_0^{\widetilde{\rabs}_0}\frac{1}{2\pi}\sinh(\alpha r_0)dr_0d\theta_0
= O(ne^{-\alpha R})\int_{\phi(R)+\kappa\phi^{(\ss)}}^{\frac{\pi}{2}}e^{\alpha\widetilde{\rabs}_0}d\theta_0.
\]
Thus, since $e^{\alpha\widetilde{\rabs}_0}=e^{\alpha\rabs_0}e^{\alpha(\widetilde{\rabs}_0-\rabs_0)}$, by Fact~\ref{fct:radial-varphi2},
\[
\mu(\dt\setminus\Upsilon_Q) = 
  O(ne^{-\alpha R}\delta^{-1})
  \int_{\phi(R)+\kappa\phi^{(\ss)}}^{\frac{\pi}{2}}e^{\alpha\rabs_0}d\theta_0.
\]
Recall that by definition $|\theta_0|=\phi(\rabs_0)$, so that by Part~\eqref{itm:phi4} of Lemma~\ref{lem:phi}, we have  $|\theta_0|=\phi(\rabs_0)=\Theta(e^{-\frac12\rabs_0})$, 
and
\[
 \mu(\dt\setminus\Upsilon_Q) = O(ne^{-\alpha R}\delta^{-1})\int_{\phi(R)+\kappa\phi^{(\ss)}}^{\frac{\pi}{2}}\theta_0^{-2\alpha}d\theta_0
= O(ne^{-\alpha R}\delta^{-1}(\phi(R)+\kappa\phi^{(\ss)})^{-(2\alpha-1)}).
\]
By our choices of $\phi^{(\ss)}$ and $\delta$ 
  we have $e^{-\alpha R}\delta^{-1}=\frac{1}{1+\ss}(\phi(R)+\kappa\phi^{(\ss)})^{2\alpha}
  \leq (\phi(R)+\kappa\phi^{(\ss)})^{2\alpha}$.
Thus, $\mu(\dt\setminus\Upsilon_Q)=O(n(\phi(R)+\kappa\phi^{(\ss)}))=O(\mu(\Upsilon_Q))$ as claimed. 
\end{proof}

\subsection{Uniform upper bound}\label{sec:upper_rad}
The goal of this subsection is to prove the following result from which 
the upper bounds of Theorem~\ref{thm:rad} immediately follow.
\begin{proposition}\label{prop:rad-lowerBnd}
If  $\kappa\geq 1$, then
\[
\sup_{x_0 \in \ndt\!(\kappa), |\theta_0| < \frac{\pi}{2}}  \P_{x_0}(T_{det}\leq\ss)=O(\delta(\kappa,\ss)) 
\quad\text{ and } \quad 
\sup_{x_0 \in \ndt\!(\kappa), |\theta_0| \geq \frac{\pi}{2}}  \P_{x_0}(T_{det}\leq\ss)=0. 
\]
Furthermore, if $\phi(R)+\kappa\phis\leq\frac{\pi}{2}$, then
\[
\int_{\ndt\!(\kappa)}\P_{x_0}(T_{det}\leq\ss)d\mu(x_0) = O(\mu(\dt(\kappa))).
\]
\end{proposition}
\begin{proof}
By the discussion previous to Lemma~\ref{lemmaradial}, we have $\P_{x_0}(T_{det}\leq\ss)=0$
  when $x_0=(r_0,\theta_0)\in\ndt$ is such that $|\theta_0|\geq\frac{\pi}{2}$, as claimed.
Henceforth, let $x_0=(r_0,\theta_0)\in\ndt$ be such that
$|\theta_0|<\frac{\pi}{2}$.  
As in Lemma~\ref{lemmaradial}, 
let $\P_{r_0}$ be the law of $r_\ss$ with initial position $r_0$
and $T_{\rabs}$ be the hitting time of $\rabs$.

Observe that a particle initially at $x_0$ reaches the boundary of $B_{Q}(R)$ either without ever visiting the boundary of $B_O(R)$ (which happens with probability at most $\P_{r_0}(T_{\rabs_0}<T_R))$ or after visiting the 
said boundary for the first time (in which case the time to reach the boundary of $B_Q(R)$ is dominated by the time of a particle starting at the boundary of $B_O(R)$ and having 
the same angular coordinate as the point $x_0$).
Thus,
\[
\P_{r_0}(T_{\det}\leq\ss) \leq \P_{r_0}(T_{\rabs_0}<T_R)+\P_{R}(T_{\rabs_0}\leq\ss).
\]
We will show that each of the right hand side terms above is $O(\delta)$
where $\delta:=\delta(\kappa,\ss)$. This immediately implies
the first part of the proposition's statement.

By our choice of $\widetilde{\rabs}_0$
and since $r_0\geq\widetilde{\rabs}_0$, we have 
\[
\P_{r_0}(T_{\rabs_0}<T_R) \leq \P_{\widetilde{\rabs}_0}(T_{\rabs_0}<T_R)
= \delta.
\]
By Part~\eqref{radial:itm:phi1} of Lemma~\ref{lemmaradial} with $\auxy:=r$, $\yabs_0:=\rabs_0$ and $\auxY:=R$, and Markov's inequality,
for $\lambda>0$, $\lambda_1:=\sqrt{\frac{\alpha^2}{4}+2\lambda}+\frac{\alpha}{2}>0$ and $\lambda_2:=\sqrt{\frac{\alpha^2}{4}+2\lambda}-\frac{\alpha}{2}$, we get that
\begin{equation}
\P_{r_0}(T_{\rabs_0}\leq\ss) 
= \P_{r_0}(e^{-\lambda T_{\rabs_0}}\geq e^{-\lambda\ss})\leq \frac{\lambda_1 e^{-\lambda_2(R-r_0)}+\lambda_2 e^{\lambda_1(R-r_0)}}{\lambda_2e^{\lambda_1(R-\rabs_0)}}e^{\lambda \ss}.
\label{eqn:radial-upper-aux}
\end{equation}
Observing that $\lambda_1\lambda_2=2\lambda$, that $\lambda_2>0$ and $\lambda_1>\alpha$,
and taking $\lambda=\frac{1}{1+\ss}<1$ (since by hypothesis $\ss> 0$) it follows that
\begin{equation}\label{eqn:radial-upper-aux2}
\P_{R}(T_{\rabs_0}\leq\ss) 
\leq \frac{\lambda_1+\lambda_2}{\lambda_2 e^{\lambda_1(R-\rabs_0)}}\cdot e^{\lambda\ss}
= O\Big(\frac{e^{\lambda\ss}}{\lambda}e^{-\lambda_1(R-\rabs_0)}\Big)
= O((1+\ss) e^{-\alpha(R-\rabs_0)}).
\end{equation}
By our choice of $\delta$, we have $1+\ss\leq\delta(1+\kappa\ss^{\frac{1}{2\alpha}})^{2\alpha}
=\Theta(\delta e^{\alpha R}(\phi(R)+\kappa\phi^{(\ss)})^{2\alpha})$.
Moreover, by Part~\eqref{itm:phi4} of Lemma~\ref{lem:phi}, we have 
$|\theta_0|=\phi(\rabs_0)=\Theta(e^{-\frac12\rabs_0})$.
Since by hypothesis $x_0\in\ndt$, we know that $|\theta_0|\geq\phi(R)+\kappa\phi^{(\ss)}$,
it follows that $e^{-\frac12\rabs_0}=\Omega(\phi(R)+\kappa\phi^{(\ss)})$.
Thus, 
\[
\P_{R}(T_{\rabs_0}\leq\ss) 
= O\big(\delta (\phi(R)-\kappa\phis)^{2\alpha}e^{\alpha \rabs_0}\big)=O(\delta)
\]
which establishes the first part of our stated result.

Next, we consider the second stated claim (the integral bound).
By~\eqref{eqn:radial-upper-aux}, denoting by~$\mu_{\theta_0}$ the measure induced by $\mu$ for the angle $\theta_0$, recalling that $\lambda_1-\lambda_2=\alpha$, it follows that
\begin{align*}
\int_{\widetilde{\rabs}_0}^R\P_{x_0}(T_{det}\leq\ss)d\mu_{\theta_0}(r_0)
&= O(n)\int_{\widetilde{\rabs}_0}^R\frac{\lambda_1 e^{-\lambda_2 (R-r_0)}+\lambda_2 e^{\lambda_1(R-r_0)}}{{\lambda_2 e^{\lambda_1(R-\rabs_0)}}}\cdot e^{\lambda\ss}\cdot e^{-\alpha (R-r_0)}dr_0\\
&= O(n)\frac{e^{\lambda_2 (R-\widetilde{\rabs}_0)}-e^{-\lambda_1(R-\widetilde{\rabs}_0)}}{{\lambda_2 e^{\lambda_1(R-\rabs_0)}}}\cdot e^{\lambda\ss} \\
&= O(n) \frac{e^{\lambda_2 (R-\widetilde{\rabs}_0)}}{\lambda_2 e^{\lambda_1 (R-\rabs_0)}}\cdot e^{\lambda\ss}. 
\end{align*}
Using once more that $\lambda_1-\lambda_2=\alpha$ and $\lambda_1\lambda_2=2\lambda$, taking again $\lambda=\frac{1}{1+\ss}< 1$,
we obtain
\[
\int_{\widetilde{\rabs}_0}^{R}\P_{x_0}(T_{det}\leq\ss)d\mu_{\theta_0}(r_0)
= O(n (1+\ss) e^{-\alpha (R-\rabs_0)}e^{-\lambda_2(\widetilde{\rabs}_0-\rabs_0)}).
\]
Therefore, since $\widetilde{\rabs}_0\geq\rabs_0$ and $\lambda_2>0$ imply that 
$\lambda_2(\widetilde{\rabs}_0-\rabs_0)$ is positive, it follows that
\[
  \int_{\ndt}\P_{x_0}(T_{det}\leq \ss)d\mu(x_0)
  = O(n(1+\ss)e^{-\alpha R})\int_{\phi(R)+\kappa\phi^{(\ss)}}^{\frac{\pi}{2}}e^{\alpha\rabs_0}d\theta_0.
\]
Using that $|\theta_0|=\phi(\rabs_0)=\Theta(e^{-\frac12\rabs_0})$, since
  $\alpha>\frac12$, by definition of $\delta$ we get
\[
\int_{\ndt}\P_{x_0}(T_{det}\leq \ss)d\mu(x_0)
  = O(n(1+\ss)e^{-\alpha R})\int_{\phi(R)+\kappa\phi^{(\ss)}}^{\frac{\pi}{2}}\theta_0^{-2\alpha}d\theta_0 
  = O(n\delta(\phi(R)+\kappa\phi^{(\ss)})).
\]
Since $\delta\le 1$, the claimed integral bound follows from Lemma~\ref{lem:radialmeasure}.
\end{proof}

\subsection{Uniform lower bound}\label{sec:lower_rad}
The goal of this subsection is to prove the following lower bound matching the upper bound of Proposition~\ref{prop:rad-lowerBnd}  under the assumption that at least a constant amount of time has passed since the process started (as stated in Theorem~\ref{thm:rad}).
\begin{proposition}\label{prop:rad-upperBnd}
For every fixed arbitrarily small constant $c>0$
there are large enough  constants $\kappa_0, C\geq 1$ such that 
if $\kappa\geq\kappa_0$ and $\ss\geq\frac{16}{\alpha}\log\kappa+C$ satisfy $\phi(R)+\kappa\phi^{(\ss)}\leq\frac{\pi}{2}-c$, then 
\[
\inf_{x_0\in\dt(\kappa)}\P_{x_0}(T_{det}\leq \ss) = \Omega(\kappa^{-2\alpha}).
\]
\end{proposition}
\begin{proof}
Let $\dt:=\dt(\kappa)$.
To obtain the lower bound on the detection probability recall that any $x_0\in\dt$ either satisfies $|\theta_0|\leq\phi(R)+\kappa\phi^{(\ss)}$ or $r_0\leq \widetilde{\rabs}_0$ where $\widetilde{\rabs}_0$ by definition is such that $G_{\rabs_0}(\widetilde{\rabs}_0)=\delta(\kappa,\ss)$. We deal first with the case $\rabs_0<r_0\leq \widetilde{\rabs}_0$ (the case $r_0\leq \rabs_0$ is trivial since it implies that $x_0\in B_Q(R)$ and thus $T_{det}=0$) by noticing that 
\begin{align*}\P_{x_0}(T_{det}\leq \ss)&\geq\,\P_{x_0}(T_{\rabs_0}\leq \ss\,\big|\,T_{\rabs_0}<T_{R})\P_{x_0}(T_{\rabs_0}<T_{R})\\&\geq\,\big(1-\tfrac{1}{\ss}\EE_{x_0}(T_{\rabs_0}\,\big|\,T_{\rabs_0}<T_{R})\big)\P_{x_0}(T_{\rabs_0}<T_{R}).
\end{align*}
Using Part~\eqref{radial:itm:phi4} of Lemma~\ref{lemmaradial} with $\auxy:=r$, $\yabs_0:=\rabs_0$ and $\auxY:=R$, recalling that $G_{\rabs_0}(r)=\P_{r}(T_{\rabs_0}<T_{R})$, since by case assumption $r_0\leq\widetilde{\rabs}_0$ and recalling again that by definition of $\widetilde{\rabs}_0$ it holds that $G_{\rabs_0}(\widetilde{\rabs}_0)=\delta:=\delta(\kappa,\ss)$, we get
\[
\P_{x_0}\big(T_{det}\leq \ss\big)\geq 
  \Big(1-\frac{1}{\ss}\Big[\frac{2}{\alpha}(\widetilde{\rabs}_0-\rabs_0)+\frac{2}{\alpha^2}(1-\delta)\Big]\Big)\delta.
\]
From Fact~\ref{fct:radial-varphi2} it follows that $\widetilde{\rabs}_0-\rabs_0=\frac{1}{\alpha}\log \tfrac{1}{\delta}+O(1)$, and since we can bound from above $1-\delta$ by $\log\tfrac{1}{\delta}$, we get that the term inside the brackets above is at most $\frac{4}{\alpha^2}\log\frac{1}{\delta}+O(1)$. By hypothesis, $\ss\ge C\ge 1$, so $\delta\geq (1+\kappa)^{-2\alpha}=\Omega(\kappa^{-2\alpha})$ 
and hence, using again the hypothesis regarding $\ss$, we see that $\ss>\frac{16}{\alpha}\log\kappa+C\geq \frac{8}{\alpha^2}\log\frac{1}{\delta}+\frac12 C$ for $C$ large enough, so taking $C$ even larger if needed we conclude that 
\[
\P_{x_0}\big(T_{det}\leq \ss\big)=\Omega(\delta)=\Omega(\kappa^{-2\alpha}).
\]
We now handle the points $x_0:=(r_0,\theta_0)\in\dt$ for which $|\theta_0|\leq\phi(R)+\kappa\phi^{(\ss)}$. Observe that for any fixed~$\theta_0$ the detection probability is minimized when $r_0=R$, and for points such that $r_0=R$ the probability decreases with $|\theta_0|$, so it will be enough to consider the case where $r_0=R$ and $|\theta_0|=\phi(R)+\kappa\phi^{(\ss)}$. To obtain lower bounds on the detection probability we will couple the radial movement $r_t$ of the particle starting at $x_0$ with a similar one, denoted by $\widetilde{r}_t$, evolving as follows:
\begin{itemize}
    \item $\widetilde{r}_0=R$ and we let $\widetilde{r}_t$ evolve through the generator $\Delta_{rad}$ on $(0,R+1]$.
    \item Every time $\widetilde{r}_t$ hits $R+1$ it is immediately restarted at $R$.
\end{itemize}
It follows that $r_t$ and $\widetilde{r}_t$ can be coupled so that $r_t\leq\widetilde{r}_t$ for every $t\geq0$, and hence the detection time $\widetilde{T}_{\rabs_0}$ of the new process bounds the original $T_{det}=T_{\rabs_0}$ from above. Thus, it will be enough to bound $\PP_{x_0}(\widetilde{T}_{\rabs_0}\leq\ss)$ from below. Observe that the trajectories of $\widetilde{r}_t$ are naturally divided into excursions from $R$ to $R+1$, which we use to define a sequence of Bernoulli random variables $\{E_i\}_{i\geq1}$ where $E_i=1$ if $\widetilde{r_t}$ reaches $\rabs_0$ on the $i$-th excursion. We also use this division to define a sequence $\{\tau_i\}_{i\geq1}$ of random variables, where $\tau_i$ is the time it takes $\widetilde{r}_t$ to reach either $\rabs_0$ or $R+1$ in the $i$-th excursion. It follows from the strong Markov property that all excursions are independent of one another, and so are the $E_i$'s and $\tau_i$'s.

Fix $m:=\lfloor \tfrac{\alpha}{24}\ss\rfloor$ which from our hypothesis on $\ss$ satisfies $m\geq 1$, and observe that, since the $E_i$'s are i.i.d., the event $\mathcal{E}:=\{\exists 1\leq i\leq m,\,E_i=1\}$ has probability
\[
\PP(\mathcal{E})=1-(1-\PP(E_1))^m\geq 1-e^{-m\PP(E_1)}.
\]
From Part~\eqref{radial:itm:phi3} of Lemma~\ref{lemmaradial} with $\auxy:=R$, $\yabs_0:=\rabs_0$ and $\auxY:=R+1$, using  $\log(1+x)\geq\frac{x}{2}$ for $|x|<1$
and $\log(\cotanh(\frac{x}{2}))\leq \frac{2e^{-x}}{1-e^{-x}}$, we have
\[
\PP(E_1)=\frac{\log\big(\tfrac{\tanh(\alpha (R+1)/2)}{\tanh(\alpha R/2)}\big)}{\log\big(\tfrac{\tanh(\alpha R/2)}{\tanh(\alpha \rabs_0/2)}\big)}
\geq\frac{\log\big(1+\frac{2e^{-\alpha R}(1-e^{-\alpha})}{(1+e^{-\alpha(R+1)})(1-e^{-\alpha R})}\big)}{\log\big(\cotanh(\alpha\frac{\rabs_0}{2})\big)}
=\Omega(e^{-\alpha(R-\rabs_0)}(1-e^{-\alpha\rabs_0})).
\]
Since $\phi(R)+\kappa\phi^{(\ss)}\leq\frac{\pi}{2}-c$, we get $\rabs_0=\Omega(1)$ and deduce that $\PP(E_1)=\Omega(e^{-\alpha(R-\rabs_0)})$.
Recalling that $\phi(R)+\kappa\phi^{(\ss)}=|\theta_0|=\phi(\rabs_0)$, by Part~\eqref{itm:phi4} of Lemma~\ref{lem:phi}, we have that
$e^{-\frac12\rabs_0}=\Theta(\phi(R)+\kappa\phi^{(\ss)})$ and $\phi(R)=\Theta(e^{-\frac12R})$.
Since by hypothesis $\ss,\kappa\geq 1$, it follows that
$e^{\frac12(R-\rabs_0)}=\Theta(e^{\frac12 R}(\phi(R)+\kappa\phi^{(s)}))=\Theta(\kappa\ss^{\frac{1}{2\alpha}})$.
Thus, $\PP(E_1)=\Omega(\kappa^{-2\alpha}/\ss)$ so setting $\kappa_0$ sufficiently large we get
\begin{equation*}
\PP(\mathcal{E}) \geq 
1-e^{-\Omega(\kappa^{-2\alpha})}
=\Omega(\kappa^{-2\alpha}).
\end{equation*}
Clearly, if $i_D$ is the first index such that $E_{i_D}=1$, then $\mathcal{E}=\{1\leq i_D\leq m\}$. Hence, 
\[\P_{x_0}(T_{det}\leq \ss)\geq \P_{x_0}(T_{det}\leq \ss\mid\mathcal{E})\P_{x_0}(\mathcal{E})=\P_{x_0}\Big(\sum_{i=1}^{i_D}\tau_i\leq \ss \mid 1\leq i_D\leq m\Big)\PP(\mathcal{E}),\]
where the $\tau_i$ were defined above. Using Markov's inequality we obtain
\[\P_{x_0}\Big(\sum_{i=1}^{i_D}\tau_i> \ss\,\big|\,1\leq i_D\leq m\Big)\leq \frac{1}{\ss}\big(m\EE_{x_0}(\tau_1 \mid E_1=0)+\EE_{x_0}(\tau_1 \mid E_1=1)\big),\]
since there are at most $m$ excursions where the particle fails to reach $\rabs_0$, followed by a single excursion where it succeeds (here we are conditioning on these events through the value of $E_1$). For the first term, we have
\[\EE_{x_0}(\tau_1 \mid E_1=0)=\EE_{R}(\widetilde{T}_{R+1} \mid \widetilde{T}_{R+1}\leq \widetilde{T}_{\rabs_0})\,\leq\, \frac{6}{\alpha},\]
as can be seen in~\cite{Borodin2002} (formulas 3.0.4 and 3.0.6).  by comparing the process to one with constant drift equal to $\tfrac{\alpha}{2}$. For the second term we use Part~\eqref{radial:itm:phi4} of Lemma~\ref{lemmaradial} with $\auxy:=R$, $\yabs_0:=\rabs_0$ and $\auxY:=R+1$ to obtain
\[
\EE_{x_0}(\tau_1 \mid E_1=1)=\EE_{R}(\widetilde{T}_{\rabs_0} \mid \widetilde{T}_{\rabs_0}<\widetilde{T}_{R+1})\leq \frac{2}{\alpha}(R+1-\rabs_0)+\frac{2}{\alpha^2}.
\]
Recalling that
$e^{\frac12(R-\rabs_0)}=\Theta(\kappa\ss^\frac{1}{2\alpha})$, 
we have $R-\rabs_0=\log(\kappa\ss^{\frac{1}{2\alpha}})+\Theta(1)$.
Summarizing,
\[\P_{x_0}\Big(\sum_{i=1}^{i_D}\tau_i> \ss\mid 1\leq i_D\leq m\Big)= \frac{6m}{\alpha\ss}+
\frac{1}{\ss}\Big(\frac{2}{\alpha}\log(\kappa\ss^{\frac{1}{2\alpha}})+O(1)\Big).\] 
By our choice of $m$ it is immediate that $\frac{6m}{\alpha\ss}\leq\frac14$. 
To bound the last term of the displayed equation recall that by hypothesis $\ss\geq\frac{16}{\alpha}\log\kappa+C$, 
so that we have $\frac{2}{\alpha}\log\kappa\le\frac18\ss$ and moreover we can choose 
$C$ large enough so that $\frac{1}{\alpha^2}\log\ss\leq\frac{1}{8}\ss$ and finally conclude that
\[
\P_{x_0}(T_{det}\leq \ss)=\Big(1-\P_{x_0}\Big(\sum_{i=1}^{i_D}\tau_i> \ss\,\big|\,1\leq i_D\leq m\Big)\Big)\PP(\mathcal{E})=\Omega(\kappa^{-2\alpha}).
\]
\end{proof}
Observe that the bounds established in the preceding proposition are exactly those claimed in the first part of  Theorem~\ref{thm:rad}.

\subsection{Time spent within an interval containing the origin}
In this section we consider a class of one-dimensional diffusion processes taking place in $[0,\auxY]$ with $\auxY$ a positive integer  
and with given stationary  distribution.
	Our study is motivated by the process $\{\auxy_s\}_{s\geq 0}$ in $(0,\auxY]$ with generator~$\Delta_h$
	(see~\eqref{radialgenerator}) and reflecting barrier at $\auxY$. This section's main result
	will later in the paper be applied precisely to this latter process. However, the arguments of this section are applicable to less constrained  diffusion processes and might even be further
	generalized to a larger class of processes, and so we have opted for an exposition which, instead
	of tailoring the proof arguments to the specific 
	process with generator $\Delta_h$, gives the conditions that need to be satisfied so that
	our results as a corollary also apply to our specific process.

    Our goal is to formalize a rather intuitive fact. Specifically, assume
    $\ss$ is not a too small amount of time and that 
    $\{\auxy_s\}_{s\geq 0}$ is a diffusion process with stationary distribution~$\pi(\cdot)$ on $[0,\auxY]$. Roughly speaking, we show that under some general hypotheses, during a time interval of length $\ss$, when starting from the stationary distribution,  with constant probability the process spends within $(0,k]$ ($k$ larger than a constant) a period of time proportional to the expected time to spend there, that is, proportional to $\ss\pi((0,k])$. We remark that $k$ only needs to be larger than a sufficiently large constant $C$, which does not depend on $Y$ (otherwise the result would follow directly from Markov's inequality), and $k$ does not depend on $\pi$ either.


	We start by establishing two lemmas. The first one states that a diffusion process $\{\auxy_s\}_{s\geq 0}$ in $[0,\auxY]$ with stationary distribution $\pi(\cdot)$,
    henceforth referred to as the \emph{original process}, can be coupled with a process $\mathfrak{D}$ that is continuous in time but discretized in space with values in $\mathbb{Z}$ and defined next.
    First, let $\pi_m:=\pi((m-1,m])$ where $1\le m\le \auxY$. 
    Let $p'_1=0$ (this restriction could be relaxed, but we impose it for the sake of simplicity of presentation), and let $p'_2,\ldots, p'_{\auxY}$ be the unique solution to the following system of equations:
    \begin{equation}\label{radial:def:pim}
    \pi_{m} = 
    \begin{cases} 
    p'_{m}\pi_{m}+p'_{m+1}\pi_{m+1}, & \text{if $m=1$,} \\    
    (1-p'_{m-1})\pi_{m-1}+p'_{m+1}\pi_{m+1}, & \text{if $1<m<\auxY$,} \\
    (1-p'_{m-1})\pi_{m-1}+(1-p'_{m})\pi_{m}, & \text{if $m=\auxY$.} 
    \end{cases}
    \end{equation}
    Intuitively, one can think of the $p'_m$ as the probability that the process
    hits $m-1$ before it hits $m$ when starting from the stationary distribution conditioned on starting in the interval~$(m-1,m]$. 
	Denoting next, for $0 \le m < \auxY$, by
	$\widehat{p}_m$ the probability that the original process starting at $m$ hits  $m-1$ before hitting $m+1$ (observe that $\widehat{p}_0=0$). 
	We then define
	$$
	p:=\sup\{\sup_{m\in \NN: 1 \le m\le\auxY} p'_m, \sup_{m \in \NN: 0 \le m < \auxY} \widehat{p}_m\},
	$$
	and let $q:=1-p$.
	 We are exclusively interested in processes for which the following holds:
	\begin{quote}
	    \mycom{hyp:H1}{(H1)} 
	    $p \in [\underline{c},\overline{c}]$, $0<\underline{c}\leq\overline{c}<1/2$, 
	    for constants $\underline{c},\overline{c}>0$ 
	    that do not depend on $\auxY$. 
	\end{quote}
	In other words, we are interested in the case
	where the original process is subject to a drift towards the boundary $\auxY$.
	
	Note however that condition~\linktomycom{hyp:H1}  does not preclude that the process $\{\auxy_s\}_{s\geq 0}$, being at position $\auxy$ (for some real $0 \le \auxy \le \auxY$), in expectation stays within an interval $(y-1,y+1] \cap [0,\auxY]$ an amount of time that increases with $\auxY$. This explains why we focus on processes that satisfy the next stated additional hypothesis concerning the maximal time $T_{\auxy\pm 1}$   the process stays within said interval:
	\begin{quote}
	   \mycom{hyp:H2}{(H2)} There is a constant $L>0$  (independent of $\auxY$) such that 
	   for all reals $0\leq \auxy \le \auxY$ it holds that 
	   $\EE_{y}(T_{\auxy\pm 1})\le L$.
	\end{quote}
	The preceding hypothesis says that the process stays, when starting at position $y$, in expectation, at most constant time within any giving interval $(y-1,y+1]$.
	A last assumption that we make in order to derive this section's main result is the following: 
	\begin{quote}
	   \mycom{hyp:H3}{(H3)} There is a constant
	   $\widetilde{C}$, independent of $\auxY$, such that if $\widetilde{C}\le k\le \auxY$, $\ss\geq 1/\pi((0,k])$ and 
	   $\auxy_s \le k$ at some time moment $s\ge 0$, then with constant probability $\auxy_t \le k$ for all $s \le t \le s+1$.
	\end{quote}
	The preceding hypothesis says that only in the vicinity of $0$ there might exist an infinite drift (towards $Y$). Once we reach a point that is a constant away from $0$, we might stay at roughly this value for at least one unit of time with constant probability (this is a technical assumption that allows us to focus only on the probability to reach such a value when considering the expected time spent roughly there.
	
	Now, we are ready to define the process $\mathfrak{D}$: set the initial position of $\mathfrak{D}$ as $\auxy_0^{\mathfrak{D}} := \lfloor \auxy_0 \rfloor$. Suppose the original process starts in an interval $(m-1,m]$ for some integer $1 \le m \le \auxY$.  If the original process hits $m$ before hitting $m-1$, the original process stays put at $m-1$ and the \emph{coupling phase} defined below starts.  If the original process hits $m-1$ before hitting~$m$, then $\mathfrak{D}$ jumps to $m-2$ and the coupling phase starts as well. In the coupling phase now iteratively the following happens: suppose the original process is initially at integer $m$, for some $1 \le m < \auxY$, and~$\mathfrak{D}$ is at some integer $m'$ (if $m=\auxY$, directly go to the \emph{move-independent phase} defined below). Mark~$m$ as the center of the current phase, and do the following: As long as the original process neither hits $m-1$ nor $m+1$, $\mathfrak{D}$ stays put. If the original process hits $m-1$, then~$\mathfrak{D}$ jumps to $m'-1$ (note that this happens with probability $\widehat{p}_m \le p$. If the original process hits $m+1$, then~$\mathfrak{D}$ jumps to $m'-1$ with probability $\frac{p-\widehat{p}_m}{\widehat{q}_m},$ and with probability $1-\frac{p-\widehat{p}_m}{\widehat{q}_m},$ $\mathfrak{D}$ jumps to $m'+1$. Note that the total probability that $\mathfrak{D}$ jumps to $m'-1$ is exactly $p$. Moreover, if $m+1=\auxY$, the move-independent phase starts. In either case, $m$ is unmarked as the center of the current phase, and the new location of the original process (either $m-1$ or $m+1$, respectively; note that $m-1=0$ could also happen, this is treated in the same way), is marked as the new center, and the next iteration of the coupling phase (move-independent phase, respectively) with this new location playing the role of $m$. 
	
	In the move-independent phase \emph{only the instants} of movements of the original process and $\mathfrak{D}$ are coupled, but not the movements itself: suppose that at the beginning of this phase the original process is at position $m$ for some integer $0 \le m \le \auxY$, whereas $\mathfrak{D}$ is at some integer $m'$: at the instant when the original process hits $m-1$ or $m+1$, $\mathfrak{D}$ jumps to $m'-1$ with probability $p$ and to $m'+1$ with probability $q$. Then, the next iteration of the independent phase starts with the new location of the original process (that is, either $m-1$ or $m+1$) playing the role of $m$.
	
	\noindent\par
	By construction, it is clear that from the beginning of the coupling phase on, $\mathfrak{D}$ jumps from $m'$ to $m'-1$ with probability~$p$ and from $m'$ to $m'+1$ with probability $1-p$. We also have the following observation that also follows immediately by construction:
	\begin{observation}\label{obs:domination}
	    As long as the original process $\{\auxy_s\}_{s \ge 0}$ does not hit $\auxY$, at every time instant $s$, $\auxy_s^{\mathfrak{D}} \le \auxy_s$.
	\end{observation}

	We next show that for $T$ being a sufficiently large constant $C_L$ that depends on $L$ only, conditional under not yet having hit the boundary, $\mathfrak{D}$ is likely to have performed already quite a few steps up to time $T$: 
	\begin{lemma}\label{geometricdominationlemma}
Assume~\linktomycom{hyp:H2} holds. Let $C_L$ be a large enough constant depending on $L$ only with $L$ as in~\linktomycom{hyp:H2}. 
For every constant $c_1$, there exists a constant $c_2 > 0$, depending only on~$c_1$ and $L$, such that for any integer $T \ge C_L$, with probability at least $1-e^{c_1 T}$, at least $c_2 T$ steps are performed by $\mathfrak{D}$ up to time $T$.
	\end{lemma}
	\begin{proof} 
	Fix any real $0 \le y \le \auxY$, and recall that  $T_{y\pm 1}$ denotes the time the original process remains in the interval $(y-1,y+1] \cap [0,\auxY]$. By~\linktomycom{hyp:H2} and Markov's inequality, with probability at most $1/2$, the original process exits $(y-1,y+1]$ before time $2 L$. Formally, for all $0 \le y \le \auxY$, we have $\PP_{y}(T_{y \pm 1}\ge 2 L) \leq 1/2$. 
Since this bound is uniform over the starting position $y$ we can restart the process every $2 L$ units of time, thus obtaining
\begin{equation}\label{eq.geometric}
    \P(T_{y\pm 1}\geq 2kL)\leq 2^{-k},
\end{equation} for every integer $k\ge 1$, regardless of $y$. Hence, if at the beginning of an iteration of the coupling phase the original process is at some integer $0 \le m \le \auxY$, the time it takes $\mathfrak{D}$ to wait until its next step from that time moment on, is bounded in the same way as $T_{m\pm 1}$, and the time for the first step to enter the coupling phase can be bounded in the same way as in~\eqref{eq.geometric}. Now, for $i\ge 1$, let $T^{(i)}$ denote the time spent between the instants $T_{i-1}$ and $T_i$, that mark the beginning and end of the $i$-th iteration of the coupling phase (the time before the start of the first iteration of the coupling phase, in case $i=1$, respectively). By~\eqref{eq.geometric}, the random variable~$\lfloor\frac{T^{(i)}}{2 L}\rfloor$ is stochastically bounded from above by a geometric random variable with success probability $\frac{1}{2}$ and support $\{1,2,3,\ldots\}$. The random variables $T^{(1)}, T^{(2)}, \ldots$ are independent but not identically distributed, as the time to remain in some intervals might be longer than in others, but since~\eqref{eq.geometric} holds for all $y$, we may bound $\sum_i \lfloor \frac{T^{(i)}}{2L} \rfloor$ by the sum of  i.i.d.\ geometric random variables with success probability $\frac12$. Hence, for any constant~$c_1$, there is a sufficiently small~$c_2>0$ such that
\[\P\Big(\sum_{i=0}^{c_2 T-1}T^{(i)}> T\Big)\leq \P\Big(\sum_{i=0}^{c_2 T-1}\left\lfloor\frac{T^{(i)}}{2L}\right\rfloor> \frac{T}{3L} \Big)\le \sum_{j=\lfloor \frac{T}{3L} \rfloor+1}^\infty \binom{j-1}{c_2 T -1}2^{-j}\leq e^{-c_1 T},\]
where we used that the sum of $c_2 T$ i.i.d.~geometric random variables has a negative binomial distribution, and where $c_2$ is chosen (depending on $c_1$ and $L$ only) so that the last inequality holds.
Hence, with probability at least $1-e^{c_1 T}$, at least $c_2 T$ steps are performed by $\mathfrak{D}$  during time $T$, and the lemma follows.
	\end{proof}
	From the previous lemma we have the following immediate corollary:
	\begin{corollary}\label{cor:geomdomination}
	Assume~\linktomycom{hyp:H2} holds. Let $C_L$ be a large enough constant depending on $L$ only with $L$ as in \linktomycom{hyp:H2}. Let $C'_L$ be a sufficiently large constant depending on~$L$ only and let $s \ge C'_L$ denote the number of steps performed by $\mathfrak{D}$. For every constant $c'_1 > 0$, there exists $C'_1$ depending only on $L$ and on $c'_1$, such that with probability at least $1-e^{c'_1 s}$, the time elapsed for the $s$ steps is at most~$C'_1 s$.
	\end{corollary}
	
When $p(\{\auxy_s\}_{s\geq 0})<1/2$, the process $\{\auxy_s\}_{s\geq 0}$ is subject to a drift towards the boundary $\auxY$, 
so intuitively the process $\mathfrak{D}$ will also hit $\auxY$ in a number of steps 
	that is proportional to $\auxY-\auxy_0^{\mathfrak{D}}$ (the initial distance of the process $\mathfrak{D}$ to the boundary
	$\auxY$) and dependent on the intensity of the drift. Moreover, one should expect that the probability of the said number of steps exceeding its expected value by much decreases rapidly. The next result is a quantitative version of this intuition.
	
	\begin{lemma}\label{lem:auxdfrac}
	Let $\{\auxy_s\}_{s\geq 0}$ be a diffusion process on $[0,\auxY]$ and
	suppose that $\mathfrak{D}$ is such that $\auxy_0^{\mathfrak{D}} \ge m$ for some $m \in \mathbb{N}$. Let $p:=p(\{\auxy_s\}_{s\ge 0})$ and assume \linktomycom{hyp:H1} holds. Then, for any $0< \delta <\frac{1}{2p_{}}-1$ and all $C\geq (1-2(1+\delta)p_{})^{-1}$, with probability at least $1-\exp(-\frac13\delta^2p_{}\tau)$, the process
	$\mathfrak{D}$ hits $\auxY$ in at most $\tau=\tau(m):=\lfloor C(\auxY-m)\rfloor$ steps.
	\end{lemma}
	\begin{proof}
	Denote by $U$ the random variable counting the number of time steps up to step $\tau$ where $\mathfrak{D}$ decreases its value (that is, jumps from some integer $m'$ to $m'-1$). 
	If the process~$\mathfrak{D}$ does not hit $\auxY$ during $\tau$ steps, then it 
	must have decreased for $U$ steps and increased during $\tau-U$ steps 
	(so $\auxy^{\mathfrak{D}}_\tau=\tau-2U+\auxy^{\mathfrak{D}}_0$), and moreover $\auxy^{\mathfrak{D}}_\tau$
	would have to be smaller than~$\auxY$.
	Thus, it will suffice to bound from above the probability that $\tau-2U=\auxy^{\mathfrak{D}}_\tau-\auxy^{\mathfrak{D}}_0< \auxY-m$.
	Since $\EE(U)=p_{}\tau$ and by Chernoff bounds (see~\cite[Theorem 4.4(2)]{mu17}), for any $0 <\delta < 1$,
    \[
    \PP(U \ge (1+\delta)\EE(U)) \le e^{-\frac13\delta^2\EE(U)},
    \]
    the claim follows observing that $\frac12(C-1)(\auxY-m)\geq (1+\delta)\EE(U)$ (by hypothesis on $\delta$ and~$C$) and $\tau-2U<\auxY-m$ if and only if $U > \tfrac12(C-1)(\auxY-m)$.
	\end{proof}

	Next, we show that, with high probability over the choice of the starting position, the boundary $\auxY$ is hit quickly:

\begin{lemma}\label{lem:coupling}
	Let $\{\auxy_s\}_{s\geq 0}$ be a diffusion process on 
	$[0,\auxY]$ with $\pi(\cdot)$ its stationary distribution.
	Assume \linktomycom{hyp:H1} and \linktomycom{hyp:H2} hold.
	Then, there is a 
	$\widehat{C} >0$ such that if $0<k\le \auxY$, then
	the original process has not hit $\auxY$
	by time $\widehat{C}\log(1/\pi((0,k]))$ with probability at most $3\pi([0,k])$.
\end{lemma}
\begin{proof}
Define the event $\mathcal{E}_1:=\{y_0\in (0,k]\}$ and observe that
\begin{equation*}\label{mixed:lower:E1Bnd}
\PP_{\pi}(\mathcal{E}_1) = \pi((0, k]).
\end{equation*}
Next, fix $0 <\delta \leq 1$, $p$ as before, and let $C':=C'(\delta)\geq 3/(\delta^2p)$ 
be a constant which is at least as large as the 
constant $C$ of Lemma~\ref{lem:auxdfrac} (this last result is applicable because~\linktomycom{hyp:H1} holds).
Define $\tau:=\lceil C'\log(1/\pi((0,k]))\rceil$ and
let~$\mathcal{E}_2$ be the event that the process $\mathfrak{D}$ does not hit $\auxY$ during $\tau$ steps.
Conditioned on $\overline{\mathcal{E}}_1$ (so $y_0^{\mathfrak{D}}\geq k$ since $y_0^{\mathfrak{D}}>y_0-1>k-1$), by our choice of~$C'$ and Lemma~\ref{lem:auxdfrac} with $m:=k$,
\begin{equation*}\label{mixed:lower:E2Bnd}
\PP_{\pi}(\mathcal{E}_2 \mid \overline{\mathcal{E}}_1) 
\leq \exp\big({-}\tfrac13\delta^2p_{}\tau\big) 
\leq \pi((0,k]).
\end{equation*}
Choosing $C'$ sufficiently large, by Corollary~\ref{cor:geomdomination}, applied with $c'_1:=1$, there exists $C'_1$ large enough, so that with probability at least $1-e^{-\tau}$, the time elapsed for $\tau$ steps is at most $C'_1 \tau$.

Let $\mathcal{E}_3$ be the event that the original process has not hit $\auxY$ by time 
$C'_1\tau$. Since $\delta,p_{}\leq 1$, by definition of $C'$, we have $C'>1$, so it follows by definition of $\tau$ that $\tau\geq \log(1/\pi((0,k]))$, and thus our preceding discussion establishes that 
\[
\P_{\pi}(\mathcal{E}_3 \mid \overline{\mathcal{E}}_1, \overline{\mathcal{E}}_2 ) \le \pi((0,k]).
\]
The lemma follows by a union bound over events $\mathcal{E}_1$, $\mathcal{E}_2$, $\mathcal{E}_3$, and by setting $\widehat{C}:=C'C'_1$.
\end{proof}

	In order to establish our main result we next state the following fact: 
	\begin{fact}\label{fact:H1}
	For all integers $1 < m < \auxY$, $q \pi_{m-1} \le p \pi_{m}$.
	\end{fact}
	\begin{proof}
Let $p'_1,...,p'_{\auxY}$ be defined as in this subsection's introduction and let $q'_i:=1-p'_i$, and recall that $p'_1=0$ by assumption.
We claim that $q'_{m-1}\pi_{m-1}=p'_m\pi_m$ for $1< m \le \auxY$.
Indeed, by~\eqref{radial:def:pim}, we have that $\pi_1=p'_1\pi_1+p'_2\pi_2$, so $q'_1\pi_1=p'_2\pi_2$. Assume the claim holds for $1\le m<\auxY$. By~\eqref{radial:def:pim} and the inductive hypothesis, we have $p'_{m+1}\pi_{m+1}=\pi_m-q'_{m-1}\pi_{m-1}=\pi_m-p'_{m}\pi_{m}=q'_{m}\pi_{m}$, which concludes the inductive proof of our claim.

By definition of  $p$, we have $p'_m\le p$ and $q'_m\ge q$ for all $1\le m < \auxY$. Hence, by the preceding paragraph's claim, for $1<m<\auxY$ we get $q \pi_{m-1} \le q'_{m-1}\pi_{m-1}=p'_{m}\pi_{m} \le p \pi_m$, and the fact follows.
	\end{proof}
We now use the previous lemmata to show this section's main result:
\begin{proposition}\label{prop:coupling}
Let $\{\auxy_s\}_{s\geq 0}$ be a diffusion process on $[0,\auxY]$ with stationary distribution~$\pi(\cdot)$. Assume~\linktomycom{hyp:H1}, \linktomycom{hyp:H2} and~\linktomycom{hyp:H3} hold.
Then, there are constants $\widetilde{C}>0$, $\widetilde{c}\in (0,1)$, $\widetilde{\eta}> 0$ all independent of $\auxY$ such that for all 
 $\widetilde{C} \le k\leq \auxY$ and $\ss\geq 1/\pi((0,k])$ we have  
\[
\PP_\pi\big(I_k\geq \widetilde{c} \ss\pi((0,k])\big) \geq \widetilde{\eta},
	\qquad\text{where $I_k := \int_0^\ss {\bf 1}_{(0,k]}(\auxy_s)ds$.}
\]
\end{proposition}
\begin{proof}
First, we establish that there is an integer $\Delta>0$ depending on $p$ alone such that
\begin{equation}\label{lower:eqn:existDelta}
\pi((0,k-\Delta])\leq\frac{1}{12}\pi((0,k]) \qquad \text{for all $1\leq k\leq\auxY$.}
\end{equation}
The claim is a consequence of~\linktomycom{hyp:H1} and Fact~\ref{fact:H1}: 
Indeed, let $p, q$ be as therein. For any~$k$, summing 
over $m$ with $1<m\le k$ we get $q\pi((0,k-1])\leq p\pi((0,k])$ and thus
$\pi((0,k-m])\leq (p/q)^m\pi((0,k])$ for any $m\in\NN$.
By~\linktomycom{hyp:H1} we know that $p/q$ is bounded away from $1$ by a constant independent of $\auxY$. 
Taking $m$ large enough so that $(p/q)^m\leq 1/12$ the claim follows setting $\Delta:=m$.

Next, we will split the time period $[0,\ss]$ into time intervals of one unit (throwing away the eventual remaining non-integer part): for the $i$-th time interval let $X_i$ be the indicator random variable being~$1$ if at time instant $i-1$ (that is, at the beginning of the $i$-th time interval) the process $\{y_s\}_{s \ge 0}$ is within $(0,k]$, and $0$ otherwise. Since $\pi$ is stationary, $\P_{\pi}(X_i=1)=\pi((0,k])$ for any $i$. Thus, setting $X:=\sum_{i=1}^{\lfloor \ss \rfloor} X_i$, we have $(\ss-1)\P_{\pi}(X_1=1)<\EE_{\pi}(X)\leq \ss\P_{\pi}(X_1=1) \le 2\EE_{\pi}(X)$.
Since~\linktomycom{hyp:H3} holds, for $\widetilde{C}$ the constant therein, if $X_i=1$ and $k\ge \widetilde{C}$, then with constant probability the process $\{y_s\}_{s \ge 0}$ stays throughout the entire $i$-th time interval within $(0,k]$.
So by our previous discussion, to establish the desired
result, it suffices to show that for some $\xi>0$  the probability that $X$ is smaller than $\xi\EE_{\pi}(X)$ is at most a constant strictly less than $1$. We rely on Chebyshev's inequality to do so and thus need to bound from above the variance of $X$. This explains why we now undertake the determination of an upper bound for $\EE_{\pi}(X^2)$.
Since there are at most $P_d \le \ss$ pairs of random variables $X_i$ and~$X_j$ such that $|i-j|=d$,
$$
\EE_{\pi}(X^2)=\sum_{i,j} \P_{\pi}(X_i=1,X_j=1) \leq \sum_{d=0}^{\lfloor \ss \rfloor} P_d\P_{\pi}(X_1=1, X_{1+d}=1). 
$$
 

In order to bound $\P_{\pi}(X_1=1, X_{1+d}=1)$ we construct two processes $\{\widetilde{\auxy}_s\}_{s\geq0}$ and $\{\widehat{\auxy}_s\}_{s\geq0}$ as follows: Initially, $\widetilde{\auxy}_0$ is sampled with 
    distribution $\frac{\pi(\cdot\cap(0,k])}{\pi((0,k])}$.
    Also, with probability $\pi((0,k])$ we let $\widehat{\auxy}_0=\widetilde{\auxy}_0$, and otherwise we sample $\widehat{\auxy}_0$ with distribution  $\frac{\pi(\cdot\cap(k,\auxY])}{\pi((k,\auxY])}$.
    Then, both processes evolve independently according to the same law as the original process until they first meet. Afterward, they still move as the original process but coupled together.
It follows directly that $\{\widehat{\auxy}_s\}_{s \ge 0}$ is a copy of the original process, while $\{\widetilde{\auxy}_s\}_{s \ge 0}$ is a version of the original process conditioned to start within $(0,k]$. Let $\widetilde{T}_{\auxY}$ be the first time that $\{\widetilde{\auxy}_s\}_{s \ge 0}$ hits $\auxY$, and define $\widetilde{X}_i$ and~$\widehat{X}_i$ analogously to the variables $X_i$ with $\widetilde{y}_s$ and $\widehat{y}_s$ playing the role of $y_s$, respectively. Since $\P_{\pi}(X_1=1, X_{1+d}=1)=\P(\widetilde{X}_{1+d}=1)\pi((0,k])$,
it follows that 
\[\P_{\pi}(X_1=1, X_{1+d}=1)\,=\,\P(\widetilde{X}_{1+d}=1, \widetilde{T}_Y
\leq d)\pi((0,k])+\P(X_1=1, X_{1+d}=1,\widetilde{T}_Y>d).
\]
For the first term on the right-hand side above observe that $\widetilde{\auxy}_0\leq \widehat{\auxy}_0$, hence on the event~$\widetilde{T}_{\auxY}\leq d$ both processes have met before time $d$, and so by definition $\widetilde{X}_{1+d}=\widehat{X}_{1+d}$, yielding
\[
\P(\widetilde{X}_{1+d}=1, \widetilde{T}_Y\leq d)\pi((0,k])\leq \P(\widehat{X}_{1+d}=1)\pi((0,k])=\big(\P_{\pi}(X_1=1)\big)^2.
\]
It remains then to bound the term $\P(X_1=1, X_{1+d}=1,\widetilde{T}_Y>d)$. To do so recall that~$\{\widetilde{\auxy}_s\}_{s \ge 0}$ is a version conditioned to start at $(0,k]$ so calling $T_Y$ the first time $\{\auxy_s\}_{s\ge0}$ hits $Y$, we have
\[\P(X_1=1, X_{1+d}=1,\widetilde{T}_Y>d)=\P_{\pi}(X_{1}=1, X_{1+d}=1, T_Y
>d).\]
First, recall that by Lemma~\ref{geometricdominationlemma}, if $d\ge C_L$ for some large enough constant depending on~$L$ only, for every $c_1 > 0$ there exists $c_2 > 0$, depending only on $c_1$ and on $L$, so that the process~$\mathfrak{D}$ performs at least $c_2 d$ steps during the time interval $[0, d]$  with probability at least~$1-e^{-c_1d}$. Note that~$c_1$ and $c_2$ do not depend on $\widetilde{C}$. Let $\mathcal{E}$ be the event that $\mathfrak{D}$ performed at least $c_2 d$ steps up to time $d$. Now, 
\begin{align*}
    & \P_{\pi}(X_{1}=1, X_{1+d}=1, T_Y
>d) \\ & \qquad \le  \P_{\pi}(X_{1}=1, X_{1+d}=1, T_Y
>d, \mathcal{E}) + \P_{\pi}(X_1=1, \overline{\mathcal{E}}) \\ \qquad
     & \qquad \le \P_{\pi}(X_{1}=1, X_{1+d}=1, T_Y
>d, \mathcal{E}) + e^{-c_1 d}\P_{\pi}(X_1=1).
\end{align*}
 Next, recall that as long as the original process does not hit the boundary, by Observation~\ref{obs:domination}, the auxiliary process satisfies $y_s^{\mathfrak{D}} \le y_s$ for every instant $s$. Formally,
$$ \P_{\pi}(X_{1}=1, X_{1+d}=1, T_Y
>d) \le \P(\auxy_0^{\mathfrak{D}} \le k, \auxy_{d}^{\mathfrak{D}} \le k)=\P_{\pi}(X_1=1)\P(\auxy_d^{\mathfrak{D}} \le k \mid \auxy_0^{\mathfrak{D}} \le k),   
$$
and also 
\begin{align*}
\P_{\pi}(X_{1}=1, X_{1+d}=1, T_Y
>d, \mathcal{E}) &\le
\P(\auxy_0^{\mathfrak{D}} \le k, \auxy_{d}^{\mathfrak{D}} \le k, \mathcal{E}) \\
& \le 
\P_{\pi}(X_1=1)\P(\auxy_d^{\mathfrak{D}} \le k \mid \mathcal{E}, \auxy_0^{\mathfrak{D}} \le k),
\end{align*}
and our goal is thus to bound $\P(\auxy_d^{\mathfrak{D}} \le k \mid \mathcal{E}, \auxy_0^{\mathfrak{D}} \le k)$.
Recalling that $p_{} < 1/2$ (by hypothesis) is the probability that~$\mathfrak{D}$ makes a decreasing step, and $q:=1-p$, we have for some large enough constant $C_1 > 0$ 
$$
 \P(\auxy_{d}^{\mathfrak{D}} = k \mid \auxy_0^{\mathfrak{D}}=k, \mathcal{E}) \le \sum_{\ell: 2\ell \ge c_2 d} \binom{2\ell}{\ell}(p_{}q)^{\ell}\le \sum_{\ell:2\ell \ge c_2 d} (4p_{}q)^{\ell} \le C_1(4p_{}q)^{c_2 d/2}.
$$
The same upper bound holds for  $\P(\auxy_{d}^{\mathfrak{D}} = k{-}1 \mid \auxy_0^{\mathfrak{D}}=k, \mathcal{E})$. For  $\P(\auxy_{d}^{\mathfrak{D}} = k{-}2 \mid \auxy_0^{\mathfrak{D}}=k, \mathcal{E})$ and  $\P(\auxy_{d}^{\mathfrak{D}} = k{-}3 \mid \auxy_0^{\mathfrak{D}}=k, \mathcal{E})$ note that $\mathfrak{D}$ has to make one more decreasing step and one less increasing step, yielding an upper bound of 
$C_1(4p_{}q)^{c_2 d/2} \frac{p_{}}{q}$. In the same way, 
$\P(\auxy_{d}^{\mathfrak{D}} = k{-}2i \mid \auxy_0^{\mathfrak{D}}=k, \mathcal{E}) \le  C_1(4p_{}q)^{c_2 d/2} (\frac{p_{}}{q})^i$, so that for $C_2:=C_2(C_1) > 0$ large enough,
$\P(\auxy_{d}^{\mathfrak{D}} \le k \mid \auxy_0^{\mathfrak{D}}=k, \mathcal{E}) \le C_2 (4p_{}q)^{c_2 d/2}$.
We can get the same upper bound also for  $\P(\auxy_{d}^{\mathfrak{D}} \le k \mid \auxy_0^{\mathfrak{D}}=k{-}1, \mathcal{E})$. 
Moreover, 
we have
\begin{align*}
\P(\auxy_{d}^{\mathfrak{D}} = k \mid \auxy_0^{\mathfrak{D}}=k{-}2, \mathcal{E}) &\le \sum_{\ell: 2\ell \ge c_2 d} 
\binom{2\ell}{\ell-1}p_{}^{\ell{-}1}q^{\ell+1} 
\le \frac{q}{p} \sum_{\ell: 2\ell \ge c_2 d} \binom{2\ell}{\ell}(p_{}q)^{\ell} 
\le \frac{q}{p}C_1(4p_{}q)^{\frac{c_2 d}{2}}.
\end{align*}
By the argument from before,  we get
 $\P(\auxy_{d}^{\mathfrak{D}} \le k \mid \auxy_0^{\mathfrak{D}}=k{-}2, \mathcal{E}) \le (q/p)C_2 (4p_{}q)^{\frac{c_2 d}{2}}$
 and iterating the argument, also obtain
 $
 \P(\auxy_{d}^{\mathfrak{D}} \le k \mid \auxy_0^{\mathfrak{D}} = k{-}2i, \mathcal{E}) \le (q/p)^i C_2 (4p_{}q)^{\frac{c_2 d}{2}}.
 $
 Hence, if~$k$ is even, 
 \begin{align*}
 & \P(\auxy_{d}^{\mathfrak{D}} \le k, \auxy_0^{\mathfrak{D}} \le k \mid \mathcal{E}) \\
 & \quad = \sum_{j=0}^{k}\PP(\auxy_d^{\mathfrak{D}} \le k \mid \auxy_0^{\mathfrak{D}}=k{-}j, \mathcal{E})\PP(\auxy_0^{\mathfrak{D}}=k{-}j)
 \\
  & \quad \le \PP(\auxy_d^{\mathfrak{D}} \le k \mid \auxy_0^{\mathfrak{D}}=0, \mathcal{E})\PP(\auxy_0^{\mathfrak{D}}=0) 
  +\!\!\sum_{i=0}^{k/2-1}\PP(\auxy_d^{\mathfrak{D}} \le k \mid \auxy_0^{\mathfrak{D}}=k{-}2i, \mathcal{E})\PP(\auxy_0^{\mathfrak{D}} \in \{k{-}2i, k{-}2i{-}1\})
 \\
& \quad \leq (q/p)^{k/2}C_2(4pq)^{\frac{c_2d}{2}}\pi((0,1])+2\sum_{i=0}^{k/2-1}\Big(\frac{q}{p}\Big)^i C_2 (4p_{}q)^{\frac{c_2 d}{2}}\pi((k{-}2i{-}2, k{-}2i])  \\
& \quad \le 4C_3(4p_{}q)^{\frac{c_2 d}{2}}\pi((0,k]),
 \end{align*}
 where for the secondlast inequality we used Fact~\ref{fact:H1}.
If $k$ is odd, then take $\{0,...,\frac12(k-1)\}$ as the range over which the second index of the summation (the one over $i$ above); the remaining bounds are still valid, and 
we obtain the same bound as for $k$ even.
Thus, recalling that $P_d\leq\ss$ denotes the number of pairs of random variables $X_i$ and $X_j$ for which $|i-j|=d$, and recalling that $\overline{\mathcal{E}}$ holds with probability at most $e^{-c_1 d}$ (in which case the joint probability is only bounded by the probability of $X_1=1$), and adding also the joint probability $(\P(X_1=1))^2$ in case $\widetilde{T}_Y \le d$, we get 
\begin{align*}
\EE_{\pi}(X^2) &\le \ss C_L \P_{\pi}(X_1=1)+ \sum_{d \ge C_L}^{\lfloor \ss \rfloor} P_d \Big( \big(4C_3(4p_{}q)^{\frac{c_2 d}{2}}+e^{-c_1d}\big)\P_{\pi}(X_1=1)+\big(\P_{\pi}(X_1=1)\big)^2\Big)
\\
& \le C_L\EE_{\pi}(X)+2C^{*}\EE_{\pi}(X)+(\EE_{\pi}(X))^2.
\end{align*}
where for the second inequality we used that $P_d \le \ss$ and then that $\ss \P_{\pi}(X_1=1)\le 2\EE_{\pi}(X)$, and that $\sum_{d\ge 0} (4C_3(4p_{}(1-p_{}))^{c_2 d/2} +e^{-c_1d})\le C^{*}$
for some $C^{*}$ depending on $c_1$ and $c_2$ only. 
Thus, since $\EE_{\pi}(X) \ge C_4$ for some constant $C_4$ large by our hypothesis $\ss\ge 1/\pi((0,k])$, we conclude that  $\EE_{\pi}(X^2)\le \frac43 (\EE_{\pi}(X))^2$, so
$\Var_{\pi}(X)\le \frac13 (\EE_{\pi}(X))^2$.
Thus, by Chebyshev's inequality,
$$
\P(|X-\EE_{\pi}(X)| \ge \tfrac34 \EE_{\pi}(X)) \le \Var_{\pi}(X)/(\tfrac34\EE_{\pi}(X))^2 \le \tfrac{16}{27},
$$
and the statement follows taking $\widetilde{c}:=\frac14$ and $\widetilde{\eta}:=\frac{11}{27}$.
\end{proof}


To conclude this section, we show that the process we studied
in Sections~\ref{sec:upper_rad} and~\ref{sec:lower_rad}
satisfies~\linktomycom{hyp:H1} through~\linktomycom{hyp:H3} and hence Proposition~\ref{prop:coupling}  holds for the said process. Reaching such conclusion was the motivation for this section, and it provides a result on which the analysis of the combined radial and angular processes, which we will treat in the next section, relies.
\begin{corollary}\label{radial:cor:coupling}
Let $\{\auxy_s\}_{s\geq 0}$ be the diffusion process on $(0,\auxY]$ with
generator $\Delta_h$ and a reflecting barrier at $\auxY$. 
Then, there are constants $\widetilde{C}>0$, $\widetilde{c},\widetilde{\eta}\in (0,1)$, such that for all 
$\widetilde{C}\le k\le \auxY$ and $\ss\geq 1/\pi((0,k])$ we have  
\[
\PP_\pi\big(I_k\geq \widetilde{c} \ss\pi((0,k])\big) \geq \widetilde{\eta},
	\qquad\text{where $I_k := \int_0^\ss {\bf 1}_{(0,k]}(\auxy_s)ds$.}
\]
\end{corollary}
\begin{proof}
It will be enough to verify that~\linktomycom{hyp:H1}-\linktomycom{hyp:H3} hold and apply
Proposition~\ref{prop:coupling} with $\widetilde{C}$ as therein and satisfying our just stated condition.

Let $C_{\alpha,\auxY}:=(\cosh(\alpha\auxY)-1)^{-1}$. By definition $\pi_m:=\pi((m{-}1,m])=C_{\alpha,\auxY}(\cosh(\alpha m)-\cosh(\alpha(m{-}1)))$.
Using $\cosh x-\cosh y=2\sinh(\frac12(x+y))\sinh(\frac12(x-y))$, we get
\begin{equation}\label{radial:eqn:piSinh}
\pi_{m}=2C_{\alpha,\auxY}\sinh(\alpha(m-\tfrac12))\sinh(\tfrac12\alpha) \qquad\text{for all $1\le m\le \auxY$}.
\end{equation}
To show that $p$ is bounded away from $0$ as required by~\linktomycom{hyp:H1}, it suffices
to observe that by Fact~\ref{fact:H1}, 
$p/q\ge p_2/q_1=\pi_1/\pi_2=\sinh(\frac12\alpha)/\sinh(\frac32\alpha)$. To establish that $p$ is also bounded from above by a constant strictly smaller than $1/2$, recall the definition of $p'_m$ from this subsection's introduction and let $q'_m:=1-p'_m$. Note that by definition of $q'_i$, we have $q'_{i}\pi_{i}=\pi_{i}-p'_{i}\pi_{i}$, so from the proof of Fact~\ref{fact:H1}, $p'_{m}\pi_{m}=q'_{m-1}\pi_{m-1}=\pi_{m-1}-p'_{m-2}\pi_{m-2}
=\pi_{m-1}-\pi_{m-2}+p'_{m-3}\pi_{m-3}$. Hence, $p'_m\pi_m\le\pi_{m-1}-\pi_{m-2}+\pi_{m-3}$ for $3<m\le\auxY$. It is easy to see that the right-hand side is, as a function in $\alpha$, maximized for $\alpha=1/2$. Moreover, for every $\alpha$, it is increasing in $m$, we see it is bounded from above by $0.4618$. By direct calculation, $p'_2=\pi_1/\pi_2=\sinh(\alpha/2) / \sinh(3\alpha/2) \le 0.3072$ and $p'_3=(1-p'_2)\pi_2/\pi_3 \le 0.3557$, where we used that both expressions for $p'_2$ and $p'_3$ are decreasing as functions in $\alpha$. 
 We conclude that $\sup_{1\le m \le \auxY} p'_m\leq 0.4618$ for all $1<m\le\auxY$. To conclude that $p < 1/2$, we also need to bound $\widehat{p}_m$, for $1 \le m < \auxY$: in order to do so, observe that the diffusion process defined by~\eqref{truegenerator} can be coupled with a process of constant drift towards $m+1$ so that the real process is always to the right of the process of constant drift; for a process with constant drift towards $m+1$ starting at $m$ it is well known that the probability to reach $m-1$ before $m+1$ is at most $1/2-\varepsilon$ for some $\varepsilon > 0$ (see~\cite{Borodin2002}, formula 3.0.4).
 Hence $\widehat{p}_m \le \frac12 - \varepsilon$ for any $1 \le m < \auxY$, and hence $p < 1/2$, as desired.

Next, we establish~\linktomycom{hyp:H2}. Recalling that $T_{(\auxY-1)\pm 1}$ is the random variable counting the maximal time the process spends in the interval $(\auxY-2,\auxY]$ starting at $\auxY-1$; by Part~\eqref{radial:itm:phi2} of Lemma~\ref{lemmaradial} applied with $\yabs_0:=\auxY-2$, for any $\auxy \in (\auxY-2,\auxY]$, we get $\EE_{\auxy}(T_{\yabs_0}) \le \frac{4}{\alpha^2}e^{2\alpha}$, and since $\sup_{\auxY-2 \le y \le \auxY} \EE_y (T_{y \pm 1}) \le \sup_{\auxY-2 < \auxy \le \auxY} \EE_{\auxy}(T_{\yabs_0})$, conditon~\linktomycom{hyp:H2} is satisfied for $y > \auxY-2$. In order to bound $\EE_y(T_{y\pm 1})$ for $y<\auxY-1$, note that the time to exit from such an interval can only increase when imposing a reflecting barrier at $y+1$, and we can then apply  Part~\eqref{radial:itm:phi2} of Lemma~\ref{lemmaradial} to the process with this reflecting barrier at $y$, applied with~$\yabs_0:=y-1$. Hence, for any starting position $y$, the expected time remaining in the interval~$(y-1,y+1]$ is at most a constant $L$, and~\linktomycom{hyp:H2} is satisfied.

Finally, to check that~\linktomycom{hyp:H3} holds, simply observe that 
for all $\widetilde{C}\le k\le \auxY$ (where $\widetilde{C}$ is at least a constant strictly greater than $1$), the process $\{\auxy_s\}_{s\ge 0}$ dominates a process of constant drift towards $\auxY$, which wherever conditioned on $\auxy_s\le k$ has a constant probability of staying within $(0,k]$ during the unit length time interval $[s,s+1]$, thus establishing the claim and concluding the proof of the corollary.
\end{proof}

\section{General case: strategies for detection}\label{sec:mix}
%

In this section we define the set $\dt$ similarly as in the radial section, without the restriction that $|\theta_0| \le \tfrac{\pi}{2}$, since also points that belong to the half disk of $B_O(R)$ opposite to $Q$ are starting positions from which particles can detect (mainly thanks to the angular movement component). As in previous sections, on a high level, this set $\dt$ is chosen in such a way that at least a constant part of the measure of the overall detection probability comes from~$\dt$. The precise definition of $\dt$ (given in the two following theorems) depends on the fact whether $\ss$ is small or large. For $\ss$ large, we can moreover provide uniform lower bounds of detection for every element in $\dt$ and matching uniform upper bounds of detection for every element in $\ndt$. To make this intuition precise, we have the following two theorems, depending on whether $\ss$ is small or large: 

\begin{theorem}\label{thm:mixedSmall} 
Assume $\beta \le 1/2$,  $\ss=\Omega((e^{\beta R}/n)^2)\cap O(1)$. Let
\begin{equation}\label{mix:lower:defD}
\dt := \{ x_0=(r_0,\theta_0)\in B_O(R) \mid |\theta_0|\leq \phi(r_0) +\sqrt{\ss}e^{-\beta r_0} \}. 
\end{equation}
Then 
\[
\int_{\dt} \PP_{x_0}(T_{det}\leq\ss)d\mu(x_0)=\Theta(\mu(\dt))=\Theta(n e^{-\beta R}\sqrt{\ss})
\]
and 
\[
\int_{\ndt} \PP_{x_0}(T_{det}\leq\ss)d\mu(x_0)=O(\mu(\dt)).
\]
\end{theorem}
\begin{theorem}\label{thm:mixedLarge}
Let $\KA > 1$. Define then $\KR:=\KA$ if $\alpha < 2\beta$ and $\KR:=e^{\KA^2}$ if $\alpha \ge 2\beta$. Let 
\[
\dt(\KA):=\{x_0=(r_0,\theta_0)\in B_O(R) \mid |\theta_0|\leq\phi(R)+\KA\phis \vee |\theta_0|\leq \phi(r_0)+\KR e^{-(\beta\wedge\frac12) r_0}\}, 
\]
where 
\[\phis:=\begin{cases}
(\ss^{\frac{1}{\alpha}}/e^R)^{\beta\wedge\frac{1}{2}}, 
 & \text{if $\alpha<2\beta$},\\[2pt]
e^{-\beta R}\sqrt{\ss \log\ss}, & \text{if $\alpha=2\beta$},\\[2pt]
e^{-\beta R}\sqrt{\ss},&\text{if $\alpha>2\beta$}.
\end{cases}
\]
Furthermore, define $$\mathfrak{Z} := \begin{cases}
e^{\alpha R}, &
\text{if $\alpha < 2\beta$,} \\
e^{\alpha R}/(\alpha R), &
\text{if $\alpha=2\beta$,} \\
e^{2\beta R}, & \text{if $\alpha > 2\beta$.} 
\end{cases} $$
Then, for $\ss=\Omega(1)\cap O(\mathfrak{Z})$ we have 
\[
\int_{\ndt(\KA)} \PP_{x_0}(T_{det}\leq\ss)d\mu(x_0)=O(\mu(\dt(\KA))).
\]
Furthermore, under the additional assumption $\ss=\omega(1)$ we have
$$
\inf_{x_0\in\dt(\KA)}\P_{x_0}(T_{det}\leq \ss) =
\begin{cases}
e^{-O(\KA^2)},
& \text{ if $\alpha\geq2\beta$,} \\[2pt]
\Omega(\KA^{-\alpha/(\beta\wedge\frac12)}),
& \text{ if $\alpha<2\beta$,} 
\end{cases}
$$
and
\[
\sup_{x_0\in\ndt(\KA)}\P_{x_0}(T_{det}\leq \ss) =
\begin{cases}
e^{-\Omega(\KA^2)},
& \text{ if $\alpha\geq2\beta$,} \\[2pt]
O(\KA^{-\alpha/(\beta\wedge\frac12)}),
& \text{ if $\alpha<2\beta$.} 
\end{cases}
\]
\end{theorem}

\begin{figure}
    \centering
    
\begin{tikzpicture}[x=1cm,y=1cm,scale=0.65,
     decoration={markings,
       mark=at position 1 with {\arrow[scale=1.5,black]{latex}};
      }]
      
      \def\c{12}
      \def\barr{4}
      \def\radp{6.8}
      \def\radb{2.95}
      \def\delta{1.5}
      \def\angb{45}
      \def\phip{200}
      \def\phiro{266.5}
      \def\phiss{185}
      \def\angabs{-71.4};
      \node[inner sep=0] (O) at (\c,\c) {};
      \node[inner sep=0] (P) at (2,4) {};
      \node[inner sep=0] (Q) at (\c,0) {};
  \draw[fill=gray!20] (8.204cm,0.616cm) --(10.262cm,6.803cm) -- (10.178cm,7.003cm) -- (10.092cm,7.208cm) -- (10.005cm,7.419cm) -- (9.918cm,7.636cm) -- (9.831cm,7.859cm) -- (9.747cm,8.090cm) -- (9.664cm,8.328cm) -- (9.586cm,8.575cm) -- (9.513cm,8.830cm) -- (9.447cm,9.094cm) -- (9.391cm,9.367cm) -- (9.346cm,9.649cm) -- (9.315cm,9.940cm) -- (9.301cm,10.238cm) -- (9.307cm,10.543cm) -- (9.336cm,10.852cm) -- (9.391cm,11.162cm) -- (9.476cm,11.470cm) -- (9.593cm,11.771cm) -- (9.744cm,12.059cm) -- (9.930cm,12.326cm) -- (10.150cm,12.565cm) -- (10.402cm,12.767cm) -- (10.631cm,12.900cm) -- (10.678cm,12.923cm) -- (10.727cm,12.944cm) -- (10.775cm,12.963cm) -- (10.824cm,12.981cm) -- (10.873cm,12.997cm) -- (10.923cm,13.011cm) -- (10.972cm,13.024cm) -- (11.022cm,13.035cm) -- (11.072cm,13.044cm) -- (11.122cm,13.052cm) -- (11.172cm,13.058cm) -- (11.222cm,13.062cm) -- (11.271cm,13.064cm) -- (11.321cm,13.064cm) -- (11.370cm,13.063cm) -- (11.418cm,13.060cm) -- (11.466cm,13.054cm) -- (11.513cm,13.048cm) -- (11.560cm,13.039cm) -- (11.606cm,13.028cm) -- (11.651cm,13.016cm) -- (11.695cm,13.002cm) -- (11.738cm,12.987cm) -- (11.780cm,12.969cm) -- (11.820cm,12.950cm) -- (11.860cm,12.930cm) -- (11.898cm,12.908cm) -- (11.934cm,12.884cm) -- (11.969cm,12.859cm) -- (12.002cm,12.833cm) -- (11.998cm,12.833cm) -- (12.031cm,12.859cm) -- (12.066cm,12.884cm) -- (12.102cm,12.908cm) -- (12.140cm,12.930cm) -- (12.180cm,12.950cm) -- (12.220cm,12.969cm) -- (12.262cm,12.987cm) -- (12.305cm,13.002cm) -- (12.349cm,13.016cm) -- (12.394cm,13.028cm) -- (12.440cm,13.039cm) -- (12.487cm,13.048cm) -- (12.534cm,13.054cm) -- (12.582cm,13.060cm) -- (12.630cm,13.063cm) -- (12.679cm,13.064cm) -- (12.729cm,13.064cm) -- (12.778cm,13.062cm) -- (12.828cm,13.058cm) -- (12.878cm,13.052cm) -- (12.928cm,13.044cm) -- (12.978cm,13.035cm) -- (13.028cm,13.024cm) -- (13.077cm,13.011cm) -- (13.127cm,12.997cm) -- (13.176cm,12.981cm) -- (13.225cm,12.963cm) -- (13.273cm,12.944cm) -- (13.322cm,12.923cm) -- (13.369cm,12.900cm) -- (13.416cm,12.876cm) -- (13.463cm,12.851cm) -- (13.509cm,12.824cm) -- (13.554cm,12.796cm) -- (13.598cm,12.767cm) -- (13.642cm,12.736cm) -- (13.685cm,12.704cm) -- (13.728cm,12.671cm) -- (13.769cm,12.637cm) -- (13.810cm,12.601cm) -- (13.850cm,12.565cm) -- (13.889cm,12.527cm) -- (13.927cm,12.489cm) -- (13.964cm,12.450cm) -- (14.000cm,12.409cm) -- (14.035cm,12.368cm) -- (14.070cm,12.326cm) -- (14.103cm,12.283cm) -- (14.136cm,12.240cm) -- (14.167cm,12.195cm) -- (14.198cm,12.150cm) -- (14.227cm,12.105cm) -- (14.256cm,12.059cm) -- (14.283cm,12.012cm) -- (14.310cm,11.965cm) -- (14.336cm,11.917cm) -- (14.360cm,11.869cm) -- (14.384cm,11.820cm) -- (14.407cm,11.771cm) -- (14.429cm,11.722cm) -- (14.449cm,11.672cm) -- (14.469cm,11.622cm) -- (14.488cm,11.571cm) -- (14.507cm,11.521cm) -- (14.524cm,11.470cm) -- (14.540cm,11.419cm) -- (14.556cm,11.368cm) -- (14.570cm,11.317cm) -- (14.584cm,11.265cm) -- (14.597cm,11.214cm) -- (14.609cm,11.162cm) -- (14.620cm,11.110cm) -- (14.630cm,11.059cm) -- (14.640cm,11.007cm) -- (14.649cm,10.955cm) -- (14.657cm,10.903cm) -- (14.664cm,10.852cm) -- (14.671cm,10.800cm) -- (14.677cm,10.748cm) -- (14.682cm,10.697cm) -- (14.686cm,10.645cm) -- (14.690cm,10.594cm) -- (14.693cm,10.543cm) -- (14.696cm,10.492cm) -- (14.698cm,10.441cm) -- (14.699cm,10.390cm) -- (14.700cm,10.339cm) -- (14.700cm,10.288cm) -- (14.699cm,10.238cm) -- (14.698cm,10.188cm) -- (14.697cm,10.138cm) -- (14.695cm,10.088cm) -- (14.692cm,10.038cm) -- (14.689cm,9.989cm) -- (14.685cm,9.940cm) -- (14.681cm,9.891cm) -- (14.677cm,9.842cm) -- (14.672cm,9.793cm) -- (14.666cm,9.745cm) -- (14.661cm,9.697cm) -- (14.654cm,9.649cm) -- (14.648cm,9.601cm) -- (14.641cm,9.554cm) -- (14.634cm,9.507cm) -- (14.626cm,9.460cm) -- (14.618cm,9.413cm) -- (14.609cm,9.367cm) -- (14.601cm,9.321cm) -- (14.592cm,9.275cm) -- (14.582cm,9.229cm) -- (14.573cm,9.184cm) -- (14.553cm,9.094cm) -- (14.487cm,8.830cm) -- (14.414cm,8.575cm) -- (14.336cm,8.328cm) -- (14.253cm,8.090cm) -- (14.169cm,7.859cm) -- (14.082cm,7.636cm) -- (13.995cm,7.419cm) -- (13.908cm,7.208cm) -- (13.822cm,7.003cm) -- (13.738cm,6.803cm) -- (15.845cm,0.633cm) --(15.475cm,0.514cm) --(15.101cm,0.408cm) --(14.724cm,0.313cm) --(14.344cm,0.231cm) --(13.961cm,0.161cm) --(13.577cm,0.104cm) --(13.190cm,0.059cm) --(12.803cm,0.027cm) --(12.415cm,0.007cm) --(12.026cm,0.000cm) --(11.637cm,0.005cm) --(11.249cm,0.024cm) --(10.861cm,0.054cm) --(10.475cm,0.097cm) --(10.090cm,0.153cm) --(9.707cm,0.221cm) --(9.327cm,0.302cm) --(8.949cm,0.394cm) --(8.575cm,0.499cm) --(8.204cm,0.616cm) --cycle;

\draw[fill=gray!40] (11.941cm,0.000cm) -- (11.902cm,1.200cm) -- (11.842cm,2.401cm) -- (11.748cm,3.604cm) -- (11.607cm,4.811cm) -- (11.404cm,6.030cm) -- (11.136cm,7.278cm) -- (10.842cm,8.591cm) -- (10.798cm,8.820cm) -- (10.758cm,9.051cm) -- (10.725cm,9.285cm) -- (10.698cm,9.521cm) -- (10.681cm,9.760cm) -- (10.675cm,9.999cm) -- (10.681cm,10.239cm) -- (10.704cm,10.477cm) -- (10.744cm,10.711cm) -- (10.804cm,10.938cm) -- (10.885cm,11.154cm) -- (10.988cm,11.356cm) -- (11.113cm,11.538cm) -- (11.260cm,11.696cm) -- (11.426cm,11.825cm) -- (11.608cm,11.921cm) -- (11.801cm,11.980cm) -- (12.000cm,12.000cm) --(12.199cm,11.980cm) -- (12.392cm,11.921cm) -- (12.574cm,11.825cm) -- (12.740cm,11.696cm) -- (12.887cm,11.538cm) -- (13.012cm,11.356cm) -- (13.115cm,11.154cm) -- (13.196cm,10.938cm) -- (13.256cm,10.711cm) -- (13.296cm,10.477cm) -- (13.319cm,10.239cm) -- (13.325cm,9.999cm) -- (13.319cm,9.760cm) -- (13.302cm,9.521cm) -- (13.275cm,9.285cm) -- (13.242cm,9.051cm) -- (13.202cm,8.820cm) -- (13.158cm,8.591cm) -- (13.112cm,8.366cm) -- (13.063cm,8.144cm) -- (13.013cm,7.924cm) -- (12.963cm,7.707cm) -- (12.913cm,7.492cm) -- (12.864cm,7.278cm) -- (12.815cm,7.067cm) -- (12.768cm,6.857cm) -- (12.723cm,6.649cm) -- (12.679cm,6.441cm) -- (12.636cm,6.235cm) -- (12.596cm,6.030cm) -- (12.557cm,5.825cm) -- (12.521cm,5.621cm) -- (12.486cm,5.418cm) -- (12.453cm,5.215cm) -- (12.422cm,5.013cm) -- (12.393cm,4.811cm) -- (12.366cm,4.609cm) -- (12.340cm,4.408cm) -- (12.316cm,4.206cm) -- (12.293cm,4.005cm) -- (12.272cm,3.805cm) -- (12.252cm,3.604cm) -- (12.158cm,2.401cm) -- (12.098cm,1.200cm) -- (12.059cm,0.000cm) -- (12.059cm,0.000cm) -- (12.045cm,0.000cm) -- (12.030cm,0.000cm) -- (12.015cm,0.000cm) -- (12.000cm,0.000cm) -- (11.985cm,0.000cm) -- (11.970cm,0.000cm) -- (11.955cm,0.000cm) -- (11.941cm,0.000cm) -- cycle;
\draw[fill=black] (O) ++(\phiss:\radp-1) circle (0.07) node[right, below] {$x_{\ss}$};
\draw[fill=black] (O) ++(\phip:\radp) circle (0.07) node[above left] {$x_0$};
\draw[dashed] (O)  ++(\phip:0) -- ++(\phip:\c);

      \draw[dotted] (O) -- (Q);
      \draw[dotted] (O) -- ++(\angabs:\c+2);
      \draw[dotted] (\c,-0.75) -- (\c,-2);
      \begin{scope}[xshift=340,yshift=340]
        \draw[dash dot] (\phip-\angb:\c) -- (\phip-\angb:\radb) arc (\phip-\angb:\phip+\angb:\radb) -- (\phip+\angb:\c);
        \draw[postaction={decorate}] (\phip:\radp) arc (\phip:\phip+\angb:\radp);
        \draw[postaction={decorate}] (\phip+\angb:\radp) arc (\phip+\angb:\phip:\radp) node[midway,xshift=7.5pt,yshift=5pt] {$_{\widehat{\theta}_0}$};
        \draw[dashed] (0,0) -- (\phip+\angb:\radb);
        \draw[postaction={decorate}] (\phiro:\radp) arc (\phiro:\phip+\angb:\radp) node[midway, above, xshift=5pt] {$_{\frac12\widehat{\theta}_0}$};
        \draw[postaction={decorate}] (\phip+\angb:\radp) arc (\phip+\angb:\phiro:\radp);
        \draw[postaction={decorate}] (270:\radp+\delta) arc (270:\phip:\radp+\delta) node[midway, above] {$_{\theta_0}$};
        \draw[postaction={decorate}] (270:\c+1.5) arc (270:360+\angabs:\c+1.5) node[midway, below] {$_{\phi(R)+\kappa_A\phi^{(\ss)}}$};
        \draw[postaction={decorate}] (360+\angabs:\c+1.5) arc (360+\angabs:270:\c+1.5);
        \draw[postaction={decorate}] (\phiro:\c+1.5) arc (\phiro:270:\c+1.5) node[midway, below] {$_{\phi(r_0)}$};
        \draw[postaction={decorate}] (270:\c+1.5) arc (270:\phiro:\c+1.5);
        \draw[postaction={decorate}] (0,0) -- (\phiro:\radp) node[midway, left, xshift=2pt] {$_{r_0}$};
      \draw[dotted] (0,0) -- (\phiro:\c+2);
      \draw[dashed] (0,0) circle (5.4);
      \draw[postaction={decorate}] (0,0) -- (-50:5.4) node[midway, right] {$_{r''}$};
      \draw[dashed] (0,0) circle (0.87);
      \draw (0,0) -- (30:1.37);
      \draw[postaction={decorate}] (30:1.57) -- (30:0.87) node [right, xshift=2pt,yshift=-1pt] {$_{r'}$};
      \draw[postaction={decorate}] (0,0) -- (180:\radb) node [midway, above, yshift=-2pt] {$_{\widehat{r}_0}$};
      \end{scope}
      \draw[thick,black] (\c,\c) circle (\c);
      \draw[fill=black] (Q) circle (0.07);
      \draw[fill=black] (O) circle (0.07);
      \node[below] at (Q) {$_Q$};      
      \node[above] at (O) {$_O$}; 
	\end{tikzpicture}
    \caption{(a) The lightly shaded area corresponds to $\dt(\KA)\setminus B_{Q}(R)$ as defined in Theorem~\ref{thm:mixedLarge} and the strongly shaded region represents $B_Q(R)$. (b) The smallest (respectively largest) circumference whose
    boundary is a dashed line is of radius $r'$
    ($r''$, respectively). 
    (c) The region whose boundary is the dashed-dot line corresponds to the box $\mathcal{B}(x_0):=[\widehat{r}_0,R]\times [\theta_0-\widehat{\theta}_0,\theta_0+\widehat{\theta}_0]$ when $\ss$ is large.}
    \label{fig:mixto}
\end{figure}
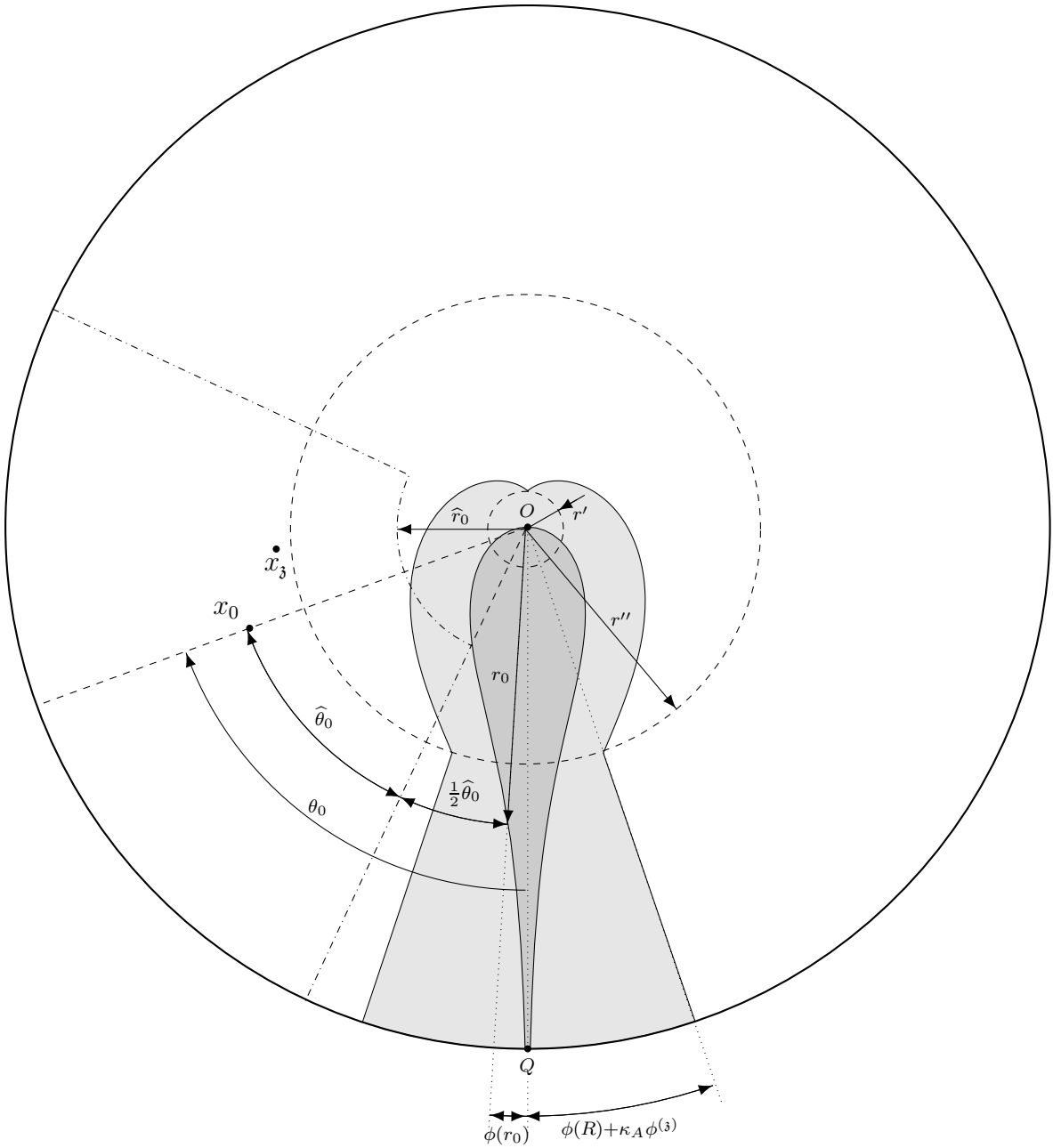

We refer to Figure~\ref{fig:mixto}(a) for an illustration of $\dt(\KA)$ as defined in the last theorem.

\medskip
The following subsection is dedicated to proving the lower bounds of the two theorems just stated (in fact, as can be seen from the proof, the assumptions about $\ss$ are slightly milder in the proofs). The subsequent subsection then deals with the corresponding upper bounds.

\subsection{Lower bounds}\label{sec:Lower}
In this subsection we prove the lower bounds of Theorem~\ref{thm:mixedSmall} and~\ref{thm:mixedLarge}.

\subsubsection{The case $\ss$ small}
We start with the proof of the lower bound of Theorem~\ref{thm:mixedSmall}. 
We establish the following proposition, which proves the  lower bound of Theorem~\ref{thm:mixedSmall}. 
\begin{proposition}\label{generalschico_lower}
 Fix $\beta\le\frac{1}{2}$ and take $\dt$ as in~\eqref{mix:lower:defD}. If $\ss=O(1)$, then
\[\int_{\dt}\P_{x_0}(T_{det}\leq \ss)d\mu(x_0) 
\;=\; \Omega(ne^{-\beta R}\sqrt{\ss}).\]
\end{proposition}
\begin{proof}
Consider $x_0$ with $|\theta_0| \le \phi(r_0)+\sqrt{\ss} e^{-\beta r_0}$ and $r_0 \ge c$ for some constant $c > 0$. Note that for such choice of $r_0$ we have $\coth(r_0)\leq\coth(c)=O(1)$ and thus within time $O(1)$, with probability bounded away from zero, the radial coordinate changes by at most $1$. Conditioning on this, we consider the angular movement variance $1$ Brownian motion $B_{\II{\ss}}$ where now $\II{\ss}:=\int_0^\ss\cosech^2(\beta r_{\ss})d\ss\geq (\sqrt{\ss}/\sinh(\beta(r_0+1))^2$. As in the proof of Theorem~\ref{thm:mainangular}, define the exit time $H_{[-a,b]}$ from the interval $[-a,b]$ where
 $a:=\phi(r_0+1)-|\theta_0|$ and $b:=2\pi-\phi(r_0+1)-|\theta_0|$, and as in~\eqref{eqn:angular-exitt} in the proof of Theorem~\ref{thm:mainangular},
$$
\mathbb{P}_{x_0}(T_{det} \le \ss) \ge \mathbb{P}(H_{[-a,b]} \le \II{\ss})=\Omega\big(\Phi\big((\phi(r_0+1)-|\theta_0|)\sinh(\beta (r_0+1))/\sqrt{\ss}\big)\big).
$$
Since $\phi(\cdot)>0$ and $\sinh(x)\leq e^{x}$ we conclude that $\PP_{x_0}(T_{det}\le \ss)
  =\Omega(\Phi(-|\theta_0|e^{\beta (r_0+1)}/\sqrt{\ss}))$.
Integrating over the elements of $\dt$ satisfying $|\theta_0| \le \sqrt{\ss} e^{-\beta r_0}$, 
we have
$$
\int_{\dt} \mathbb{P}_{x_0}(T_{det} \le \ss) d\mu(x_0) =
\Omega(ne^{-\alpha R})\int_{1}^R \int_0^{\sqrt{\ss} e^{-\beta \widehat{r}_0}} 
\Phi(-\theta_0 e^{\beta \widehat{r}_0}/\sqrt{\ss})\sinh(\alpha \widehat{r}_0) d\theta_0 d\widehat{r}_0. 
$$
Performing the change of variables $y_0:=\theta_0 e^{\beta \widehat{r}_0}/\sqrt{\ss}$, since $\alpha>\beta$, and using $\sinh(x)=\Theta(e^x)$ for $x=\Omega(1)$, we get for the latter
$$
\Omega(ne^{-\alpha R})\int_0^{1} \Phi(-y_0)\sqrt{\ss}dy_0\int_{1}^R e^{(\alpha-\beta)\widehat{r}_0} d\widehat{r}_0=\Omega(ne^{-\beta R}\sqrt{\ss})\int_0^{1} \Phi(-y_0) dy_0.
$$
Note that the last integral is $\Omega(1)$, and thus the stated result follows.
\end{proof}

\subsubsection{The case $\ss$ large}
We now prove the lower bound of Theorem~\ref{thm:mixedLarge}. In fact, we will only need to assume the milder condition $\ss \ge Ce^{2\KA^2}$ with $C$ being a large enough constant. 
We start by showing that particles that start close to the origin detect quickly with at least constant probability:
\begin{lemma}\label{lem:closetoorigin}
Let $\tau > 0$, and let $c^*:=C^*/\beta$ for some arbitrarily large constant $C^* > 0$. Then,
$$
 \inf_{x_0\in B_O(c^*)}\PP_{x_0}(T_{det}\leq \tau)=\Omega(1).
$$
\end{lemma}
\begin{proof}
Let $T_{h}$ be the smallest time $t$ such that $r_t=2c^*$ and notice that, calling $\PP^r$ the law of the radial component of the process, for any point $x_0=(r_0, \theta_0)$ with~$r_0 \le c^*$, it holds that
$\PP^r_{r_0}(T_{h}\ge \tau)\,\geq\,\PP^r_{c^*}(T_{h} \ge \tau).$
Now, fix any realization $\{r_s\}_{s\geq 0}$ of the radial component of the process such that~$T_h\ge\tau $ and observe that for any such realization, detection is guaranteed if $|\theta_s-\theta_0|>2\pi$ for some $0 \le s \le \tau$. Under the event $T_{h}\ge\tau$ the radial coordinate at any time $s \le \tau$ is at most $2c^*$. Thus, the angular movement $\theta_\tau-\theta_0$ has a normal distribution centered at zero with variance at least
\[\int_0^{\tau}\sinh^{-2}(\beta r_s)ds\,\geq\, \tau\sinh^{-2}(2C^*)  =\Omega(1).\]
Thus, with constant probability, within time~$\tau$ the angular movement covers an angle of $2\pi$. The lemma follows. 
\end{proof}

We now deal with the remaining starting points $x_0\in\dt$.
Before doing so we establish a simple fact concerning the random variables defined for nonnegative integer values of $k$ as follows:
\begin{equation}\label{mixedLower:def:Ik}
I_k := \int_0^\ss {\bf 1}_{(0,k]}(r_s)ds.
\end{equation}
Recall that $\pi(r):=\frac{\alpha\sinh(\alpha r)}{\cosh(\alpha R)-1}$, $0\leq r\leq R$, is the stationary distribution of the process~$\{r_s\}_{s\geq 0}$.
\begin{fact}\label{fact:stationary}
If $k\geq\frac{1}{\alpha}\log 2$, 
then $\EE_{\pi}(I_k) \ge \tfrac14\ss e^{-\alpha (R-k)}$.
\end{fact}
\begin{proof}
Suppose $r_0$  is distributed according to the stationary distribution $\pi$. 
Since $\cosh(x)-1\leq \frac12 e^x$, by definition of $I_k$, we have
\[
\EE_{\pi}(I_k) = \int_0^{\ss}\PP_{\pi}(r_s\leq k)ds =\ss\pi((0,k]) = \ss\cdot\frac{\cosh(\alpha k)-1}{\cosh(\alpha R)-1}
\geq 2\ss e^{-\alpha R}(\cosh(\alpha k)-1).
\]
The desired conclusion follows observing that the map $x\mapsto 1+e^{-2x}-2e^{-x}$ is non-decreasing, and thus for $x\geq\log 2$, we have 
$\cosh(x)-1=\frac12 e^x(1+e^{-2x}-2e^{-x})\geq\frac14 e^{x}$.
\end{proof}

We are now ready to state and prove the proposition dealing with particles whose initial position are points satisfying the first condition in the definition of $\dt$:
\begin{proposition}\label{mixedLowerBoundschico}
Let $\KA > 1$. There is a sufficiently large constant $C>0$ such that if $\ss\ge Ce^{2\KA^2}$ and $\ss=O(\mathfrak{Z})$, then for any $x_0=(r_0, \theta_0)$ with $|\theta_0| \le \phi(R)+\KA \phis$, we have the following: 
\begin{enumerate}[(i)]

\item\label{mixedLowerBoundschico:imt1}  If $\alpha \ge 2\beta$, then $\displaystyle 
\PP_{x_0}(T_{det} \leq \ss)=\Omega(\tfrac{1}{\KA}e^{-\frac12\KA^2}) =\Omega(e^{-0.7\KA^2})$. 

\item\label{mixedLowerBoundschico:imt2} 
If $\alpha < 2\beta$, then 
$\displaystyle
\PP_{x_0}(T_{det} \leq \ss) = \Omega(\KA^{-\alpha/(\beta\wedge\frac12)})$.
\end{enumerate}
\end{proposition}
To prove the just stated proposition we use the following standard fact several times, so we state it explicitly.
\begin{fact}[\cite{gordon41}]\label{mixedLower:fact:brownian}
Given the radial component $\{r_{s}\}_{s\geq 0}$ of a particle's trajectory,  the angular component $\{\theta_s\}_{s\geq 0}$ law is that of a Brownian motion indexed by $\II{\ss}:=\int_0^{\ss}\cosech^2(\beta r_s)ds$. 
If $\II{\ss} \ge \sigma^2 > 0$ and $\kappa>0$, then 
\[
\PP(\sup_{0\leq s\leq \ss} B_{\II{s}}\geq \kappa \sigma \mid \{r_s\}_{s\geq 0}) \geq  \frac{\kappa}{\sqrt{2\pi}(\kappa^2+1)}e^{-\frac12\kappa^2}=\Omega\Big(\frac{1}{\kappa}e^{-\frac12\kappa^2}\Big).
\]
\end{fact}
Now we proceed with the pending proof.
\begin{proof}[Proof of Proposition~\ref{mixedLowerBoundschico}.]
Assume first $\beta \ge 1/2$
(and thus necessarily $\alpha < 2\beta$). In this case the radial movement dominates and the proof of~\eqref{mixedLowerBoundschico:imt2} is very similar to the one given for $x_0 \in \dt$ where $|\theta_0|\le \phi(R)+\kappa\phis$: assume first that $\ss$ (and $\KA$) is such that $|\theta_0|\le \phi(R)+\KA \phis \le \pi/2 - c$ for some $c > 0$. Define $\overline{r}_0$ be as $\rabs_0$ in the proof of Proposition~\ref{prop:rad-upperBnd} (that is $\overline{r}_0$ corresponds to the absorption radius in case there was radial movement only; since $\phi(R)+\KA \phis \le \pi/2 - c$, we have $\overline{r}_0=\Omega(1)$). By the same argument as given in the proof of Proposition~\ref{prop:rad-upperBnd}, with $\KA$ playing the role of $\kappa$ therein, with probability $\Omega(\KA^{-2\alpha})$, there exists a time moment $T \le \ss$, at which a radial value of $\overline{r}_0$ is reached. In this case, by symmetry of the angular movement, with probability at least $1/2$, either the angle at time $T$ satisfies $|\theta_T| \le |\theta_0|$ or there exists $t \le T$ where the angle $\theta_t=0$ and we detected already by time $t$, and hence in this case $Q$ is detected by time $T$ with probability $\Omega(\KA^{-2\alpha})$. Assume then that~$\ss$ (and $\KA$) is such that $\pi/2 - c \le \phi(R)+\KA \phis$. In this case, let $\overline{r}_0$ be equal  (assuming there were radial movement only) to the absorption radius $\rabs_0$ that would correspond to an angle exactly $|\theta_0|=\pi/2 -c$. Note that in this case $\overline{r}_0=\Theta(1)$. By the proof of Proposition~\ref{prop:rad-upperBnd}, with probability $\Omega(\KA^{-2\alpha})$ a radial value of $\overline{r}_0$ is reached at a time moment  $T \le \frac12\ss$. In this case, with probability at least $1/2$, as before, either there existed a moment $t$ with $\theta_t=0$, and we would have detected $Q$, or $|\theta_T|\le|\theta_0|$. In this case, by Lemma~\ref{lem:closetoorigin}, since $\overline{r}_0=\Theta(1)$, with constant probability we detect $Q$ in time $\frac12\ss=\Omega(1)$, and hence also in this case $Q$ is detected by time $\ss$ with probability $\Omega(\KA^{-2\alpha})$.

Consider thus $x_0 \in \dt$ with $|\theta_0| \le \phi(R)+\KA\phi^{(\ss)}$ under the assumption $\beta < 1/2$. Recall that we also may assume that $r_0 \ge c^*$ for $c^*$ as in the statement of Lemma~\ref{lem:closetoorigin} since other values of $r_0$ were already dealt with in said lemma.
First, note that for $x_0 \in B_Q(R)$ the lower bound trivially holds. 
So, henceforth let $x_0 \in \dt \setminus B_Q(R)$. Since by hypothesis $\ss\in O(\mathfrak{Z})=O(e^{\alpha R})$, we may assume that $\ss < \exp(\alpha(R-\frac{1}{\alpha}\log 2 - 1+c))$ for $c>0$ sufficiently large. 
Let $\rho:=R-\frac{1}{\alpha}\log\ss+c$ and note that $\frac{1}{\alpha}\log 2+1 < \rho$ by the previous assumption on $\ss$. 
Denote by $T_{\rho-1}$ the random variable corresponding to the first time the process reaches $B_O(\rho-1)$. Recall from Part~\eqref{radial:itm:phi2} of Lemma~\ref{lemmaradial}, with $\yabs_0=\rho-1$ and $\auxY=R$, that for the radial component $\{r_s\}_{s\geq 0}$ of the movement, and for any $r_0 \in [\rho-1, R]$,
	\[\EE_{r_0}(T_{\rho-1})\leq\frac{4}{\alpha^2}e^{\alpha(R-\rho+1)}
	=\frac{4}{\alpha^2} \ss e^{-\alpha(c-1)}.
	\]
	By Markov's inequality, for $c$ sufficiently large so that  $\frac{4}{\alpha^2}e^{-\alpha(c-1)} \le \frac14$, it follows that the event $\mathcal{A}$ corresponding to the process reaching $B_O(\rho-1)$ 
	before time $\frac12\ss$ happens with
	probability
	\[\PP(\mathcal{A})=1-\PP_{r_0}(T_{\rho-1}\geq \tfrac12 \ss) \ge \tfrac{1}{2}.\]
	Now, let $\overline{\rho}:=\max\{\rho-\beta \log \KA, c^*\}.$ Moreover, define $\mathcal{B}$ as the event that starting from radius $\rho-1$ we hit the radial value $\overline{\rho}$ before hitting the radial value $R$.
	Define $g(r):=\log(\tanh(\frac12\alpha R)/\tanh(\frac12\alpha r))$ as in Fact~\ref{fct:radial-varphi2} and observe that as argued therein,
	since $\rho=\Omega(1)$ we have $g(\rho-1)=O(e^{-\alpha \rho})$ and because
	$\overline{\rho}=R-\Omega(1)$ we have $g(\overline{\rho})=\Omega(e^{-\alpha \overline{\rho}})$. Recall from Part~\eqref{radial:itm:phi3} of Lemma~\ref{lemmaradial} that $g(\rho-1)/g(\overline{\rho})$ is the probability that starting from radius $\rho-1$ we hit the radial value $\overline{\rho}$ before hitting the radial value $R$, and we have
$$
\PP(\mathcal{B})=\frac{g(\rho-1)}{g(\overline{\rho})}=\Omega(e^{\alpha(\overline{\rho}-\rho)})=
\Omega(\KA^{-\alpha/\beta}).
$$
 From Part~\eqref{radial:itm:phi4} of Lemma~\ref{lemmaradial} we obtain 
$$
\EE_{\rho-1}(T_{\overline{\rho}} \mid T_{\overline{\rho}} < T_R) \le \frac{2}{\alpha}(\beta \log \KA+1)+\frac{2}{\alpha^2} \le \frac{\ss}{8},
$$
where the last inequality follows (with room to spare) by our assumption of $\ss$.
Thus, by Markov's inequality, conditionally under $\mathcal A \cap \mathcal B$,
with probability at least $\frac12$, we reach $\overline{\rho}$ before time $\frac34\ss$. Let $\mathcal C$ be the corresponding event. Conditional under $\mathcal C$, with constant probability the particle stays one unit of time inside $B_O(\overline{\rho}+1)$ before time $\ss$, and let this event be called $\mathcal D$. Conditional under $\mathcal D$, the angular component's law during one unit of time inside $B_O(\overline{\rho}+1)$ is $B_{\II{\ss}}$ with $$\II{\ss}\ge \cosech^2(\beta(\overline{\rho}+1)) \ge e^{-2\beta(\overline{\rho}+1)} \geq e^{-2\beta (R-\frac{1}{\alpha}\log\ss+c-\frac{1}{\beta}\log \KA +1)}=:\sigma^2.$$
Note that $\sigma = e^{-\beta R}\ss^{\frac{\beta}{\alpha}}\KA e^{-\beta(c+1)}$ for the absolute constant $c > 0$ independent of $\KA$ from above, and recall also that
 $|\theta_0| \le \phi(R)+\KA \ss^{\frac{\beta}{\alpha}}e^{-\beta R} \le 2\KA \ss^{\frac{\beta}{\alpha}}e^{-\beta R}$ by assumptions on $\beta < 1/2$ and $\ss$, with room to spare. 
Let $\mathcal{E}$ be the event that when reaching $B_O(\overline{\rho}+1)$ the angle at the origin spanned by the particle and $Q$ is (in absolute value) is at most $|\theta_0|$ or there was a moment $t$ before reaching $B_O(\overline{\rho}+1)$ with $\theta_t=0$ (and thus we detected at time $t$). Note that by symmetry of the angular movement, $\PP(\mathcal{E})\ge \frac12$.
 Conditional under $\mathcal{D} \cap \mathcal{E}$, by Fact~\ref{mixedLower:fact:brownian},  since $\KA > 1$, for some $c_1=c_1(c)$ we have
\[
\P_{x_0}(T_{det}\leq\ss)\geq \PP(\sup_{0\leq s\leq \ss} B_{\II{s}}\geq c_1 \sigma \mid \{r_s\}_{s\geq 0}) =\Omega(1),
\]
 and since $\mathcal{D} \cap \mathcal{E}$ holds with probability $\Omega(\KA^{-\alpha/\beta})$, we have
 $
\P_{x_0}(T_{det}\leq\ss)= \Omega(\KA^{-\alpha/\beta})$,
which finishes the proof of the case $\alpha < 2\beta$ and $\beta < 1/2$ (and thus the proof of~\eqref{mixedLowerBoundschico:imt2}).

 We now deal with the remaining $\alpha \ge 2\beta$ case (and therefore necessarily $\beta < 1/2$). Assume first $\alpha > 2\beta$. This argument is analogous to the one for angular movement: we repeat it for convenience. 
Given the trajectory of $\{r_{s}\}_{s\geq 0}$, recall that the angular component's law is that of a Brownian motion $B_{\II{\ss}}$, where $\II{\ss}:=\int_0^{\ss} \cosech^2(\beta r_s)ds \ge 4\ss e^{-2\beta R} =: \sigma^2$.
Hence, using Fact~\ref{mixedLower:fact:brownian}, since $\KA > 1$ and using that $0.2x^2\ge \log x$ for all $x\ge 1$,
$$
\P_{x_0}(T_{det}\leq\ss)\geq \PP(\sup_{0\leq s\leq \ss} B_{\II{s}}\geq \KA \sigma \mid \{r_s\}_{s\geq 0}) =\Omega\big(-\tfrac{1}{\KA}e^{-\frac12\KA^2}\big)={\Omega(e^{-0.7\KA^2})}, 
$$
showing the result in the case $\alpha > 2\beta$.

Finally, suppose $\alpha=2\beta$ (therefore $\beta > \frac14$).
First, suppose that the starting point $r_0$ of the radial component is distributed according to the stationary distribution of the process, that is, $\pi(r):=\frac{\alpha\sinh(\alpha r)}{\cosh(\alpha R)-1}$ with $0\leq r\leq R$. 
We will see that the contribution to  $\II{\ss}:=\int_0^{\ss} \cosech^2(\beta r_s) ds$ of the time spent around different radial values is roughly the same, forcing a logarithmic correction. To bound $\II{\ss}$ from below, let $\overline{k}=R-\lfloor\tfrac{1}{\alpha}\log (\ss/4) \rfloor$, which by our hypothesis $\ss= O(\mathfrak{Z})=O(e^{\alpha R}/(\alpha R))$ implies $\overline{k}=\omega(1)$ (in particular it implies also $\overline{k} \ge \frac{1}{\alpha}\log 2$). 
Note that 
\[
\II{\ss} = \int_0^{\overline{k}}\cosech^2(\beta r_s)ds
  + \!\sum_{k=\overline{k}+1}^R \int_{k-1}^{k}\cosech^2(\beta r_s)ds
      \geq 4\Big(I_{\overline{k}}e^{-2\beta \overline{k}}+\!\sum_{k=\overline{k}+1}^R (I_k - I_{k-1})e^{-2\beta k}\Big).
\]
Hence, using that $\beta\geq\frac14$ implies that $4(1-e^{-2\beta})\geq 1$,
\begin{equation}
    \II{\ss} 
    \geq 4\Big(\sum_{k=\overline{k}}^{R-1}I_ke^{-2\beta k}(1-e^{-2\beta}) +I_R e^{-2\beta R} \Big)
    \ge \sum_{k=\overline{k}}^R e^{-2\beta k}I_{k}.
\end{equation}
 So, recalling that $\overline{k}\ge \frac{1}{\alpha}\log 2$, by Fact~\ref{fact:stationary}, for all $k \in \{\overline{k},\ldots,R\}$, we have $\EE_{\pi}(I_k)\geq \frac14 \ss e^{-\alpha(R-k)}$,
which gives an estimate for the value of the $I_k$ variables. Moreover, by Corollary~\ref{radial:cor:coupling}, 
since $k \ge \overline{k}$ and by definition of $\overline{k}$ we have $\ss \ge 4e^{\alpha (R-\overline{k})} \ge 4e^{\alpha (R-k)}$. Since $\cosh(x)-1=\frac12 e^x(1-e^{-x})^2\le \frac12 e^x$, using once more that $k\ge\overline{k}\ge\frac{1}{\alpha}\log 2$, we obtain 
$\pi((0,k])\ge e^{-\alpha(R-k)}(1-e^{-\frac12\alpha k})^2\ge \frac{1}{4}e^{-\alpha(R-k)}$, so $\ss\ge 1/\pi((0,k])$ and the assumptions of Corollary~\ref{radial:cor:coupling} are satisfied. Hence, there exist $\widetilde{c}, \widetilde{\eta} \in (0,1)$, such that for each $\overline{k} \le k \le R$, we have that the expectation of the indicator $Z_k$ of the event $\{I_k\geq \widetilde{c}\ss e^{-\alpha (R-k)}\}$ is at least~$\widetilde{\eta}$. Let $Z:=\sum_{k=\overline{k}}^R Z_k$ and that $\EE(Z) \ge (R-\overline{k}+1)\widetilde{\eta}$. Define then the event 
$\mathcal{E}:=\{Z>\frac{\widetilde{\eta}}{2}(R-\overline{k}+1)\}$. Since $\EE(Z)\le (\PP(\mathcal{E})+\frac{\widetilde{\eta}}{2}\PP(\overline{\mathcal{E}}))(R-\overline{k}+1)$,
it must be the case that $\PP(\mathcal{E})\ge\frac{\widetilde{\eta}}{2}/(1-\frac{\widetilde{\eta}}{2})\ge\frac{\widetilde{\eta}}{2}$.
Thus, for a fixed realization of $\{r_s\}_{s\geq 0}$ satisfying $\mathcal{E}$, since $\alpha=2\beta$, we have 
\[
\sum_{k=\overline{k}}^{R}e^{-2\beta k}I_k 
\geq\sum_{k=\overline{k}}^{R}e^{-2\beta k}\widetilde{c}\ss e^{-\alpha (R-k)}Z_k
= \widetilde{c}\ss e^{-2\beta R}\sum_{k=\overline{k}}^R Z_k 
> \widetilde{c}\ss e^{-2\beta R}\frac{\widetilde{\eta}}{2}(R-\overline{k}+1).
\]
Note that  $R-\overline{k}+1\ge \frac{1}{\alpha}\log (\ss/4)$ by definition of $\overline{k}$, so 
$
\sum_{k=\overline{k}}^R e^{-2\beta k}I_k\geq \frac{\widetilde{c}}{2\alpha}\widetilde{\eta}\ss e^{-2\beta R}\log(\ss/4)) \ge \frac{\widetilde{c}}{3\alpha}\widetilde{\eta}\ss e^{-2\beta R}\log\ss,
$
where we used the assumption that $\ss$ is at least a sufficiently large constant.
Thus, under $\mathcal{E}$ the angular movement dominates stochastically
a Brownian motion $B_{\frac{\widetilde{c}}{3\alpha}\widetilde{\eta}\ss e^{-2\beta R}\log\ss}$ with standard deviation 
$e^{-\beta R}\sqrt{(\widetilde{c}\ss \widetilde{\eta}/(3\alpha))\log\ss}=:\sigma$. By Fact~\ref{mixedLower:fact:brownian}, 
\[
\PP(\sup_{0\leq s\leq \ss} B_{\II{s}}\geq \KA \sigma \mid \{r_s\}_{s\geq 0})=\Omega(\tfrac{1}{\KA}e^{-\frac12\KA^2}).
\]
Thus, for $x_0 \in \dt$ with $|\theta_0| \le \phi(R)+\KA \phis\le 2\KA e^{-\beta R}\sqrt{\ss \log \ss}$, since $\P_{x_0}(\mathcal{E})\geq \frac{\widetilde{\eta}}{2}$ and using that $0.2x^2\ge \log x$ for all $x\ge 1$,
\begin{equation}\label{startstationary}
\P_{x_0}(T_{det}\le \ss)\geq \tfrac{\widetilde{\eta}}{2}\P_{x_0}(T_{det} \le \ss \mid \mathcal{E})=\Omega(\tfrac{1}{\KA}e^{-\frac12\KA^2})=\Omega(e^{-0.7\KA^2}).
\end{equation}

It remains to show thus that with positive probability the trajectory starting with a fixed initial radius for $x_0 \in \dt$ can be coupled in such a way that the probability of detection of the target by time $\ss$ can be bounded from below by the probability of detection when the radius is chosen according to the stationary distribution $\pi(\cdot)$. Denote by $\widehat{r}_t$ the radial component at time $t$ when starting according to the stationary distribution.
Consider the event $\mathcal{A}$ that the initial radial value $\widehat{r}_0$ is at most one unit away from the 
boundary of $B_O(R)$, that is, $\mathcal{A}:=\{\widehat{r}_0\in [R-1,R]\}$.
Clearly, $\PP(\mathcal{A})=\pi([R-1,R])=\Omega(1)$. 
Let $\mathcal{B}$ be the event that $\{\widehat{r}_s\}_{s\geq 0}$ starting from $R-1$ hits $R$ by time $\frac12\ss$. 
Conditional under $\mathcal{A}$, since the time to hit $R$ is clearly dominated by the time a standard  Brownian motion (corresponding to a one-dimensional radial movement) hits $R$ starting from $R-1$, by our lower bound on $\ss$, we have that $\PP(\mathcal{B}\mid \mathcal{A})=\Omega(1)$.  Note that conditional under $\mathcal{A}\cap\mathcal{B}$ either the trajectories starting with a fixed value of $r_0$ on the one hand and with $\widehat{r}_0$ according to the stationary distribution $\pi(\cdot)$ on the other hand must cross by time $\frac12\ss$ (and they can be coupled naturally from then on for a time interval of $\frac12\ss$), or $r_t \le \widehat{r}_t$ for any $0 \le t \le \frac12\ss$, and during an initial time period of length $\frac12\ss$ the detection probability starting from $r_0$ stochastically dominates the one when starting from $\widehat{r}_0$. Thus, with probability $\PP(\mathcal{A}\cap\mathcal{B})=\Omega(1)$, the process $\{r_s\}_{s\geq 0}$ can be successfully coupled with $\{\widehat{r}_s\}_{s\geq 0}$, and since we aim for a lower bound, we assume that $\mathcal{A}\cap\mathcal{B}$ holds. Conditional under $\mathcal{A}\cap\mathcal{B}$, we may thus apply the reasoning yielding~\eqref{startstationary} with $\ss$ replaced by $\frac12\ss$, and the result follows.
\end{proof}

We still have to deal with particles whose initial location are points satisfying the second condition in the definition of $\dt$. The next proposition does this.
\begin{proposition}\label{prop:uniformlowerboundN2}
Let $\KR >1$ and assume that $\ss\ge Ce^{2\KA^2}$ for $C$ large enough.
Then, for any $x_0=(r_0,\theta_0)$ with 
$(|\theta_0|-\phi(r_0))e^{(\beta\wedge\frac12)r_0} \le \KR$, we have the following: 
\begin{enumerate}[(i)]

\item\label{prop:uniformLowerboundN2:itm1} If $\alpha \ge 2\beta$, then 
$\displaystyle 
\PP_{x_0}(T_{det} \leq \ss)=\Omega(\KR^{-\frac{\alpha}{\beta}})
=\Omega(e^{-{\frac{\alpha}{\beta}\KA^2}}).
$
\item\label{prop:uniformLowerboundN2:itm2} If $\alpha < 2\beta$, 
then  
$\displaystyle
\PP_{x_0}(T_{det} \leq \ss) = \Omega(\KR^{-\alpha/(\beta\wedge\frac12)})=\Omega(\KA^{-\alpha/(\beta\wedge\frac12)}).
$
\end{enumerate}
\end{proposition}
\begin{proof}
Thanks to Lemma~\ref{lem:closetoorigin} we may, and will, assume throughout the proof that
$r_0>c^*$ for an arbitrarily large $c^*$. Under said condition, the proof argument formalizes the following intuitive fact: Either there is a chance for the particle to move radially towards the origin so that the boundary of $B_Q(R)$ is reached and detection happens, or there is a chance to move radially towards the origin, enough so that the particle stays long enough in such a region close enough to the origin so that the relatively large angle the particle traverses makes detection probable.

We first assume $\beta<1/2$.
Observe that we may assume $|\theta_0|\ge 2\phi(r_0)$, since otherwise during one time unit there is constant probability that the radial value is at most $\min\{R,r_0+1\}$, and conditional under this event, during this time unit with constant probability an angular movement of standard deviation at least $e^{-\beta (r_0+1)}$ is performed, thus covering with constant probability an angular distance of $e^{-\beta(r_0+1)} \ge 2\phi(r_0)$ in the case of $\beta < 1/2$ for $r_0 > c^*$. We may and will thus below replace the condition $(|\theta_0|-\phi(r_0))e^{(\beta\wedge\frac12)r_0} \le \KR$ by $|\theta_0|e^{(\beta\wedge\frac12)r_0} \le 2\KR$. We consider the worst-case scenario, i.e., $|\theta_0|=2\KR e^{-\beta r_0}$, or equivalently $r_0=\frac{1}{\beta}\log(2\KR/|\theta_0|)$. 
We restrict our discussion to vertices with $r_0 < R-\log \KA$, as other vertices were already considered before: indeed, for values of $r_0 \ge R-\log \KA$, in case~\eqref{prop:uniformLowerboundN2:itm1}
we have $2\KR e^{-\beta r_0} = 2e^{\KA^2}e^{-\beta r_0} \le e^{\KA^2}e^{-\beta R}\KA \le \KA \sqrt{\ss}e^{-\beta R}=\KA \phis$, where the second inequality follows from our assumption of $\ss\ge Ce^{2\KA^2}$ with $C$ large. Hence such values of $x_0=(\theta_0, r_0)$ satisfy the first condition of $\dt$ and were treated in Proposition~\ref{mixedLowerBoundschico}. In case~\eqref{prop:uniformLowerboundN2:itm2}, for values of $r_0 \ge R-\log \KA$, we have $2\KR e^{-\beta r_0}=2\KA e^{-\beta r_0}\le 2\KA^2 e^{-\beta R}\le \KA \phis$, where again the second inequality follows from the same assumption on $\ss \ge Ce^{2\KA^2}$ (again with room to spare), and again this case was already dealt with in Proposition~\ref{mixedLowerBoundschico}.
Define $g(r):=\log(\tanh(\frac12 \alpha R)/\tanh(\frac12 \alpha r))$ as in Fact~\ref{fct:radial-varphi2} and let $\overline{r}_0:=\max\{r_0-\frac{1}{\beta}\log \KR, c^*\}$. By arguments given in said fact, since we have restricted our discussion to the case where $r_0=R-\Omega(1)$, we have $g(r_0)=\Omega(e^{-\alpha r_0})$ and, since $\overline{r}_0\geq c^*$ with $c^*$ large, we also have $g(\overline{r}_0)=O(e^{-\alpha \overline{r}_0})$. 
Recall from Part~\eqref{radial:itm:phi3} of Lemma~\ref{lemmaradial} that $g(r_0)/g(\overline{r}_0)$ is the probability that starting from radius $r_0$ we hit the radial value $\overline{r}_0$ before hitting the radial value $R$, and we have
$$
\frac{g(r_0)}{g(\overline{r}_0)}=\Omega(e^{\alpha(\overline{r}_0-r_0)})=
\Omega(e^{\alpha(\max\{r_0-\frac{1}{\beta}\log \KR, c^*\}-r_0)})=\Omega(\KR^{-\alpha/\beta}).
$$
Thus, by Part~\eqref{radial:itm:phi4} of Lemma~\ref{lemmaradial} with $\auxy_0:=\overline{r}_0$ and $\auxY:=R$, we have that $\EE_{r_0} (T_{\overline{r}_0} \mid T_{\overline{r}_0} < T_R) \le \frac{2}{\alpha\beta}\log \KR + \frac{2}{\alpha^2} \le \frac14\ss$ by our lower bound hypothesis of $\ss$ with $C:=C(\beta)$ sufficiently large. 
By Markov's inequality,  conditional under $T_{\overline{r}_0} < T_R$, starting at radius $r_0$, with probability at least $1/2$, the value of $\overline{r}_0$ is hit by time $\frac12\ss$. In this case, if $\overline{r}_0=c^*$, then by Lemma~\ref{lem:closetoorigin} the target is detected with constant probability in constant time, and we are done. If $\overline{r}_0 =r_0-\frac{1}{\beta}\log \KR > c^*$, then with constant probability, during the ensuing unit time interval, the radial coordinate $r$ is always at most $\overline{r}_0+1$. Call this event $\mathcal{A}$. By symmetry of the angular movement, with probability at least $1/2$, either the angle $\theta$ at the hitting time of~$\overline{r}_0$ is in absolute value at most $|\theta_0|$ or there was a time moment $t$ with $\theta_t=0$ before, and we had detected $Q$ by time $t$ already. Call this event~$\mathcal{B}$. Conditional under $\mathcal{A} \cap \mathcal{B}$, the variance of the angular movement after reaching $\overline{r}_0$ during the following unit time interval is at least 
$
e^{-2\beta (r_0-\frac{1}{\beta}\log\KR+1)} =e^{-2\beta (r_0+1)}\KR^2$,
so recalling that by our worst-case assumption $|\theta_0|=2\KR e^{-\beta r_0}$,  the standard deviation is thus at least $e^{-\beta (r_0+1)}\KR=e^{-\beta}|\theta_0|/2$. Hence, with constant probability, independently of $\KR$ (and also independently of $\beta$), in this case an angle of $|\theta_0|$ is covered during a unit time interval, and by our lower bound assumption of $\ss$, the particle is detected by time $\ss$, finishing also this case and 
establishing the proposition when $\beta<1/2$ 
by plugging in the respective relations between $\KA$ and $\KR$ in order to obtain the last equality of the the two parts of the statement.

Next, we consider Part~\eqref{prop:uniformLowerboundN2:itm2} which encompasses a case similar to the one analyzed in Section~\ref{sec:radial}, so we give only a short sketch of how it is handled. By hypothesis and case assumption $|\theta_0| \le \phi(r_0)+\KR e^{-\frac12r_0}$. Since $\KR >1$, we thus have $|\theta_0|\le 3\KR e^{-\frac12r_0}$. Again, it suffices to show the statement for the worst-case scenario, that is, we may assume $|\theta_0|= 3\KR e^{-\frac12r_0}$. We may assume that $r_0 < R- \KR$, since for vertices with $r_0 \ge R-\KR$, using our assumption of $\ss\ge Ce^{2\KA^2}$ (with room to spare) we have
$3\KR e^{-\frac12r_0} \le 3\KR e^{\KR} e^{-\frac12 R} \le \KA \ss^{1/(2\alpha) e^{-\frac12 R}}=\KA \phis$, and such cases were already dealt with in the proof of Part~\eqref{mixedLowerBoundschico:imt2}  of Proposition~\ref{mixedLowerBoundschico}. By Lemma~\ref{lem:closetoorigin} we may also assume $r_0 > c^*$ for arbitrarily large $c^*$.
Define $\overline{r}_0:=\max\{r_0-2 \log \KR-2, c^{*}\}$. 
As before, recall from Part~\eqref{radial:itm:phi3} of Lemma~\ref{lemmaradial} that $g(r_0)/g(\overline{r}_0)$ is the probability that starting from radius $r_0$ we hit the radial value $\overline{r}_0$ before hitting the radial value $R$. Since $r_0=R-\Omega(1)$ and since $\overline{r}_0 \ge c^*$, we have
$$
\frac{g(r_0)}{g(\overline{r}_0)}=\Omega(e^{\alpha(\overline{r}_0-r_0)})=
\Omega(e^{\alpha(\max\{r_0-2\log \KR-2, c^*\}-r_0)})=\Omega(\KR^{-2\alpha}),
$$
and let $\mathcal{A}$ be the event that this indeed happens. 
By Part~\eqref{radial:itm:phi4} of Lemma~\ref{lemmaradial} with $\auxy_0:=\overline{r}_0$ and $\auxY:=R$, we have that $\EE_{x_0}(T_{\overline{r}_0}\mid \mathcal{A}) \le \frac{4}{\alpha}\log \KR +\frac{4}{\alpha}+ \frac{2}{\alpha^2} \le \frac14\ss$ by our lower bound hypothesis on $\ss$. By Markov's inequality, with probability at least $1/2$, $T_{\overline{r}_0} < \frac12\ss$. Let $\mathcal{B}$ be the event to reach the radius $\overline{r}_0$ by time $\frac12\ss$.
We have
$$
\PP_{x_0}(\mathcal{B}) \ge \PP_{x_0}(\mathcal{B} \mid \mathcal{A})\PP_{x_0}(\mathcal{A}) =\Omega(\KR^{-2\alpha}).
$$

 Let $\mathcal{C}$ be the event that one of the following events occurs: Either at the moment $T$ when reaching $\overline{r}_0$ the angle made at the origin between the particle and $Q$ is in absolute value at most $|\theta_0|$, or there was a moment $t \le T$ where $\theta_t=0$ (and we already detected $Q$ by time $t$). Note that by symmetry of the angular movement, $\PP_{x_0}(\mathcal{C} \mid \mathcal{B})=\PP_{x_0}(\mathcal{C}) \ge 1/2$. Hence, $\PP_{x_0}(\mathcal{C} \cap \mathcal{B})=\Omega(\KR^{-2\alpha})$. To conclude observe first the following:
 If it is the case that $\overline{r}_0=c^*$, then conditional under $\mathcal{C} \cap \mathcal{B}$, by Lemma~\ref{lem:closetoorigin}, with constant probability, in time $O(1)$ the particle $Q$ is detected, and since $\frac12\ss+O(1) \le \ss$, we are done. In particular, if 
we had $|\theta_0| > \pi/2 - c$ for some sufficiently small $c > 0$, then, since $|\theta_0| \le 3\KR e^{-\beta r_0}$, we must have $r_0 < c^*$, and therefore clearly also $\overline{r}_0=c^*$. If on the other hand we had  $|\theta_0| \le \pi/2 - c$ for some arbitrarily small $c > 0$ and also $\overline{r}_0=r_0-2\log \KR$, then observe that since $\overline{r}_0 > c^*$ with $c^*$ large enough, we have
$\phi(\overline{r}_0) \ge e 1.99\KR e^{-r_0/2} \ge 3\KR e^{-r_0/2} \ge |\theta_0|$, and $Q$ is detected by time $T$, since $\phi(\overline{r}_0)$ corresponds to the maximum angle at the origin between a particle at radius $\overline{r}_0$ and $Q$ that guarantees detection. 
\end{proof}

The uniform lower bound of Theorem~\ref{thm:mixedLarge} now follows directly by combining Proposition~\ref{mixedLowerBoundschico} and Proposition~\ref{prop:uniformlowerboundN2}. The integral lower bound of Theorem~\ref{thm:mixedLarge} follows directly from Proposition~\ref{mixedLowerBoundschico} applied with $\KA$ being any fixed constant and the fact that $\mu(\dt)=\Omega(n\phis)$ (by considering the condition $|\theta_0| \le \phi(R)+\KA \phis$ only), and the proofs of the lower bounds of Theorem~\ref{thm:mixedSmall} and Theorem~\ref{thm:mixedLarge} are finished. 

\subsection{Upper bounds}\label{sec:Upper}

In this section we will show the corresponding upper bounds of Theorems~\ref{thm:mixedSmall} and~\ref{thm:mixedLarge}: that is, we show that particles initially placed outside $\dt$, with $\dt$ as in the corresponding theorems, have only a small chance of detecting the target by time $\ss$. For all values of $\ss$, we will show that
\[\int_{\ndt}\PP_{x_0}(T_{det}\leq \ss)d\mu(x_0)=O(\mu(\dt)),\]
thus establishing that the significant contribution 
to the tail probabilities comes from particles inside $\dt$, and for large values of $\ss$ we will show uniform upper bounds for the detection probability of every point outside $\dt$. 
In order to analyze the trajectory of a particle we will make use of the fact that the generator driving its motion can be separated into its radial and angular component, given respectively by
\[\Delta_{rad} = \frac{1}{2}\frac{\partial^2}{\partial r^2}+\frac{\alpha}{2}\frac{1}{\tanh(\alpha r)}\frac{\partial}{\partial r}\quad\text{and}\quad \Delta_{ang}=\frac{1}{2\sinh^2(\beta r)}\frac{\partial^2}{\partial\theta^2}.\]
Since the generator driving the radial part of the motion does not depend on $\theta$, our approach consists in sampling the radial component $\{r_s\}_{s\geq0}$ of the trajectories first and then, conditional under a given radial trajectory, sampling the angular component $\{\theta_s\}_{s\geq0}$. With this approach we can make use of the results obtained in Section~\ref{sec:radial} in order to study $\{r_s\}_{s\geq 0}$, while the angular component is distributed according to a time-changed Brownian motion $\theta_s=B_{\II{s}}$, where
\[\II{s}:=\int_0^s\cosech^{2}(\beta r_u)du\]
relates the radial trajectory to the angular variance of $\{\theta_s\}_{s\geq 0}$, as already seen in Section~\ref{sec:angular}. To be able to apply the insight obtained by studying each component separately, we will need to replace the hitting time of the target (which translates into hitting $B_Q(R)$ whose boundary defines a curve relating $r$ and $\theta$) by the exit time of a simpler set: let $x_0\in\ndt$ be any fixed initial location with $\theta_0\in(0,\pi)$  and define a box $\mathcal{B}(x_0)$ containing~$x_0$, of the form 
\begin{equation}\label{eqn:up:box}
\mathcal{B}(x_0):=[\hatr,R]\times[\theta_0-\hatt,\theta_0+\hatt]\subseteq \overline{B}_Q(R)
\end{equation}
where $\hatt$ and $\hatr$ will be defined later on. 
Since $\mathcal{B}(x_0)\subseteq\overline{B}_Q(R)$, it follows that in order for a particle starting at $x_0$ to detect the target, it must first exit the box through either its upper boundary or its side boundaries. Denoting by $\trad$ and $\tang$ the respective hitting times of said boundaries, this implies that $\trad\wedge \tang\leq T_{det}$, and we can obtain the bound
\begin{equation}\label{boxboundsNew}\PP_{x_0}(T_{det}\leq \ss)\,\leq\,\PP_{x_0}(\trad\leq \ss)+\PP_{x_0}(\tang\leq \ss<\trad).\end{equation}
The advantage of addressing the exit time of $\mathcal{B}(x_0)$ instead of $T_{det}$ is straightforward; the first event $\{T_0^{rad}\le \ss\}$ is independent of the angular component of the trajectory, allowing us to bound from above its probability with the tools developed in Section~\ref{sec:radial}, while the treatment of the second event $\{T_0^{ang}\le \ss<T_0^{rad}\}$ will follow from standard results on Brownian motion and the control of~$\II{\ss}$. The following result allows us to bound the second term on the right-hand side of~\eqref{boxboundsNew}:
\begin{proposition}\label{prop:mainuppermixed}
Define $J(r):=\cosech^{-2}(\beta r)$. Then,
\[\PP_{x_0}(\tang\leq \ss\leq\trad)\,\leq\,4\Phi\Big({-}\frac{\hatt}{\sqrt{\ss\JJ(\hatr)}}\Big).\]
where $\Phi$ stands for the error function. Furthermore, 
\[\PP_{x_0}(\tang\leq \ss)\leq \sqrt{\frac{2}{\pi}}\int_{0}^{\infty}\frac{\hatt}{\sigma^{3/2}}e^{-\frac{\hatt^2}{2\sigma}}\PP_{x_0}\big(\II{\ss}\geq\sigma\big)d\sigma.
\]
\end{proposition}
\begin{proof}
Denote by $\PP_{x_0}(\cdot\mid\{r_u\}_{u\geq0})$ the law of the angular process given a realization of its radial component. Given such a realization we already know that the trajectory $\{\theta_u\}_{u\geq0}$ is equal in law to that of the time-changed Brownian motion $\{B_{\II{u}}+\theta_0\}_{u\geq0}$, for which $\tang$ is equal to the exit time of $B_{\II{u}}$ from $[-\hatt,\hatt]$. Since the exit time of a Brownian motion is well known, we proceed as in Section~\ref{sec:angular} using the reflection principle to obtain 
\[\PP_{x_0}(\tang\leq \ss\,|\{r_u\}_{u\geq0})\leq 4\PP_{x_0}(B_{\II{\ss}}\leq -\hatt\mid\{r_u\}_{u\geq 0})=4\Phi\Big({-}\frac{\hatt}{\sqrt{\II{\ss}}}\Big)\]
where the factor 4 is obtained since the probability of exiting by hitting one of the angles is twice the probability of hitting any of the two angles, which is twice the probability of hitting one fixed angle. Taking expectations with respect to the law of $\{r_u\}_{u\geq 0}$ we obtain
\[\PP_{x_0}(\tang\leq \ss\leq\trad)\leq 4\EE_{x_0}\Big(\Phi\Big({-}\frac{\hatt}{\sqrt{\II{\ss}}}\Big){\bf 1}_{\{\ss\leq\trad\}}\Big),\]
where the expression within the expectation depends on the radial movement alone. To apply this last bound observe that for any realization of $\{r_u\}_{u\geq0}$ such that $\trad\geq\ss$ we have that $\inf_{0\leq u\leq\ss}r_u\geq \hatr$, so $\II{\ss}\leq \ss\JJ(\hatr)$, which gives the first part of the proposition's statement. For the second part of the proposition we can make the same computation as before without taking into account the event $\{\trad\geq\ss\}$, giving 
\[\PP_{x_0}(\tang\leq \ss)\,\leq\,4\EE_{x_0}\Big(\Phi\Big({-}\frac{\hatt}{\sqrt{\II{\ss}}}\Big)\Big)\,=\,-4\int_{0}^\infty\Phi\Big({-}\frac{\hatt}{\sqrt{\sigma}}\Big)d\PP_{x_0}\big(\II{\ss}\geq\sigma\big),\]
and the result follows after using integration by parts, since $\frac{d}{d\sigma}\Phi(-\frac{\hatt}{\sqrt{\sigma}})=\frac{1}{2\sqrt{2\pi}}\frac{\hatt}{\sigma^{3/2}}e^{-\frac{\hatt^2}{2\sigma}}$.
 \end{proof}
 
 \smallskip
 
On a high level, in the proposition above, the bound for $\PP_{x_0}(\tang\leq \ss<\trad)$ will be useful as long as $\ss=O(1)$ since in this case the inequality $\ss \JJ(\hatr)\leq \II{\ss}$ is not too loose. However, when $\ss=\omega(1)$ we need to control this term using the second part of Proposition~\ref{prop:mainuppermixed}, which relies heavily on bounding probabilities of the form $\PP(\II{\ss}\geq\sigma)$. To provide these bounds observe that when $\inf_{0\leq u\leq\ss}r_u$ is not too close to zero the drift function is almost constant so the trajectories of $\{r_u\}_{u\geq0}$ resemble those of a Brownian motion with drift $\tfrac{\alpha}{2}$ (with a reflecting boundary at $R$). Even further, on the event where $\inf_{0\leq u\leq\ss}r_u$ is not too close to zero we also have
\begin{equation}\label{eq:approximation}\II{\ss}\approx\int_0^\ss e^{-2\beta r_u}du,\end{equation}
where  
the integral appearing above on the right-hand side has been widely studied in the case when $\{r_u\}_{u\geq 0}$ is an (unbounded) Brownian motion with drift (see~\cite{matsumoto2005,Yor1992,Feng2020,dufresne}), revealing that their distributions are heavy-tailed with the tail exponent given analytically in terms of $\beta$ and the drift of $\{r_u\}_{u\geq 0}$. In order to make use of known results  we must first address the change in behavior arising from the reflecting boundary of our process. To do so, we consider an auxiliary process $\{\widetilde{r}_u\}_{u\geq 0}$ akin to the one introduced in the proof of Proposition~\ref{prop:rad-upperBnd}:
    \begin{itemize}
	\item $\{\widetilde{r}_u\}_{u\geq 0}$ begins at $r_0$ and evolves according to $\Delta_{rad}$ up until hitting $R$.
	\item Every time $\{\widetilde{r}_u\}_{u\geq 0}$ hits $R$, it is immediately restarted at $R-1$ and continues to evolve as before.
\end{itemize}
It is not hard to see that $\{\widetilde{r}_u\}_{u\geq 0}$
 is stochastically dominated by $\{r_u\}_{u\geq 0}$, in the sense that the radius of the auxiliary process is always smaller. Thus, in particular, $\II{\ss}\leq \tII{\ss} := \int_0^s\cosech^{2}(\beta \widetilde{r}_u)du$, and hence it will be enough to bound probabilities of the form $\PP(\tII{\ss}\geq\sigma)$ from above. Now, from the definition of $\widetilde{r}_u$, it is natural to use the subsequent hitting times of $R$, say $0<T_R^{(1)}<T_R^{(2)}<\ldots$, to divide the trajectory of $\{\widetilde{r}_u\}_{u\geq 0}$ into excursions from $R-1$ to $R$ (or from $r_0$ to $R$, in the case of the first one), giving
\[\tII{\ss}\,\leq\,\int_0^{T^{(1)}_R}\cosech^{2}(\beta\widetilde{r}_u)du\,+\,\sum_{i=1}^{M(\ss)}\int_{T^{(i)}_R}^{T^{(i+1)}_R}\cosech^{2}(\beta\widetilde{r}_u)du,\]
where $M(\ss)$ is a random variable equal to the amount of times $\{\widetilde{r}_u\}_{u\geq 0}$ has hit $R$ by time~$\ss$. Let $\{\widetilde{r}^{(0)}_u\}_{u\geq0},\{\widetilde{r}^{(1)}_u\}_{u\geq0},\{\widetilde{r}^{(2)}_u\}_{u\geq0},\ldots$ be independent random diffusion processes (that will correspond to the different excursions), each evolving on $(0,\infty)$ according to the generator $\Delta_{rad}$ without the reflecting boundary at $R$, and such that $\{\widetilde{r}^{(0)}_u\}_{u\geq0}$ starts at $r_0$ while the rest of the $\{\widetilde{r}^{(i)}_u\}_{u\geq0}$ start at $R-1$. From the definition of the auxiliary process it follows that within each time interval of the form $[T^{(i)}_R,T^{(i+1)}_R]$ (or $[0,T^{(1)}_R]$ in the case of the first excursion) the trajectory of $\{\widetilde{r}_u\}_{u\geq0}$ is equal to the one of $\{\widetilde{r}^{(i)}_{u'}\}_{u'\geq0}$ for $u'\leq T^{(i+1)}_R-T^{(i)}_R$. Hence,
\[\tII{\ss}\,\leq\,\int_0^{T^{(1)}_R}\cosech^{2}(\beta\widetilde{r}^{(0)}_u)du\,+\,\sum_{i=1}^{M(\ss)}\int_{0}^{T^{(i+1)}_R-T^{(i)}_R}\cosech^{2}(\beta\widetilde{r}^{(i)}_u)du\,\leq\,\sum_{i=0}^{M(\ss)}\tI^{(i)},\]
where for $i=0, \ldots, M(\ss)$
\[\tI^{(i)}:=\int_{0}^{\infty}\cosech^{2}(\beta\widetilde{r}^{(i)}_u)du.\]
Observe that the bound on $\tII{\ss}$ involves a sum of $M(\ss)+1$ random variables. Intuitively speaking the number of times the process hits $R$ should be proportional to $\ss$, so we introduce an auxiliary parameter $v_0>0$ depending on $x_0$, to be fixed later, so that the event $M(\ss)>\lceil v_0\ss\rceil$ occurs with a very small probability. Using this new parameter we obtain
\begin{equation}\label{eq:division}\PP_{x_0}(\II{\ss}\geq\sigma)\,\leq\,\PP_{x_0}\big(M(\ss)>\lceil v_0\ss\rceil \big)+\PP_{x_0}\big(\tI^{(0)}\geq\tfrac12\sigma\big)+ \PP_{x_0}\Big(\sum_{i=1}^{\lceil v_0\ss\rceil}\tI^{(i)}\geq\tfrac12\sigma\Big).\end{equation}
The advantage of the bound above is that it involves the sum of independent random variables~$\tI^{(i)}$ which, aside from $\tI^{(0)}$, are also identically distributed. The bound also gives valuable intuition regarding the detection of the target: roughly speaking, in order for detection to take place either $x_0$ must be sufficiently close to the origin so that the angular variance coming from $\tI^{(0)}$ becomes significant, or $\ss$ must be large enough so that many excursions starting near the boundary occur, allowing for the contribution $\sum_{i=1}^{M(\ss)}\tI^{(i)}$ to be large enough. As mentioned above (the discussion following~\eqref{eq:approximation}), the distributions of the $\tI^{(i)}$ variables are heavy-tailed. In order to control their sum we make use of the following lemma, whose proof mimics closely what was done in \cite{Omelchenko2019}: since our result is slightly more precise than the one in~\cite{Omelchenko2019}, we provide it in the appendix for the sake of completeness. 

\begin{lemma}\label{lem:cotapower}
Let $S_m:=\sum_{i=1}^m Z_i$ where $\{Z_i\}_{i\in\NN}$  is a sequence of i.i.d.~absolutely continuous random variables taking values in $[1,\infty)$ such that there are $V,\gamma>0$ for which for all $x\geq 0$
\[1-F_{Z_i}(x)=\PP(Z_i\geq x)\leq Vx^{-\gamma}.\]
Then, there are $\auxc,\auxl>0$ depending on $V$, $\gamma$ and $\EE(Z_1)$  (if it exists) alone such that:
\begin{itemize}
    \item If $\gamma<1$ and $L>\auxl$, then
    $\displaystyle\PP(S_m\geq Lm^{\frac{1}{\gamma}})\leq \auxc L^{-\gamma}$.
    \item If $\gamma=1$ and $L>\auxl$, then $\displaystyle\PP(S_m\geq Lm\log(m))\leq \Big(\frac{\auxc}{L\log(m)}\Big)^{1-\frac{\auxl}{L}}$.
    \item If $\gamma>1$ and  $L>\auxl$, then
    $\displaystyle\PP(S_m\geq Lm)\,\leq\,\auxc L^{-\gamma}m^{-((\gamma-1)\wedge\frac{\gamma}{2})}$.
\end{itemize}
\end{lemma}
We can now prove the following result, which will allow us to control $\PP_{x_0}(\II{\ss}\ge\sigma)$:
\begin{proposition}\label{prop:merged-varianza}
For any $0<c<1$ there are $\auxc',\auxl'$ large depending on $\alpha$ and $\beta$ alone, such that for any $\ss\geq c$, $v_0>\frac{4}{c}$, and $x_0\in B_O(R)$ with $r_0>1$, defining $\sigma_0$ as
\[
\sigma_0 := \begin{cases}
\auxl'e^{-2\beta R}(v_0\ss)^{1\vee\frac{2\beta}{\alpha}}, &
\text{if $\alpha \neq 2\beta$,} \\
\auxl'e^{-2\beta R}v_0\ss\log(v_0\ss), &
\text{if $\alpha=2\beta$,}
\end{cases}
\]
the following statements hold for all $\sigma>\sigma_0$ and all $R$ sufficiently large:
\begin{enumerate}[(i)]
\item\label{prop:merged-varianza:itm1} $\PP_{x_0}(M(\ss)>\lceil v_0\ss\rceil )\leq e^{-\frac{1}{16} v_0^2\ss}$,
\item\label{prop:merged-varianza:itm2} $\PP_{x_0}(\tI^{(0)}\geq\tfrac12\sigma)\leq \frac{\log(\tanh(\alpha r_0/2))}{\log(\tanh(\alpha/2))}+\auxc'\sigma^{-\frac{\alpha}{2\beta}} e^{-\alpha r_0}$, and 
\item\label{prop:merged-varianza:itm3} $\PP_{x_0}\big(\sum_{i=1}^{\lceil v_0\ss\rceil}\tI^{(i)}\geq\tfrac12\sigma\big)\leq 2v_0\ss\frac{\log(\tanh(\alpha ((R-1)/2))}{\log(\tanh(\alpha/2))}+\big(\auxc'(v_0\ss)^{1\vee\frac{\alpha}{4\beta}}\sigma^{-\frac{\alpha}{2\beta}}e^{-\alpha R}\big)^{1-\mathfrak{e}}$
where
\[\mathfrak{e}:=\begin{cases}
0, &\mbox{ if }\alpha\neq 2\beta, \\[3pt]\frac{1}{\sigma e^{2\beta R}}\auxl'v_0\ss\log(v_0\ss), &\mbox{ if }\alpha=2\beta.
\end{cases}\]
\end{enumerate}
\end{proposition}
\begin{proof}
We begin with the upper bound for the term $\PP_{x_0}(M(\ss)>\lceil v_0\ss\rceil)$. By definition, the event $\{M(\ss)>  \lceil v_0\ss \rceil\}$ is equal to $\{T^{(\lceil v_0\ss\rceil)}_R\leq\ss\}$. Writing $T^{(\lceil v_0\ss\rceil)}_R=T^{(1)}_R+\sum_{i=2}^{\lceil v_0\ss\rceil}(T^{(i)}_R-T^{(i-1)}_R)$,  by Markov's inequality, for any $\lambda>0$ we deduce
\[\PP_{x_0}\big(M(\ss)>\lceil v_0\ss\rceil\big)\,\leq\,\PP_{x_0}\Big(\sum_{i=2}^{\lceil v_0\ss\rceil}(T^{(i)}_R-T^{(i-1)}_R)\leq\ss\Big)\,\leq\,e^{\lambda \ss}\EE_{x_0}\Big(\exp\big({-}\lambda\sum_{i=2}^{\lceil v_0\ss\rceil}(T^{(i)}_R-T^{(i-1)}_R)\big)\Big)\]
where the differences $T^{(i)}_R-T^{(i-1)}_R$ are i.i.d.~random variables which indicate the time it takes for excursion $\widetilde{r}^{(i-1)}$ starting at $R-1$ to reach $R$. Since each $\{\widetilde{r}^{(i)}_u\}_{u\geq 0}$ is stochastically bounded from below by a Brownian motion with constant drift $\tfrac{\alpha}{2}$, we can use Formula 2.0.1 in~\cite{Borodin2002} to obtain
\[\PP_{x_0}\big(M(\ss)>\lceil v_0\ss\rceil\big)\,\leq\,e^{\lambda \ss}\EE\big(e^{-\lambda(T^{(2)}_R-T^{(1)}_R)}\big)^{\lceil v_0\ss\rceil-1}\,\leq\,e^{\lambda \ss}\big( e^{\frac{\alpha}{2}-\sqrt{\frac{\alpha^2}{4}+2\lambda}}\big)^{\lceil v_0\ss\rceil-1}.\]
The exponent $\lambda\ss+(\lceil v_0\ss\rceil-1)(\frac{\alpha}{2}-\sqrt{\frac{\alpha^2}{4}+2\lambda})$ appearing in the last term is minimized as a function of $\lambda$ when $\ss\sqrt{\tfrac{\alpha^2}{4}+2\lambda}=\lceil v_0\ss\rceil-1$, which defines a positive $\lambda$ since $v_0\ss>4$ and $\alpha,c<1$, giving an expression for the exponent of the right-hand side of the form 
\[-\frac{1}{2\ss}(\lceil v_0\ss\rceil-1)^2-\frac{1}{8}\alpha^2\ss+\frac{1}{2}\alpha(\lceil v_0\ss\rceil-1)\leq\frac12(\lceil v_0\ss\rceil-1)\Big(\alpha-\frac{1}{\ss}(\lceil v_0\ss\rceil-1)\Big)\,\leq\,-\frac{1}{16}\,v_0^2\ss\]
where in the last inequality we used that $\lceil v_0\ss\rceil-1\geq \frac12v_0\ss$ and that $\alpha-\frac12v_0\leq -\frac14 v_0$. This proves the bound in~\eqref{prop:merged-varianza:itm1}. In order to control the probabilities appearing in~\eqref{prop:merged-varianza:itm2} and~\eqref{prop:merged-varianza:itm3} it is convenient from the point of view of computations to work under the assumption that the $\widetilde{r}^{(i)}$ processes never get too close to $0$. To do so observe that by our assumption of $r_0>1$, all the $\widetilde{r}^{(i)}$ start in $(1,\infty)$, and denote by $\tau^{(i)}_1$ the hitting time of radius $1$ by $\widetilde{r}^{(i)}$. Observe that \[\PP_{x_0}(\tI^{(0)}\geq\tfrac12\sigma)\,\leq\,\PP_{x_0}(\tau_{1}^{(0)}<\infty)+\PP_{x_0}(\tI^{(0)}\geq\tfrac12\sigma, \tau^{(0)}_{1}=\infty).\]
Using Part~\eqref{radial:itm:phi3} from Lemma~\ref{lemmaradial}, with $\yabs_0:=1$ and $\auxY:=R$,  we obtain 
\[\PP_{x_0}(\tau_{1}^{(0)}<\infty)=\lim_{R \to \infty}\PP_{x_0}(\tau_{1}^{(0)}<T_R)=\frac{\log(\tanh(\alpha r_0/2))}{\log(\tanh(\alpha/2))},\]
which gives the first term of the bound for $R$ (and thus $n$) sufficiently large. For the second term, observe that on the event $\tau^{(0)}_1=\infty$ we have $\widetilde{r}^{(0)}_u>1$ for all $u\geq0$ so there is some constant $C$ depending on $\beta$ alone such that for all $u\geq 0$, we have $\cosech^2(\beta \widetilde{r}^{(0)}_u)\leq Ce^{-2\beta \widetilde{r}^{(0)}_u}$ and hence
\[\PP_{x_0}(\tI^{(0)}\geq\tfrac12\sigma,\tau^{(0)}_{1}=\infty)\,\leq\,\PP_{x_0}\Big(C\int_0^{\infty}e^{-2\beta\widetilde{r}^{(0)}_u}du\geq\tfrac12\sigma\Big).\]
Now, notice that $\{\widetilde{r}_u^{(0)}-r_0\}_{u\geq0}$ 
is stochastically bounded from below by $X_u$, a Brownian motion with constant drift $\tfrac{\alpha}{2}$ so we can bound the integral within the last term from above by $\int_0^{\infty}e^{-2\beta(X_u+r_0)}du$. This particular functional of Brownian motion was studied in~\cite[Proposition~4.4.4]{dufresne}, where it was shown that
\[\int_0^\infty e^{-2\beta X_u}du\,=\,W,\quad\text{where }W\text{ is a positive r.v.~with }\quad f_{W}(x):=\frac{(2\beta^2)^\frac{\alpha}{2\beta}}{\Gamma(\frac{\alpha}{2\beta})}x^{-\frac{\alpha}{2\beta}-1}e^{-\frac{2\beta^2}{x}}\]
that is, $W$ is distributed according to an inverse gamma distribution with parameters $\frac{\alpha}{2\beta}$ and~$2\beta^2$. From the density of the variable $W$ above the only relevant feature for our purposes is its heavy tail: in order to ease calculations we will bound $W$ stochastically from above by a random variable $Z$ following a Pareto distribution with a sufficiently large scale $\omega$ depending on $\alpha$ and $\beta$ alone (in particular we will assume $\omega>1$), and shape $\frac{\alpha}{2\beta}$. That is, $W \preccurlyeq Z$ with
\[f_{Z}(x)\,=\,\frac{2\beta}{\alpha\omega}\Big(\frac{\omega}{x}\Big)^{\frac{\alpha}{2\beta}+1}{\bf 1}_{(\omega,\infty)}(x).\]
Using this auxiliary variable we readily obtain 
\[\PP_{x_0}\Big(C\int_0^{\infty}e^{-2\beta\widetilde{r}^{(0)}_u}du\geq\tfrac12\sigma\Big)\leq\PP\big(Z\geq\tfrac{\sigma}{2C}e^{2\beta r_0}\big)\leq\Big(\frac{\sigma e^{2\beta r_0}}{2C\omega}\Big)^{-\frac{\alpha}{2\beta}},\]
giving the last term of~\eqref{prop:merged-varianza:itm2}. To obtain the bound in~\eqref{prop:merged-varianza:itm3}, define the event $\mathcal{E}:=\{\exists 1\leq i\leq \lceil v_0\ss\rceil,\,\tau_1^{(i)}<\infty\}$ where $\tau_1^{(i)}$ was defined previously as the hitting time of $r=1$ on the $i$-th excursion. By the same analysis as in the previous paragraph, we get 
\begin{align*}\PP_{x_0}\Big(\sum_{i=1}^{\lceil v_0\ss\rceil}\tI^{(i)}\geq\tfrac12\sigma\Big)&\leq\,\PP_{x_0}(\mathcal{E})+\PP_{x_0}\Big(\sum_{i=1}^{\lceil v_0\ss\rceil}\tI^{(i)}\geq\tfrac12\sigma,\,\overline{\mathcal{E}}\Big)\\&\leq\,\lceil v_0\ss\rceil\PP_{x_0}(\tau_{1}^{(1)}<\infty)+\PP_{x_0}\Big(C\sum_{i=1}^{\lceil v_0\ss\rceil}\int_0^{\infty}e^{-2\beta\widetilde{r}^{(i)}_u}du\geq\tfrac12\sigma\Big)\\&\leq\,\lceil v_0\ss\rceil\PP_{x_0}(\tau_{1}^{(1)}<\infty)+\PP_{x_0}\Big(\sum_{i=1}^{\lceil v_0\ss\rceil}Z^{(i)}\geq\tfrac{\sigma}{2C}e^{2\beta(R-1)}\Big)\end{align*}
where in this case we have used that each $\widetilde{r}^{(i)}$ begins at $R-1$ and the $Z^{(i)}$'s stand for i.i.d.~Pareto random variables with scale $\omega$ and shape $\frac{\alpha}{2\beta}$ as in the previous case. For the first term we use the same treatment as with $\PP_{x_0}(\tau_{1}^{(0)}<\infty)$ to obtain
\[\lceil v_0\ss\rceil\PP_{x_0}(\tau_{1}^{(1)}<\infty)\,\leq\,2v_0\ss\frac{\log(\tanh(\alpha (R-1)/2))}{\log(\tanh(\alpha/2))}\]
where we have used that in the first excursion the starting radius is $R-1$, and that $v_0\ss\geq 4$. For the second term recall that the $Z^{(i)}$ variables are bounded from below by $\omega\geq 1$, and that $1-F_Z(x)=(\frac{\omega}{x})^{\frac{\alpha}{2\beta}}$ for any $x\geq\omega$. Hence, we can apply Lemma~\ref{lem:cotapower} with $m:=\lceil v_0\ss\rceil$ and $S_m:=\sum_{i=1}^{\lceil v_0\ss\rceil} Z^{(i)}$ depending on the value of $\gamma:=\frac{\alpha}{2\beta}$ as follows:
\begin{itemize}
\item If $\frac{\alpha}{2\beta}<1$, then we can take $L:=\tfrac{\sigma}{2C}\lceil v_0\ss\rceil^{-\frac{2\beta}{\alpha}}e^{2\beta(R-1)}$ in the lemma so that 
\[\PP_{x_0}\Big(\sum_{i=1}^{\lceil v_0\ss\rceil}Z^{(i)}\ge\tfrac{\sigma}{2C}e^{2\beta(R-1)}\Big)=\PP_{x_0}\Big(S_{\lceil v_0\ss\rceil}\ge L \lceil v_0\ss\rceil^{\frac{2\beta}{\alpha}}\Big)\leq \auxc L^{-\frac{\alpha}{2\beta}}\]
as soon as $L> \auxl$ for some $\auxl$ and $\auxc$ depending on $\alpha,\beta$ and $\omega$ alone. Observing that $L^{-\frac{\alpha}{2\beta}}=(2C)^{\frac{\alpha}{2\beta}}\lceil v_0\ss\rceil\sigma^{-\frac{\alpha}{2\beta}}e^{-\alpha (R-1)}$ and that the condition $L> \auxl$ is equivalent to
$\sigma > 2\auxl C\lceil v_0\ss\rceil^{\frac{2\beta}{\alpha}}e^{-2\beta (R-1)}$, the result follows by defining $\auxl'$ and $\auxc'$ accordingly.
\item If $\frac{\alpha}{2\beta}=1$ we can take $L:=\tfrac{\sigma}{2C}(\lceil v_0\ss\rceil\log \lceil v_0\ss\rceil)^{-1} e^{2\beta(R-1)}$ in the lemma so that
\[\PP_{x_0}\Big(\sum_{i=1}^{\lceil v_0\ss\rceil}Z^{(i)}\ge\tfrac{\sigma}{2C}e^{2\beta(R-1)}\Big)=\PP_{x_0}\big(S_{\lceil v_0\ss\rceil}\ge L \lceil v_0\ss\rceil\log \lceil v_0\ss\rceil\big)\leq \Big(\frac{\auxc}{L\log \lceil v_0\ss\rceil}\Big)^{1-\frac{\auxl}{L}}\]
as soon as $L>\auxl$ for some $\auxl$ and $\auxc$ depending on $\alpha,\beta$ and $\omega$ alone, which is satisfied since by hypothesis $\sigma > \auxl'e^{-2\beta R}v_0\ss\log(v_0\ss)$ and we can choose $\auxl'>2Ce^{2\beta}\auxl$. Since $\alpha=2\beta$, we have $\tfrac{\auxc}{L\log\lceil v_0\ss\rceil}=\frac{\auxc'}{\sigma}e^{-\alpha R}\lceil v_0\ss\rceil $ for $\auxc':=2C\auxc e^{\alpha}$, giving the result.
\item Finally, assume that $\frac{\alpha}{2\beta}>1$ and take $L:=\tfrac{\sigma}{2C\lceil v_0\ss\rceil} e^{2\beta(R-1)}$ in the lemma so that
\[\PP_{x_0}\Big(\sum_{i=1}^{\lceil v_0\ss\rceil}Z^{(i)}\ge\tfrac{\sigma}{2C}e^{2\beta(R-1)}\Big)=\PP_{x_0}\big(S_{\lceil v_0\ss\rceil}\ge L \lceil v_0\ss\rceil\big)\leq \auxc L^{-\frac{\alpha}{2\beta}}\lceil v_0\ss\rceil^{-(\frac{\alpha}{2\beta}-1\wedge\frac{\alpha}{4\beta})}\]
as soon as $L> \auxl$ for some $\auxl$ and $\auxc$ depending on $\alpha,\beta$ and $\omega$ alone. The condition follows directly from the hypothesis $\sigma>\auxl' e^{-2\beta R}(v_0\ss)$ for $\auxl'$ adequately chosen, and the bound in~\eqref{prop:merged-varianza:itm3} is obtained from the previous inequality by noticing that $L^{-\frac{\alpha}{2\beta}}=(2C\lceil v_0\ss\rceil)^{\frac{\alpha}{2\beta}}\sigma^{-\frac{\alpha}{2\beta}}e^{-\alpha (R-1)}$ and defining $\auxc'$ accordingly.
\end{itemize}
\end{proof}

\subsubsection{The case $\ss$ small}
In this subsection we obtain upper bound for $\int_{\ndt}\PP_{x_0}(T_{det}\leq \ss)d\mu(x_0)$ under the assumption of $\beta\leq1/2$ and $\ss=O(1)$ as required by Theorem~\ref{thm:mixedSmall}. Since for $\beta=1/2$ we always have $\ss=\Omega(1)$ and the next subsection deals with all such values, here we even assume $\beta < 1/2$. The main result in this section is the following:
\begin{proposition}\label{generalschico}
Let $\beta<\frac{1}{2}$. If $\dt$ is as defined in Theorem~\ref{thm:mixedSmall} and $\ss=\Omega((e^{\beta R}/n)^{2})\cap O(1)$, then
\[\int_{\ndt}\P_{x_0}(T_{det}\leq \ss)d\mu(x_0) 
\;=\;O\big(ne^{-\beta R}\sqrt{\ss}\big).\]
\end{proposition}
\begin{proof}
Notice first that for any $C'>0$, Lemma~\ref{lem:muBall} gives $\mu(B_O(C'))=O(ne^{-\alpha R})=o(1)$ so the contribution of points in $\ndt\cap B_O(C')$ to the upper bound is $o(ne^{-\beta R}\sqrt{\ss})$ from our assumption $\ss=\Omega((e^{2\beta R}/n)^{2})$, and hence such points can be neglected. Take then $C'>0$ large (to be fixed later) so we only need to consider particles initially located at points with radial component in $[C',R]$. For such points we follow the proof argument described in the previous section, where we bound the detection time of the target by the exit time of the box
$\mathcal{B}(x_0)$ for $\hatt$ and $\hatr$ given in the next sentence, allowing us to bound $\PP_{x_0}(T_{det}\leq\ss)$ as in~\eqref{boxboundsNew}. More in detail, given $x_0\in\ndt$ we construct the box $\mathcal{B}(x_0)$ by choosing 
\[\hatt:=\tfrac{1}{3}(|\theta_0|-\phi(r_0))\qquad\text{and}\qquad\hatr:=r_0-\log\big(1+(|\theta_0|-\phi(r_0))e^{\beta r_0}\big).\]
Notice first that $\hatt>0$, $\hatr<r_0$, and assuming $C'>13$ we also have $\hatr>\frac13 r_0$: indeed, this last statement is equivalent to 
$e^{\frac23 r_0}>1+(|\theta_0|-\phi(r_0))e^{\beta r_0}$
which is satisfied for $r_0>2/(\frac{2}{3}-\beta)+1\geq 13$ since $\beta\leq1/2$ and $|\theta_0|\leq\pi$. The previous inequalities show that the box is well defined and contain the point $(r_0,\theta_0)$, but we still need to prove that $\mathcal{B}(x_0)\subseteq \overline{B}_Q(R)$. To do so 
assume, without loss of generality, that $\theta_0>0$  and 
observe that it is enough to show that the upper-left corner of the box, $(\hatr,\theta_0-\hatt)$ is in $\overline{B}_Q(R)$, or equivalently, that $\phi(\hatr)<\theta_0-\hatt=\frac{2}{3}(\theta_0-\phi(r_0))+\phi(r_0)$. Using that $\hatr<r_0$ and Part~\eqref{itm:phi1} of Lemma~\ref{lem:phi} we obtain that 
\begin{align*}\cos(\phi(r_0))-\cos(\phi(\hatr))&=(\tanh(\tfrac{r_0}{2})-\tanh(\tfrac{\hatr}{2}))\coth(R)=\coth (R)\frac{2(e^{r_0}-e^{\hatr})}{(e^{r_0}+1)(e^{\hatr}+1)}\\[3pt]&\leq4(e^{-\hatr}-e^{-r_0})=4e^{(\beta-1)r_0}(|\theta_0|-\phi(r_0)). 
\end{align*}
Moreover, for $C'$ sufficiently large, since $\cos(x)-\cos(x')=2\sin(\frac12(x+x'))\sin(\frac12(x-x'))$, 
\[\cos(\phi(r_0))-\cos(\phi(\hatr))\ge 2\sin(\tfrac12\phi(r_0))\sin(\tfrac12(\phi(\hatr)-\phi(r_0)))\geq\tfrac18\phi(r_0)(\phi(\hatr)-\phi(r_0)).\]
Using Part~\eqref{itm:phi4} of Lemma~\ref{lem:phi} we have $\phi(r_0)\geq e^{-\frac{r_0}{2}}$ for $r_0>C'$ so we conclude that 
\[\phi(\hatr)-\phi(r_0)\,\leq\,32e^{(\beta-\frac{1}{2})r_0}(|\theta_0|-\phi(r_0))\]
and since $\beta<\frac{1}{2}$, taking $C'$ sufficiently large the last term is bounded by $\frac{2}{3}(\theta_0-\phi(r_0))$ which proves that  $\mathcal{B}(x_0)\subseteq \overline{B}_Q(R)$ for all such $r_0$. 

Now, we follow the proof strategy presented in the previous section by splitting the detection probability as in~\eqref{boxboundsNew} and using the first bound in Proposition~\ref{prop:mainuppermixed}, which by setting $\Omega:=\ndt\cap\overline{B}_O(C')$ and recalling that $\JJ(\hatr):=\cosech^{2}(\beta\hatr)$, gives
\begin{equation}\label{eq:mainosmall}
    \int_{\Omega}\P_{x_0}(T_{det}\leq\ss)d\mu(x_0) \leq \int_{\Omega} \PP_{x_0}(T_{0}^{rad}\leq \ss)d\mu(x_0) +\int_{\Omega}4\Phi\Big({-}\frac{\hatt}{\sqrt{\ss \JJ(\hatr)}}\Big)d\mu(x_0)
\end{equation}
where $T_0^{rad}$ is the hitting time of $\hatr$. For the first term on the right of~\eqref{eq:mainosmall} we can stochastically dominate from above the trajectory of $\{r_s\}_{s\geq0}$ by that of a simple Brownian motion with a reflecting barrier at $R$ (thus removing the drift towards the boundary) which decreases the hitting time of $\hatr$. Since hitting $\hatr$ implies moving outside of an interval of length $2(r_0-\hatr)$, by the reflection principle, we obtain
\[\PP_{x_0}(T_{0}^{rad}\leq \ss)\,\leq\,4\Phi\Big({-}\frac{1}{\sqrt{\ss}}(r_0-\hatr)\Big)\,=\,4\Phi\Big({-}\frac{1}{\sqrt{\ss}}\log\Big(1+\frac{|\theta_0|-\phi(r_0)}{e^{-\beta r_0}}\Big)\Big).\]
For the term on the right we have
\[\int_{\phi(r_0)+\sqrt{\ss}e^{-\beta r_0}}^\infty\PP_{x_0}(T_{0}^{rad}\le\ss)d\theta_0=
4\sqrt{\ss}e^{-\beta r_0}\int_1^\infty \Phi\Big({-}\frac{1}{\sqrt{\ss}}\log(1+\sqrt{\ss}w)\Big)dw\]
but since the function $x\to\frac{1}{\sqrt{x}}\log(1+\sqrt{x}w)$ is decreasing and $\ss=O(1)$ we deduce that the term on the right-hand side is $O(\sqrt{\ss}e^{-\beta R})$, so
\begin{align*}
\int_{\Omega} \PP_{x_0}(T_{0}^{rad}\leq \ss)d\mu(x_0)&=\;O\Big(n\int_{C'}^R\int_{\phi(r_0)+\sqrt{\ss}e^{-\beta r_0}}^\infty\PP_{x_0}(T_{0}^{rad}\leq \ss)d\theta_0e^{-\alpha(R-r_0)}dr_0\Big)\\&=\;O\Big(n\int_{C'}^R\sqrt{\ss}e^{-\beta r_0}e^{-\alpha(R-r_0)}dr_0\Big)=O\big(n\sqrt{\ss}e^{-\beta R}\big)
\end{align*}
where in the last equality we have used that $\beta<\frac{1}{2}<\alpha$. Next, we must bound the second term on the right of~\eqref{eq:mainosmall}, for which we observe first that since $\hatr>\frac13 r_0$ and assuming that $C'>0$ is large we have $\sinh(\beta \hatr)>\frac{1}{4}e^{\beta\hatr}$ so that 
\[{-}\frac{\hatt}{\sqrt{\ss \JJ(\hatr)}}=\frac{\theta_0-\phi(r_0)}{3\sqrt{\ss J(\hatr)}}\leq {-}\frac{4(\theta_0-\phi(r_0))}{3\sqrt{\ss} e^{-\beta r_0}}\Big(1+\frac{\theta_0-\phi(r_0)}{e^{-\beta r_0}}\Big)^{-\beta}.\]
Evaluating at $\Phi(\cdot)$, integrating over $\Omega$ and using the change of variables $w:=\frac{\theta_0-\phi(r_0)}{\sqrt{\ss}e^{-\beta r_0}}$ gives
\[\int_{\Omega}\Phi\Big({-}\frac{\hatt}{\sqrt{\ss \JJ(\hatr)}}\Big)d\mu(x_0)\;=\;O\Big(n\int_{C'}^R\sqrt{\ss}e^{-\beta r_0}\Big(\int_{1}^\infty\Phi\Big({-}\frac{4w}{3(1+\sqrt{\ss}w)^{\beta}}\Big)dw\Big)e^{-\alpha(R-r_0)}dr_0\Big)\]
but since $\ss=O(1)$ we have $\int_{1}^\infty\Phi(-\tfrac{4w}{3(1+\sqrt{\ss}w)^{\beta}})dw=O(1)$ (it is finite since $\beta<1/2$) and hence
\[\int_{\Omega}\Phi\Big({-}\frac{\hatt}{\sqrt{\ss \JJ(\hatr)}}\Big)d\mu(x_0)\,=\,O\Big(n\sqrt{\ss}\int_{C'}^Re^{-\beta r_0+\alpha(R-r_0)}dr_0\Big)\,=\,O\big(n\sqrt{\ss}e^{-\beta R}\big)\]
where again we used that $\beta<\frac{1}{2}<\alpha$ in the last equality. 
\end{proof}

\subsubsection{The case $\ss$ large}
%
 In this subsection we will prove the upper bounds in Theorem~\ref{thm:mixedLarge} for both $\int_{\ndt}\PP_{x_0}(T_{det}\leq \ss)d\mu(x_0)$ as well as for $\sup_{x_0\in\ndt}\PP_{x_0}(T_{det}\leq \ss)$, where for convenience we use the shorthand notation $\ndt=\ndt(\KA)$ that will be used throughout this section. Recall that we may assume $\ss=\Omega(1)$ for this theorem. We handle this case with the help of Propositions~\ref{prop:mainuppermixed} and~\ref{prop:merged-varianza}, as well as with the results obtained in Section~\ref{sec:radial}.
The main result in this section is the following, which directly implies the corresponding upper bounds of Theorem~\ref{thm:mixedLarge}:
\begin{proposition}\label{generalsgrande}
    Let $\dt$ be as in Theorem~\ref{thm:mixedLarge}. Then there is a $C>0$ independent of $\ss$ and $n$ such that for $\ss=\Omega(1)$ satisfying the conditions of Theorem~\ref{thm:mixedLarge}, fixing $\KA:=C$ gives  
\[\int_{\ndt}\P_{x_0}(T_{det}\leq \ss)d\mu(x_0) 
\;=\;O(n\phis).\]
Furthermore, for $\KA\geq C$ and under the additional assumption $\ss=\omega(1)$ we have
\[
\sup_{x_0\in\ndt}\P_{x_0}(T_{det}\leq \ss) =
\begin{cases}
e^{-\Omega(\KA^2)},
& \text{ if $\alpha\geq2\beta$,} \\[2pt]
O(\KA^{-\alpha(\frac{1}{\beta}\vee 2)}),
& \text{ if $\alpha<2\beta$.} 
\end{cases}
\]
\end{proposition}
We begin by introducing some notation and deriving a few facts that will come in handy throughout this section. The points in $\partial \ndt$ with a positive angle define a curve in polar coordinates which can be parameterized by the radius as $(r,\gamma(r))$ with $\gamma(r):=\max\{\phi(R)+\KA\phis,\phi(r)+\KR e^{-(\beta\wedge\frac12) r}\}$, and where $r$ takes values between $r'$ and $R$, for $r'$ being the solution of (see Figure~\ref{fig:mixto}(b) for a depiction of $r'$)
\begin{equation}\label{eq:def:rprime}
\phi(r')+\KR e^{-(\beta\wedge\frac12) r'}=\pi.
\end{equation}
Observing that $\gamma$ is the maximum between two functions, a major role will be played by the intersecting point $(r'',\gamma(r''))$ which satisfies (see Figure~\ref{fig:mixto}(b) for a depiction of $r''$)
\begin{equation}\label{eq:def:r2prime}
\phi(R)+\KA\phis=\phi(r'')+\KR e^{-(\beta\wedge\frac12) r''}.
\end{equation}
\smallskip

Next we derive several inequalities that will be useful in the proof of Proposition~\ref{generalsgrande}:
\begin{fact}\label{fact:uppermix}
If the constant $C>0$ appearing in the statement of Proposition~\ref{generalsgrande} is sufficiently large, then under the assumption $\ss=\Omega(1)$ the following hold:
\begin{enumerate}[(i)]
\item\label{itm:upp1} $\phis=\Omega(\phi(R))$ and for all $0<r_0\leq R$, $e^{-(\beta\wedge\frac12) r_0}=\Omega(\phi(r_0))$. In particular, for any $C'>0$ by taking $C$ large we have $\KA\phis\geq C'\phi(R)$ and $\KR e^{-(\beta\wedge\frac12) r_0}\geq C'\phi(r_0)$.
\item\label{itm:upp2} $(\beta\wedge\frac12)^{-1}\log(\frac{1}{\pi} \KR)\leq r'\leq (\beta\wedge\frac12)^{-1}\log(\frac{2}{\pi}\KR)$.
\item\label{itm:upp3} $e^{-(\beta\wedge\frac12) r''}=\Theta(\tfrac{\KA}{\KR}\phis)$.
\item\label{itm:upp4} $\gamma(r)-\phi(r)$ is minimized at $r''$, and in particular 
$\gamma(r)-\phi(r)\geq(\KR e^{-(\beta\wedge \frac{1}{2})r}\vee\tfrac{1}{2}\KA\phis)$ for all $r'\leq r\leq R$.
\item\label{itm:upp5} For any fixed value of $\KA$ and $\KR$,  $\displaystyle\int_{r'}^R(\gamma(r)-\phi(r))e^{-\alpha(R-r)}dr=O(\phis).$
\end{enumerate}
\end{fact}
\begin{proof}
The first part follows directly from Part~\eqref{itm:phi4} of Lemma~\ref{lem:phi}, 
the second one from the definition of $r'$ and the fact that $0\leq\phi(\cdot)\leq\frac{\pi}{2}$, and 
the third part follows from~\eqref{itm:upp1} and the definition of $r''$. To prove part~\eqref{itm:upp4} observe that in $[r',r'']$ the function $\gamma(r)-\phi(r)=\KR e^{-(\beta\wedge \frac{1}{2})r}$ is decreasing while in $[r'',R]$ it is increasing (since $\gamma$ is constant and $\phi$ is decreasing) so the function is minimized at $r''$. The inequality $\gamma(r)-\phi(r)\geq(\KR e^{-(\beta\wedge \frac{1}{2})r}\vee\tfrac{1}{2}\KA\phis)$ follows from
\[\gamma(r'')-\phi(r'')=\KR e^{-(\beta\wedge \frac{1}{2})r''}\geq\tfrac{1}{2}(\KR e^{-(\beta\wedge \frac{1}{2})r''}+\phi(r''))=\tfrac{1}{2}(\KA\phis+\phi(R))\geq \tfrac{1}{2}\KA\phis.\]
To prove~\eqref{itm:upp5} we fix $\KA=C_A$ and $\KR=C_R$ equal to some constants and split the range of integration into two intervals, specifically $[r',r'']$ and $[r'',R]$. For the first one we have $\gamma(r)=\phi(r)+C_R e^{-(\beta\wedge\frac{1}{2})r}$ and since $\alpha>\frac{1}{2}$, using~\eqref{itm:upp3}, we get
\[\int_{r'}^{r''}(\gamma(r)-\phi(r))e^{-\alpha(R-r)}dr=O\big(e^{-\alpha R}e^{(\alpha-(\beta\wedge\frac{1}{2})r'')}\big)=O\big(\phis e^{-\alpha (R-r'')}\big)=O(\phis).
\]
To conclude the proof of~\eqref{itm:upp5} observe that for the second interval~\eqref{itm:upp1} gives
\[\int_{r''}^{R}(\gamma(r)-\phi(r))e^{-\alpha(R-r)}dr=O(\phis).
\]
\end{proof}

Following the strategy presented at the beginning of Section~\ref{sec:Upper}, for each $x_0\in\ndt$ with $\theta_0>0$ (the case $\theta_0<0$ is analogous) we define a box $\mathcal{B}(x_0)\subseteq \overline{B}_Q(R)$ where this time we set (see Figure~\ref{fig:mixto}(c) for an illustration of $\mathcal{B}(x_0)$)
\begin{equation}\label{def:hatthatr}
\hatt=\tfrac{2}{3}(\theta_0-\phi(r_0))\qquad\text{ and }\qquad \phi(\hatr)=\theta_0-\hatt.
\end{equation}
Since by Part~\eqref{itm:phi2} of Lemma~\ref{lem:phi} the function $\phi(\cdot)$ is decreasing, it can be easily checked that $x_0\in\mathcal{B}(x_0)\subseteq \overline{B}_Q(R)$ so in order to detect the target, a particle must first exit its respective box.
In particular, by~\eqref{boxboundsNew} and the second bound in Proposition~\ref{prop:mainuppermixed}, we obtain
\begin{equation}\label{eq:mainupper1}\PP_{x_0}(T_{det}\leq \ss)\,\leq\,\PP_{x_0}(\trad\leq \ss)+\sqrt{\frac{2}{\pi}}\int_{0}^{\infty}\frac{\hatt}{\sigma^{3/2}}e^{-\frac{\hatt^2}{2\sigma}}\PP_{x_0}(\II{\ss}\geq\sigma)d\sigma.\end{equation}
It suffices thus to bound the right-hand side. We first address the term $\PP_{x_0}(\trad\leq \ss)$, which depends on the radial trajectory of the particle alone, and hence we can make use of the results obtained in Section~\ref{sec:radial}. 
\begin{proposition}\label{generalschicorad}
Under the conditions of Proposition~\ref{generalsgrande} we have:
\begin{itemize}
    \item If $\KA=C$, then
    $\displaystyle\int_{\ndt}\P_{x_0}(\trad\leq \ss)d\mu(x_0) 
=O(n\phis)$.
\item If $\KA\geq C$ and $\ss=\omega(1)$, then
\[
\sup_{x_0\in\ndt}\P_{x_0}(\trad\leq \ss) =
\begin{cases}
e^{-\Omega(\KA^2)},
& \text{ if $\alpha\geq2\beta$,} \\[2pt]
O(\KA^{-\alpha(\frac{1}{\beta}\vee 2)}),
& \text{ if $\alpha<2\beta$.} 
\end{cases}
\]
\end{itemize}
\end{proposition}
\begin{proof}
Throughout this proof, let $x_0\in\ndt$ be such that $\theta_0>0$ (the case $\theta_0<0$ can be dealt with analogously).
As defined in Part~\eqref{radial:itm:phi3} of Lemma~\ref{lemmaradial},  with $\auxy=r_0$, $\yabs_0=\hatr$ and $\auxY=R$, let
\[
G_{\hatr}(r_0)
    := \frac{g(r_0)}{g(\hatr)}, \qquad \text{where}\qquad g(r):=\log(\tanh(\tfrac12\alpha R)/\tanh(\tfrac12\alpha r))
\]
Using the same arguments as in the proof of Proposition~\ref{prop:rad-lowerBnd} in Section~\ref{sec:radial}, bounding the term $\P_{R}(T_{\hatr}\leq\ss)$ analogously to what was done in~\eqref{eqn:radial-upper-aux2}, we have
\begin{equation}\label{eq:bound1radmix}\PP_{x_0}(\trad\leq \ss)\;\leq\;\P_{r_0}(T_{\hatr}<T_R)+\P_{R}(T_{\hatr}\leq\ss) \; = \; G_{\hatr}(r_0)+O(\ss e^{-\alpha(R-\hatr)}).
\end{equation}
To bound from above the numerator of $G_{\hatr}(r_0)$, note that the largest possible value of $\hatr$ among points in $\ndt$ is obtained at $r_0=R$ and $\theta_0=\phi(R)+\KA\phis$ so it follows from the definition of $\hatr$, $\hatt$, $\phi(\cdot)$, and Part~\eqref{itm:phi1} in Fact~\ref{fact:uppermix}, that
\[\phi(\hatr)\ge\phi(R)+\tfrac{1}{3}(\theta_0-\phi(R))=\phi(R)+\tfrac{1}{3}\KA\phis=\Omega(\KA e^{-R/2}).\]
Also note that $\phi(\hatr)=O(e^{-\hatr/2})$, so taking $\KA$ larger than a fixed constant gives $R-\hatr=\Omega(1)$, and hence
\[\log\Big(\frac{\tanh(\alpha R/2)}{\tanh(\alpha \hatr/2)}\Big)\;=\;\log\Big(1+\frac{2-2e^{\alpha(\hatr-R)}}{(e^{-\alpha R}+1)(e^{\alpha \hatr}-1)}\Big)\,=\,\log\big(1+O(e^{-\alpha\hatr})\big)=O(e^{-\alpha\hatr}).\]
From Part~\eqref{itm:phi2} in Fact~\ref{fact:uppermix} it follows that $r'=\Omega(\log\KR)=\Omega(1)$, and since $r_0\ge r'$ (because $x_0\in\ndt$) we can argue as in the proof of Fact~\ref{fct:radial-varphi2} that $g(r_0)=O(e^{-\alpha r_0})$, and hence combining with the previous bound we get $G_{\hatr}(r_0)=O(e^{-\alpha(r_0-\hatr)})$. 
Notice that both terms bounding $\PP_{x_0}(\trad\leq \ss)$ in~\eqref{eq:bound1radmix} involve a factor $e^{\alpha\hatr}$, which by definition of $\hatr$ and Part~\eqref{itm:phi4} of Lemma~\ref{lem:phi} is $O((\theta_0-\hatt)^{-2\alpha})$, giving
\begin{equation}\label{eq:mixradialuppermain}\PP_{x_0}(\trad\leq \ss)\;=\;O\big((\theta_0-\hatt)^{-2\alpha}e^{-\alpha r_0}+\ss(\theta_0-\hatt)^{-2\alpha}e^{-\alpha R}\big).\end{equation}
To obtain a uniform upper bound for $\PP_{x_0}(\trad\leq \ss)$ on $\ndt$, observe that for any fixed $r_0$, since $x_0\in\ndt$, we have $\theta_0\ge\gamma(r_0)$, where $\gamma$ was defined at the beginning of this section, giving
\begin{equation}\label{eq:diffangle}
\theta_0-\hatt=\tfrac13\theta_0-\tfrac23\phi(r_0)\ge\tfrac13\gamma(r_0)+\tfrac23\phi(r_0)\geq \tfrac13\KR e^{-(\beta\wedge\frac12) r_0}+\phi(r_0)\geq\tfrac13\KR e^{-(\beta\wedge\frac12) r_0}\end{equation}
and hence $(\theta_0-\hatt)^{-2\alpha}e^{-\alpha r_0}=O((\KR^{2}e^{(1-2(\beta\wedge\frac12))r_0})^{-\alpha})$.
For the case $\beta\geq\frac{1}{2}$ this bound is $O(\KR^{-2\alpha})$, while for the case $\beta<\frac{1}{2}$ the exponential term is equal to $e^{(1-2\beta)r_0}$ which is at least $e^{(1-2\beta)r'}$ since as already observed $r_0\ge r'$. By Part~\eqref{itm:upp2} in Fact~\ref{fact:uppermix} we have $r'=\frac{1}{\beta}\log(\KR)-O(1)$. Thus,
\[\KR^{-2\alpha}e^{-\alpha(1-2\beta)r_0}\leq \KR^{-2\alpha}e^{-\alpha(1-2\beta)r'}=O\big(\KR^{-2\alpha}\KR^{-\frac{\alpha}{\beta}(1-2\beta)}\big)=O\big(\KR^{-\frac{\alpha}{\beta}}\big)\]
so putting both cases together we obtain a bound of the form $O(\KR^{-\alpha/(\beta\wedge\frac12)})$. 
Next we address the term $\ss(\theta_0-\hatt)^{-2\alpha}e^{-\alpha R}$ appearing in~\eqref{eq:mixradialuppermain}. For $r'\le r_0\le r''$, by~\eqref{eq:diffangle}, we have
\[\ss(\theta_0-\hatt)^{-2\alpha}e^{-\alpha R}=O\big(\ss\KR^{-2\alpha}e^{2\alpha(\beta\wedge \frac12)r_0}e^{-\alpha R}\big)\]
which is an increasing function of $r_0$ and hence within $[r',r'']$ the bound is maximized at $r''$. Thus, it is enough to deal with points with radial component $r_0$ in $[r'',R]$. Now, for said points a similar argument to the one in~\eqref{eq:diffangle} gives $\theta_0-\hatt\geq\tfrac{1}{3}\KA \phis$ and hence
\[\ss(\theta_0-\hatt)^{-2\alpha}e^{-\alpha R}=O\big(\ss\KA^{-2\alpha}\big(e^{\frac{R}{2}}\phis\big)^{-2\alpha}\big).\]
The analysis of this term depends on $\alpha$ and $\beta$, and we begin addressing the case $\beta<\frac{1}{2}$ and $\alpha<2\beta$ since it is the most delicate. Notice that in this case $\phis=e^{-\beta R}\ss^{\frac{\beta}{\alpha}}$ so in particular $\ss(e^{\frac{R}{2}}\phis)^{-2\alpha}=(\ss e^{-\alpha R})^{1-2\beta}$ and hence the upper bound is an increasing function of $\ss$. Now, in order for $\ndt$ to be non-empty we need $\phi(R)+\KA\phis\leq\pi$ we deduce that $\ss e^{-\alpha R}=O(\KA^{-\frac{\alpha}{\beta}})$, which gives 
\[
\ss\KA^{-2\alpha}(e^{\frac{R}{2}}\phis)^{-2\alpha}=O\big(\KA^{-2\alpha}\big(\ss e^{-\alpha R}\big)^{1-2\beta}\big)=O\big(\KA^{-2\alpha}\KA^{-\frac{\alpha}{\beta}(1-2\beta)}\big)=O(\KA^{-\frac{\alpha}{\beta}}).\]
The remaining cases are easier: If $\beta\geq\frac{1}{2}$, then $e^{\frac{R}{2}}\phis=\ss^{\frac{1}{2\alpha}}$ so the term $\ss\KA^{-2\alpha}(e^{\frac{R}{2}}\phis)^{-2\alpha}$ is $O(\KA^{-2\alpha})$, while for the cases $\alpha=2\beta$ and $\alpha>2\beta$ it can be easily checked that since $\ss= O(e^{2\beta R})$ and $\ss= O(e^{2\beta R}/\log(e^{2\beta R}))$ respectively, we obtain $\ss(e^{\frac{R}{2}}\phis)^{-2\alpha}=o(1)$. Since $\KR=\KA$ whenever $\alpha<2\beta$, putting together the bounds for both terms in~\eqref{eq:mixradialuppermain} we finally obtain the sought after bound:
\[\PP_{x_0}(\trad\leq \ss)\;=\;O(\KR^{-\alpha/(\beta\wedge\frac12)})\;=\;\begin{cases}
e^{-\Omega(\KA^2)},
& \text{ if $\alpha\geq2\beta$,} \\[2pt]
O(\KA^{-\alpha/(\beta\wedge\frac12)}),
& \text{ if $\alpha<2\beta$.}
\end{cases}\]

Next, we show that for $\KA=C$ we have $\int_{\ndt}\PP_{x_0}(\trad\leq \ss)d\mu(x_0)=O(n\phis)$ by integrating both terms in~\eqref{eq:mixradialuppermain} over $\ndt$.
For the first one we use that $\theta_0-\hatt\geq\frac{1}{3}\theta_0$ to obtain 
\[
\int_{\ndt}(\theta_0-\hatt)^{-2\alpha}e^{-\alpha r_0}d\mu(x_0)= O(ne^{-\alpha R})\int_{r'}^R\int_{\gamma(r_0)}^{\infty}\theta_0^{-2\alpha}d\theta_0dr_0 =O(ne^{-\alpha R})\int_{r'}^{R}\frac{1}{\gamma(r_0)^{2\alpha-1}}dr_0.
\]
Splitting the range of integration of the last integral into $[r',r'']$ and $[r'',R]$ we obtain
\begin{align*}
\int_{\ndt}(\theta_0-\hatt)^{-2\alpha}e^{-\alpha r_0}d\mu(x_0)
& = O\big(ne^{-\alpha R}\big(e^{(2\alpha-1)(\beta\wedge\frac{1}{2})r''}+(R-r'')(\phis)^{-(2\alpha-1)}\big)\big) \\[2pt]
& = O\big(ne^{-\alpha R}(\phis)^{-(2\alpha-1)}\log(e^{(\beta\wedge\frac12)R}\phis)\big)
\end{align*}
where the last equality is by definition of $r''$ and since $e^{(\beta\wedge\frac{1}{2})R}\phis=\Omega(1)$ implies that $R-r''=\Omega(\log(e^{(\beta\wedge\frac{1}{2})R}\phis))$. To address this last bound suppose first that $\beta<\frac{1}{2}$ and observe that since $\phis=\Omega(e^{-\beta R})\cap O(1)$ we have
$e^{-\alpha R}(\phis)^{-2\alpha}\log(e^{(\beta\wedge\frac12)R}\phis)=O\big(Re^{-2\alpha(\frac{1}{2}-\beta)R}\big)=O(1)$,
while for $\beta\geq\frac{1}{2}$ we have $\alpha<2\beta$, hence $\phis=e^{-\frac{R}{2}}\ss^{\frac{1}{2\alpha}}$ so $e^{-\alpha R}(\phis)^{-2\alpha}\log(e^{(\beta\wedge\frac12)R}\phis)=O\left(\ss^{-1}\log\ss\right)=O(1)$ (since $\ss=\Omega(1)$) and we conclude that in any case scenario
\[\int_{\ndt}(\theta_0-\hatt)^{-2\alpha}e^{-\alpha r_0}d\mu(x_0)=O(n\phis).\]
For the second term in~\eqref{eq:mixradialuppermain} we use $\theta_0-\hatt\geq\frac{1}{3}\theta_0$ again, together with $\gamma(r_0)\geq C\phis$ to obtain 
\begin{align*}
\int_{\ndt}\ss (\theta_0-\hatt)^{-2\alpha}e^{-\alpha R}d\mu(x_0)&= O(\ss ne^{-\alpha R})\int_{r'}^R\int_{C\phis}^{\infty}\theta_0^{-2\alpha}d\theta_0e^{-\alpha(R-r_0)}dr_0
\\[2pt]&= O\big(\ss ne^{-\alpha R}(\phis)^{-(2\alpha-1)}\big),
\end{align*}
and it can be checked directly from the definition of $\phis$ and our assumption $\ss=\Omega(1)$ that this last expression is also $O(n\phis)$.
\end{proof}

\medskip

The previous proposition gave the uniform bound for the first term appearing in~\eqref{eq:mainupper1}, as well as a bound for its integral over $\ndt$. The following proposition allows us to bound from above the second term in~\eqref{eq:mainupper1} with the use of Proposition~\ref{prop:merged-varianza} to handle the function $\PP_{x_0}(\II{\ss}\geq\sigma)$. Recall that in order to apply said proposition we require $\sigma$ to be larger than $\sigma_0$ with
\begin{equation}\label{def:sigma0}
\sigma_0 := \begin{cases}
\auxl'e^{-2\beta R}(v_0\ss)^{1\vee\frac{2\beta}{\alpha}}, &
\text{if $\alpha \neq 2\beta$,} \\[2pt]
\auxl'e^{-2\beta R}v_0\ss\log(v_0\ss), &
\text{if $\alpha=2\beta$,}
\end{cases}
\end{equation}
where $\auxl'$ is a large constant that depends on $\alpha$ and $\beta$, and where $v_0$ is some function of $x_0$ for which we only ask to satisfy $v_0>\frac{4}{c}$ for $0<c< 1$ being a constant lower bound for $\ss$ (which exists since $\ss=\Omega(1)$).

\begin{proposition}\label{prop:splitterms}
Let $v_0:=v_0(x_0)$ be a function satisfying $v_0>\frac{4}{c}$, and let $\sigma_0$ and $\hatt$ be as in~\eqref{def:hatthatr} and~\eqref{def:sigma0}, respectively. Assume that $\KA\geq C$ where $C>0$ is as in the statement of Proposition~\ref{generalsgrande} and $\hatt^2=\Omega(\sigma_0)$, then
\begin{equation}\label{eq:mainterms}
\int_{0}^{\infty}\tfrac{\hatt}{\sigma^{3/2}}e^{-\frac{\hatt^2}{2\sigma}}\PP_{x_0}(\II{\ss}\geq\sigma)d\sigma=O\big(e^{-\frac{\hatt^2}{2\sigma_0}}+e^{-\frac{1}{16}v_0^2\ss}+\hatt^{-\frac{\alpha}{\beta}}e^{-\alpha r_0}+v_0\ss e^{-\alpha R}+\mathfrak{t}_5\big),
\end{equation}
where
\[\mathfrak{t}_5=\begin{cases}(v_0\ss)^{1\vee\frac{\alpha}{4\beta}}e^{-\alpha R}\hatt^{-\frac{\alpha}{\beta}}, &\text{ if $\alpha\neq2\beta$,} \\[2pt]
\int_{0}^{1}\sqrt{\tfrac{w\hatt^2}{\sigma_0}}e^{-\frac{w\hatt^2}{2\sigma_0}}\left(\auxl'\log(v_0\ss)\right)^{w-1}dw, &\text{ if $\alpha=2\beta$}.\end{cases}\]
\end{proposition}

\begin{proof}
Fix any given $x_0\in\ndt$ and split $\int_{0}^{\infty}\frac{\hatt}{\sigma^{3/2}}e^{-\frac{\hatt^2}{2\sigma}}\PP_{x_0}(\II{\ss}\geq\sigma)d\sigma$ into
\begin{equation}\label{eq:splitted}\int_{0}^{\sigma_0}\frac{\hatt}{\sigma^{3/2}}e^{-\frac{\hatt^2}{2\sigma}}\PP_{x_0}\big(\II{\ss}\geq\sigma\big)d\sigma+\int_{\sigma_0}^{\infty}\frac{\hatt}{\sigma^{3/2}}e^{-\frac{\hatt^2}{2\sigma}}\PP_{x_0}(\II{\ss}\geq\sigma)d\sigma.\end{equation}
For the first integral, we can bound the probability inside it by $1$, which using that $\hatt^2/\sigma_0=\Omega(1)$
and the change of variables $w:=\hatt^2/\sigma$ gives a bound
\[\int_{0}^{\sigma_0}\frac{\hatt}{\sigma^{3/2}}e^{-\frac{\hatt^2}{2\sigma}}d\sigma\,=\,\int_{\frac{\hatt^2}{\sigma_0}}^\infty\frac{1}{\sqrt{w}}e^{-\frac{w}{2}}dw\,=\,O\big(e^{-\frac{\hatt^2}{2\sigma_0}}\big)\]
 For the second integral in~\eqref{eq:splitted}, we use~\eqref{eq:division} to bound the probability within   together with Proposition~\ref{prop:merged-varianza} (since $\sigma>\sigma_0$) and obtain
\begin{equation}\label{eq:boundPI}\PP_{x_0}(\II{\ss}\geq\sigma)\,=\,O\big(e^{-\frac{1}{16}v_0^2\ss}+e^{-\alpha r_0}+\sigma^{-\frac{\alpha}{2\beta}} e^{-\alpha r_0}+v_0\ss e^{-\alpha R}+
\big((v_0\ss)^{1\vee\frac{\alpha}{4\beta}}\sigma^{-\frac{\alpha}{2\beta}}e^{-\alpha R}\big)^{1-\mathfrak{e}}\big)\end{equation}
where $\mathfrak{e}:=\frac{1}{\sigma e^{2\beta R}}\auxl'v_0\ss\log (v_0\ss)=\frac{\sigma_0}{\sigma}$ if $\alpha=2\beta$, and $\mathfrak{e}=0$ otherwise. Grouping the first, second and fourth terms in~\eqref{eq:boundPI}, the assumption $\hatt^2/\sigma_0= \Omega(1)$ gives
\[\int_{\sigma_0}^{\infty}\frac{\hatt}{\sigma^{3/2}}e^{-\frac{\hatt^2}{2\sigma}}\big(e^{-\frac{1}{16}v_0^2\ss}+e^{-\alpha r_0}+v_0\ss e^{-\alpha R}\big)d\sigma=O\big(e^{-\frac{1}{16}v_0^2\ss}+e^{-\alpha r_0}+v_0\ss e^{-\alpha R}\big)\]
since the terms are independent of $\sigma$ and $\int_{\sigma_0}^{\infty}\frac{\hatt}{\sigma^{3/2}}e^{-\frac{\hatt^2}{2\sigma}}d\sigma=\int_0^{\hatt^2/\sigma_0}\frac{1}{\sqrt{w}}e^{-\frac{w}{2}}dw=O(1)$. For the term $\sigma^{-\frac{\alpha}{2\beta}} e^{-\alpha r_0}$ analogous computations give 
\[\int_{\sigma_0}^\infty\frac{\hatt}{\sigma^{3/2}}e^{-\frac{\hatt^2}{2\sigma}}\sigma^{-\frac{\alpha}{2\beta}} e^{-\alpha r_0}d\sigma\,=\,e^{-\alpha r_0}\hatt^{-\frac{\alpha}{\beta}}\int_0^{\frac{\hatt^2}{\sigma_0}} w^{\frac{\alpha}{2\beta}-\frac{1}{2}}e^{-\frac{w}{2}}dw\,=\,O\big(e^{-\alpha r_0}\hatt^{-\frac{\alpha}{\beta}}\big).\]
The final term in~\eqref{eq:boundPI} is a bit more delicate, and must be treated differently according to whether $\alpha= 2\beta$ or $\alpha\neq 2\beta$. In the latter case we can repeat the computations used in the previous term, giving
\[\int_{\sigma_0}^{\infty}\frac{\hatt}{\sigma^{3/2}}e^{-\frac{\hatt^2}{2\sigma}}(v_0\ss)^{1\vee\frac{\alpha}{4\beta}}\sigma^{-\frac{\alpha}{2\beta}}e^{-\alpha R}d\sigma=O\big((v_0\ss)^{1\vee\frac{\alpha}{4\beta}}e^{-\alpha R}\hatt^{-\frac{\alpha}{\beta}}\big),\]
whereas if $\alpha=2\beta$, we use the change of variables $w:=\frac{\sigma_0}{\sigma}$ and the fact that $v_0\ss\sigma^{-\frac{\alpha}{2\beta}}e^{-\alpha R}
= \sigma_0/(\sigma\auxl'\log(v_0\ss))=w/(\auxl'\log(v_0\ss))$ so that
\begin{align*}\int_{\sigma_0}^{\infty}\frac{\hatt}{\sigma^{3/2}}e^{-\frac{\hatt^2}{2\sigma}}\big((v_0\ss)^{1\vee\frac{\alpha}{4\beta}}\sigma^{-\frac{\alpha}{2\beta}}e^{-\alpha R}\big)^{1-\frac{\sigma_0}{\sigma}}d\sigma
&=\int_{0}^{1}\frac{\hatt}{\sqrt{\sigma_0 w}}e^{-\frac{w\hatt^2}{2\sigma_0}}\Big(\frac{w}{\auxl'\log(v_0\ss)}\Big)^{1-w}dw\\[2pt]
&=O\Big(\int_{0}^{1}\sqrt{\tfrac{w\hatt^2}{\sigma_0}}e^{-\frac{w\hatt^2}{2\sigma_0}}\Big(\auxl'\log(v_0\ss)\Big)^{w-1}dw\Big).\end{align*}
The result then follows by adding all the bounds and noticing that $e^{-\alpha r_0}=O(\hatt^{-\frac{\alpha}{\beta}}e^{-\alpha r_0})$.
\end{proof}

We are now ready to prove the main result of this section.

\begin{proof}[Proof of Proposition~\ref{generalsgrande}]
Using~\eqref{eq:mainupper1} and Proposition~\ref{generalschicorad} it will be enough to obtain a uniform bound for $\int_{0}^{\infty}\frac{\hatt}{\sigma^{3/2}}e^{-\frac{\hatt^2}{2\sigma}}\PP_{x_0}(\II{\ss}\geq\sigma)d\sigma$ on $\ndt$ as well as an upper bound for the integral of this term over this set. We will address the uniform bounds first, since calculations are easier, and then turn to the upper bound for the integrals over $\ndt$, whose analysis is similar.


\medskip

\noindent\textit{Uniform upper bound:} Our goal is to show that under the hypothesis $\ss=\omega(1)$, 
\begin{equation}\label{eq:targetuniform}
\sup_{x_0\in\ndt}\int_{0}^{\infty}\tfrac{\hatt}{\sigma^{3/2}}e^{-\frac{\hatt^2}{2\sigma}}\PP_{x_0}(\II{\ss}\geq\sigma)d\sigma =
\begin{cases}
e^{-\Omega(\KA^2)},
& \text{ if $\alpha\geq2\beta$,} \\[2pt]
O(\KA^{-\alpha/(\beta\wedge\frac12)}),
& \text{ if $\alpha<2\beta$}. 
\end{cases}
\end{equation}
Throughout this analysis, we take $v_0$ constant and equal to $\frac{4}{c}$. With this choice of $v_0$ we deduce from Part~\eqref{itm:upp4} in Fact~\ref{fact:uppermix} that $\sigma_0=\Theta((\phis)^{2\vee4\beta})=O((\gamma(r_0)-\phi(r_0))^2)=O(\hatt^2)$ and hence we have~\eqref{eq:mainterms} as in Proposition~\ref{prop:splitterms}, so we address each term appearing in the bound as follows:

We assume throughout the uniform upper bound proof that $x_0=(r_0,\theta_0)\in\ndt$, and hence this set is non empty so in particular we must have $\KA\phis+\phi(R)\leq\pi$.
\begin{itemize}\setlength\itemsep{1ex}
    \item For the term $e^{-\frac{\hatt^2}{2\sigma_0}}$ it follows from the definition of $\hatt$ that it is equal to $e^{-\frac{2}{9\sigma_0}(|\theta_0|-\phi(r_0))^2}$ and since $|\theta_0|\ge \gamma(r_0)$ by Part~\eqref{itm:upp4} in Fact~\ref{fact:uppermix} we deduce $(|\theta_0|-\phi(r_0))^2=\Omega(\KA^2(\phis)^2)$.
    Recalling that $\sigma_0=\Theta((\phis)^{2\vee4\beta})$ the term $e^{-\frac{\hatt^2}{2\sigma_0}}$ equals $e^{-\Omega(\KA^2)}$, which is at most of the same order than the one claimed in~\eqref{eq:targetuniform}.
    \item The term $e^{-\frac{1}{16}v_0^2\ss}$ appearing in~\eqref{eq:mainterms}  does not depend on $x_0$ and since we are assuming $\ss=\omega(1)$, it becomes $o(1)$, so it is negligible.
    \item Now, we consider the term $\hatt^{-\frac{\alpha}{\beta}}e^{-\alpha r_0}$. By definition $\hatt=\Theta(|\theta_0|-\phi(r_0))$.
    Since $|\theta_0|\ge\gamma(r_0)\ge\phi(r_0)$ for $x_0\in\ndt$, by 
    Part~\eqref{itm:upp4} in Fact~\ref{fact:uppermix}, we have
    \[
    \big((|\theta_0|-\phi(r_0))^{-\frac{\alpha}{\beta}}e^{-\alpha r_0}=O\big(\KR^{-\frac{\alpha}{\beta}}e^{\alpha((1\wedge\frac{1}{2\beta})-1)r_0}\big),\]
    If $\beta\leq\frac{1}{2}$, this bound is $O(\KR^{-\frac{\alpha}{\beta}})$ independently of $r_0$, whilst if $\beta>\frac{1}{2}$, using that $r_0\ge r'$, the bound is $O(\KR^{-\frac{\alpha}{\beta}}/e^{\frac{\alpha}{2\beta}(2\beta-1)r_0})=O(\KR^{-\frac{\alpha}{\beta}}/e^{\frac{\alpha}{2\beta}(2\beta-1)r'})$. Recalling that, by Part~\eqref{itm:upp2} in Fact~\ref{fact:uppermix}, we know that $r'=2\log\KR-O(1)$, replacing this value in the previous bound we obtain a term of order $\KR^{-2\alpha}$. Summarizing, we have shown that $\hatt^{-\frac{\alpha}{\beta}}e^{-\alpha r_0}=O(\KR^{-\alpha(2\vee\frac{1}{\beta})})$ and the result then follows by our choice of $\KR$.
    \item For the term $v_0\ss e^{-\alpha R}$ we observe that it is independent of $x_0$, and under the assumptions $\ss=O(e^{\alpha R}/R)$ if $\alpha=2\beta$ and $\ss=O(e^{2\beta R})$ if $\alpha>2\beta$ in the statement of Proposition~\ref{prop:mainuppermixed} we obtain that $\ss e^{-\alpha R}=o(1)$ and hence the term is negligible. In the case $\alpha<2\beta$, since $\KA\phis+\phi(R)\leq\pi$ we obtain $\ss e^{-\alpha R}=O(\KA^{-\alpha/(\beta\wedge\frac12)})$ by definition of $\phis$.
    \item To address the term $\mathfrak{t}_5$ we first treat the case $\alpha\neq2\beta$. Using the definition of $\hatt$ and the fact that $\theta_0\geq\gamma(r_0)\ge \phi(r_0)$ this term is of order 
    \[\ss^{1\vee\frac{\alpha}{4\beta}}\hatt^{-\frac{\alpha}{\beta}}e^{-\alpha R}=O\big(\ss^{1\vee\frac{\alpha}{4\beta}}(\gamma(r_0)-\phi(r_0))^{-\frac{\alpha}{\beta}}e^{-\alpha R}\big)\]
    and using Part~\eqref{itm:upp4} in Fact~\ref{fact:uppermix} we deduce
    \begin{equation}\label{eq:boundforlater}\ss^{1\vee\frac{\alpha}{4\beta}}\hatt^{-\frac{\alpha}{\beta}}e^{-\alpha R}=O\big(\ss^{1\vee\frac{\alpha}{4\beta}}(\KA\phis)^{-\frac{\alpha}{\beta}}e^{-\alpha R}\big)=O((\KA\ss^{-(\frac{\beta}{\alpha}\vee\frac14)}e^{\beta R}\phis)^{-\frac{\alpha}{\beta}}).\end{equation}    From the definition of $\phis$ we have
    \[
\ss^{-(\frac{\beta}{\alpha}\vee\frac{1}{4})}e^{\beta R}\phis =
\begin{cases}
(e^{\frac{R}{2} }\ss^{-\frac{1}{2\alpha}})^{0\vee (2\beta-1)},
& \text{ if $\alpha<2\beta$,}\\[2pt]
\ss^{(\frac12-\frac{\beta}{\alpha})\vee\frac{1}{4}},
& \text{ if $\alpha>2\beta$,}
\end{cases}
\]
where we directly observe that if $\alpha>2\beta$ this expression is $\omega(1)$ (since $\ss=\omega(1)$), and hence the upper bound is $o(1)$ so in particular it has the form $e^{-\Omega(\KA^2)}$ as desired. For the case $\alpha<2\beta$ we distinguish between the cases $\beta\leq\frac{1}{2}$ and $\beta>\frac{1}{2}$: In the former we directly obtain that $\ss^{1\vee\frac{\alpha}{4\beta}}(\KA\phis)^{-\frac{\alpha}{\beta}}e^{-\alpha R}=O(\KA^{-\frac{\alpha}{\beta}})$, whereas if $\beta>\frac12$, since $\phi(R)+\KA\phis\leq\pi$ and by definition of $\phis$, we have
$\ss^{-(\frac{\beta}{\alpha}\vee\frac14)}e^{\beta R}\phis=(1/\phis)^{2\beta-1}=\Omega(\KA^{-(2\beta-1)})$ and obtain an upper bound of $O\big(\KA^{-\frac{\alpha}{\beta}}\KA^{-\frac{\alpha}{\beta}(2\beta-1)}\big)=O(\KA^{-2\alpha})$. 

Assume now that $\alpha=2\beta$ so that $\mathfrak{t}_5=\int_{0}^{1}\sqrt{\tfrac{w\hatt^2}{\sigma_0}}e^{-\frac{w\hatt^2}{2\sigma_0}}\left(\auxl'\log(v_0\ss)\right)^{w-1}dw$ and notice that since the function $x\to xe^{-x^2}$ is $O(1)$ on $\RR^+$, and using that $\ss=\omega(1)$ we obtain
\[\mathfrak{t}_5=O\Big(\int_{0}^{1}(\auxl'\log(v_0\ss))^{w-1}dw\Big)=o(1)\]
and hence the term is negligible in this case.
\end{itemize}

\medskip

\noindent\textit{Integral upper bound:} Our goal is to show that under the hypothesis $\KA= C$ for $C$ large and $\ss=\Omega(1)$,
\[\int_{\ndt}\int_0^\infty\tfrac{\hatt}{\sigma^{3/2}}e^{-\frac{\hatt^2}{2\sigma}}\PP_{x_0}\big(\II{\ss}\geq\sigma\big)d\sigma dx_0 = O(n\phis).\]
In contrast to the choice in the uniform upper bound analysis, we will now choose 
\[v_0(x_0)\;:=\;\frac{4}{c}\Big(\frac{|\theta_0|-\phi(r_0)}{\gamma(r_0)-\phi(r_0)}\Big)^{\varepsilon}\]
for some fixed $\varepsilon>0$ satisfying $\varepsilon<(2\alpha-1)(1\wedge\frac{\alpha}{2\beta})$ (this is possible since $\alpha>\frac{1}{2}$) and where $0<c<1$ is a lower bound for $\ss$. It will be convenient to define $\varepsilon'=\varepsilon/(1\wedge\frac{\alpha}{2\beta})< 1$ since it will appear many times in subsequent computations. 
In order to apply Proposition~\ref{prop:splitterms} we need to show that $\hatt^2=\Omega(\sigma_0)$, which is less straightforward than the uniform bound case since $\sigma_0$ is also increasing with $\theta_0$. In the case $\alpha\neq2\beta$, observe that $|\theta_0|-\phi(r_0)=\Omega(v_0^{1/\varepsilon}(\gamma(r_0)-\phi(r_0)))$ while at the same time $\gamma(r_0)-\phi(r_0)=\Omega(\phis)=\Omega(e^{-\beta R}\ss^{\frac{1}{2}\vee\frac{\beta}{\alpha}})=\Omega(v_0^{-\frac{\varepsilon'}{2\varepsilon}}\sqrt{\sigma_0})$. We conclude that
\begin{equation}\label{eq:mixfracneq}
\frac{\hatt^2}{\sigma_0} = \frac{4}{9\sigma_0}(|\theta_0|-\phi(r_0))^2=\Omega(v_0^{\frac{2-\varepsilon'}{\varepsilon}}),
\end{equation}
which is $\Omega(1)$ since $v_0=\Omega(1)$ and $\varepsilon'<2$. Analogously, if $\alpha=2\beta$ we still have $|\theta_0|-\phi(r_0)=\Omega(v_0^{1/\varepsilon}(\gamma(r_0)-\phi(r_0)))$ whereas this time $\gamma(r_0)-\phi(r_0)=\Omega(\sqrt{\frac{\sigma_0\log\ss}{v_0\log(v_0\ss)}})$ so that 
\begin{equation}\label{eq:mixfraceq}\frac{\hatt^2}{\sigma_0}= \frac{4}{9\sigma_0}(|\theta_0|-\phi(r_0))^2=\Omega(v_0^{\frac{2-\varepsilon}{\varepsilon}}\tfrac{\log\ss}{\log(v_0\ss)}),\end{equation}
which is again $\Omega(1)$ since both $\ss=\Omega(1)$ and $v_0=\Omega(1)$. We have proved that in either case $\hatt^2=\Omega(\sigma_0)$. As a result we have~\eqref{eq:mainterms} where we address each term appearing in the upper bound as follows:

\medskip
\begin{itemize}\setlength\itemsep{1ex}
    \item 
    For the term $e^{-\frac{\hatt^2}{2\sigma_0}}$ assume first that $\alpha\neq2\beta$ and use \eqref{eq:mixfracneq} to obtain 
\[\int_{\ndt}\exp\Big({-}\frac{\hatt^2}{2\sigma_0}\Big)d\mu(x_0)=n\int_{r'}^R\int_{\gamma(r_0)}^\infty \exp\Big({-}\Omega\big(v_0^{\frac{2-\varepsilon'}{\varepsilon}}(\theta_0)\big)\Big)d\theta_0 e^{-\alpha(R-r_0)}dr_0,\]
so using the change of variable $w:=\frac{\theta_0-\phi(r_0)}{\gamma(r_0)-\phi(r_0)}$ we obtain 
\[\int_{\gamma(r_0)}^\infty \exp\Big({-}\Omega\big(v_0^{\frac{2-\varepsilon'}{\varepsilon}}(\theta_0)\big)\Big)d\theta_0=(\gamma(r_0)-\phi(r_0))\int_{1}^\infty e^{-\Omega(w^{2-\varepsilon'})}dw=O(\gamma(r_0)-\phi(r_0)),\]
and hence
\[\int_{\ndt}\exp\Big(-\frac{\hatt^2}{2\sigma_0}\Big)d\mu(x_0)=O\Big(n\int_{r'}^R(\gamma(r_0)-\phi(r_0))e^{-\alpha(R-r_0)}dr_0\Big)=O(n\phis),\]
where the last equality follows from Part~\eqref{itm:upp5} in Fact~\ref{fact:uppermix}.

\medskip

For the case $\alpha=2\beta$ we follow the same reasoning using \eqref{eq:mixfraceq} to deduce
\begin{align*}\int_{\ndt}\exp\Big({-}\frac{\hatt^2}{2\sigma_0}\Big)d\mu(x_0)&=n\int_{r'}^R\int_{\gamma(r_0)}^\infty \exp\left(-\Omega\left(v_0^{\frac{2-\varepsilon}{\varepsilon}}(\theta_0)\tfrac{\log(\ss)}{\log(v_0(\theta_0)\ss)}\right)\right)d\theta_0 e^{-\alpha(R-r_0)}dr_0\\[3pt]&=n\int_{r'}^R(\gamma(r_0)-\phi(r_0))\int_{1}^\infty \exp\left(-\Omega(w^{2-\varepsilon'}\tfrac{\log(\ss)}{\log(w^{\varepsilon}\ss)})\right)dwe^{-\alpha(R-r_0)}dr_0,\end{align*}
which is $O(n\phis)$ as before.

\item 
For the term $e^{-\frac{1}{16}v_0^2\ss}$ appearing in~\eqref{eq:mainterms} we use the change of variable $w=(v_0\sqrt{\ss})^{1/\varepsilon}$ so that 
\[\int_{\gamma(r_0)}^\infty e^{-\frac{1}{16}v_0^2\ss}d\theta_0=\ss^{-\frac{1}{2\varepsilon}}(\gamma(r_0)-\phi(r_0))\int_{\Omega(\ss^{\frac{1}{2\varepsilon}})}^{\infty}e^{-\Omega(w^{2\varepsilon})}dw=O(\gamma(r_0)-\phi(r_0)),\]
where the last equality follows from $\ss=\Omega(1)$. Hence,
\[
\int_{\ndt}e^{-\frac{1}{16}v_0^2\ss}d\mu(x_0)
=O\Big(n\int_{r'}^R (\gamma(r_0)-\phi(r_0))e^{-\alpha(R-r_0)}dr_0\Big)=O(n\phis),
\]
where the last equality follows from Part~\eqref{itm:upp5} in Fact~\ref{fact:uppermix}.
\item 
For the term $\hatt^{-\frac{\alpha}{\beta}}e^{-\alpha r_0}$, by definition of $\hatt$ we obtain  
\[
\int_{\ndt}\hatt^{-\frac{\alpha}{\beta}}e^{-\alpha r_0}d\mu(x_0)\,=\,O\Big(ne^{-\alpha R}\int_{r'}^R\int_{\gamma(r_0)}^\pi (\theta_0-\phi(r_0))^{-\frac{\alpha}{\beta}}d\theta_0dr_0\Big).\]
We treat first the case $\beta\leq\frac{1}{2}$, which implies in particular $\frac{\alpha}{\beta}>1$ and hence the right-hand side is of order
\[ne^{-\alpha R}\int_{r'}^R(\gamma(r_0)-\phi(r_0))^{1-\frac{\alpha}{\beta}}dr_0\,=\,O\big(ne^{-\alpha R}(e^{-(\beta\wedge\frac{1}{2})r''})^{1-\frac{\alpha}{\beta}}+ne^{-\alpha R}(R-r'')(\phis)^{1-\frac{\alpha}{\beta}}\big)\]
where we first integrate from $r'$ to $r''$, and then from $r''$ to $R$, and then used Part~\eqref{itm:upp4} in Fact~\ref{fact:uppermix} to bound $\gamma(r_0)-\phi(r_0)$ from below. Using Part~\eqref{itm:upp3} of Fact~\ref{fact:uppermix} in the first term on the right-hand side, we obtain a term of order $ne^{-\alpha R}(\phis)^{1-\frac{\alpha}{\beta}}=O(n\phis)$, while for the second term using that $\beta\leq\frac{1}{2}$ we have by definition that $r''=-\frac{1}{\beta}\log\phis+\Theta(1)=R-\Theta(\log\ss)$ so in particular $ne^{-\alpha R}(R-r'')(\phis)^{1-\frac{\alpha}{\beta}}=O(n\phis\ss^{-\delta}\log\ss)=O(n\phis)$ for some $\delta>0$ depending on $\alpha$ and $\beta$, giving the result in this case.
For the case $\alpha>\beta>\frac12$ we can repeat the calculations in the previous case to again obtain a bound of order $n\phis$.

\smallskip
For the case $\beta>\frac{1}{2}$ we have $\phis=e^{-\frac12 R}\ss^{\frac{1}{2\alpha}}$ and we still need to handle the case $\frac{\alpha}{\beta}\le 1$. If $\frac{\alpha}{\beta}<1$, we have
\[ne^{-\alpha R}\int_{r'}^R\int_{\gamma(r_0)}^\pi (\theta_0-\phi(r_0))^{-\frac{\alpha}{\beta}}d\theta_0dr_0\,=\,O\big(ne^{-\alpha R}R\big),\]
and in it is easily checked that since $\alpha>\frac12$ and by the definition of $\phis$  we have $ne^{-\alpha R}R=O(n\phis e^{(\frac{1}{2}-\alpha)R}R)=O(n\phis)$. Finally, if $\frac{\alpha}{\beta}=1$ we have
\[ne^{-\alpha R}\int_{r'}^R\int_{\gamma(r_0)}^\pi (\theta_0-\phi(r_0))^{-\frac{\alpha}{\beta}}d\theta_0dr_0\,=\,O\Big(ne^{-\alpha R}\int_{r'}^R\log\big(\tfrac{\pi}{\gamma(r_0){-}\phi(r_0)}\big)dr_0\Big)\]
where using Part~\eqref{itm:upp1} in Fact~\ref{fact:uppermix} we conclude that the right-hand term is at most of order $ne^{-\alpha R}R\log(1/\phis)=O(ne^{-\alpha R}R^2)=O(n\phis)$.

\item For the term $v_0\ss e^{-\alpha R}$ we observe that, using the change of variable
  $w:=v_0^{1/\varepsilon}$ together with Part~\eqref{itm:upp4} in Fact~\ref{fact:uppermix}, we get
\[
\int_{\gamma(r_0)}^{\pi} v_0\ss e^{-\alpha R}d\theta_0 
= O(\ss e^{-\alpha R}(\gamma(r_0)-\phi(r_0))^{-\varepsilon})\int_{0}^{\pi} w dw
= O(\ss e^{-\alpha R}(\phis)^{-\varepsilon}).
\]
Now, suppose first that $\beta>\frac{1}{2}$ so that $e^{-\alpha R}=\ss^{-1}(\phis)^{2\alpha}$ and hence $\ss e^{-\alpha R}(\phis)^{-\varepsilon}=(\phis)^{2\alpha-\varepsilon}=O(\phis)$ where the last equality 
is by our assumption of $\varepsilon<2\alpha-1$. Suppose now that $\beta\leq \frac{1}{2}$ and observe that in this case $e^{-\alpha R}=O((\phis)^{\frac{\alpha}{\beta}}\ss^{-(1\vee\frac{\alpha}{2\beta})})$ 
and $\frac{\alpha}{2\beta}\leq\alpha<1$ so using that $\ss=\Omega(1)$ we deduce again that
$\ss e^{-\alpha R}(\phis)^{-\varepsilon}=O\left((\phis)^{2\alpha-\varepsilon}\right)=O(\phis)$.
Summarizing, independent of the value taken by $\beta$ we have
\[
\int_{\ndt}v_0\ss e^{-\alpha R}d\mu(x_0) = O(n\phis)\int_{r'}^R e^{-\alpha(R-r_0)}dr_0 = O(n\phis).
\]

\item For the term $\mathfrak{t}_5$ it will be useful to notice from Part~\ref{itm:upp4} in Fact~\ref{fact:uppermix} that \begin{equation}\label{eq:mixaux}\left(\gamma(r_0)-\phi(r_0)\right)^{-\delta}=O((\phis)^{-\delta})\end{equation} for any $\delta>0$. We assume first that $\alpha\neq2\beta$ so that
\begin{equation}\label{eq:lastterm}\int_{\ndt}\mathfrak{t}_5(x_0)d\mu(x_0)\,=\,O\Big(ne^{-\alpha R}\int_{r'}^R\int_{\gamma(r_0)}^\pi (v_0\ss)^{1\vee\frac{\alpha}{4\beta}}\hatt^{-\frac{\alpha}{\beta}}d\theta_0e^{-\alpha(R-r_0)}dr_0\Big)\end{equation}
We analyze this expression first when $\alpha\leq 4\beta$ so using \eqref{eq:mixaux} the term on the right-hand side is of order
\begin{equation}\label{eq:mix:termfinal}ne^{-\alpha R}\ss(\phis)^{-\varepsilon}\int_{r'}^R\int_{\gamma(r_0)}^\pi (\theta_0-\phi(r_0))^{\varepsilon-\frac{\alpha}{\beta}}d\theta_0e^{-\alpha(R-r_0)}dr_0,\end{equation}
where the integral with respect to $\theta_0$ behaves differently according to whether $\varepsilon-\frac{\alpha}{\beta}$ is smaller, equal, or larger than ${-}1$. Suppose first that $\varepsilon-\frac{\alpha}{\beta}<{-}1$ so that the above term is of order
\[ne^{-\alpha R}\ss(\phis)^{-\varepsilon}\int_{r'}^R (\gamma(r_0)-\phi(r_0))^{1-\frac{\alpha}{\beta}+\varepsilon}e^{-\alpha(R-r_0)}dr_0\,=\,O\big(ne^{-\alpha R}\ss(\phis)^{1-\frac{\alpha}{\beta}}\big),\]
where the right hand side follows from \eqref{eq:mixaux}. Proceeding as in~\eqref{eq:boundforlater} we deduce that the bound is $O(n\phis)$. Suppose now that $\varepsilon-\frac{\alpha}{\beta}\geq-1$, which in particular implies $\beta\geq\frac{\alpha}{1+\varepsilon}\geq\frac{\alpha}{1+(2\alpha-1)}=\frac{1}{2}$ so in this case $\phis=e^{-\frac12 R}\ss^{\frac{1}{2\alpha}}$. Now, if $\varepsilon-\frac{\alpha}{\beta}>-1$ then the integral with respect to $\theta_0$ in~\eqref{eq:mix:termfinal} is $O(1)$, and thus the bound is of order
\[ne^{-\alpha R}\ss(\phis)^{-\varepsilon}\int_{r'}^Re^{-\alpha(R-r_0)}dr_0=O\big(ne^{-\alpha R}\ss(\phis)^{-\varepsilon}\big)=O\big(n(\phis)^{2\alpha-\varepsilon}\big)\]
where we have used the particular form of $\phis$. Since $\varepsilon$ satisfies $\varepsilon<2\alpha-1$ and also $\phis\le\pi=O(1)$, we conclude that the bound is $O(n\phis)$. Similarly, if $\varepsilon-\frac{\alpha}{\beta}=-1$ then~\eqref{eq:mix:termfinal} is of order 
\[ne^{-\alpha R}\ss(\phis)^{-\varepsilon}\int_{r'}^R\log\big(\tfrac{1}{\gamma(r_0)-\phi(r_0)}\big)e^{-\alpha(R-r_0)}dr_0=O\big(ne^{-\alpha R}\ss(\phis)^{-\varepsilon}\log\big(\tfrac{1}{\phis}\big)\big),\]
which is of order $n(\phis)^{2\alpha-\varepsilon}\log(\frac{1}{\phis})$ and again using that $\varepsilon<2\alpha-1$ and $\phis=O(1)$ we conclude that the bound is $O(n\phis)$.

\medskip

We now turn to the analysis of the right-hand side of~\eqref{eq:lastterm} when $\alpha>4\beta$, in which case we obtain a term of order
\begin{equation}\label{eq:mix:termfinal2}ne^{-\alpha R}\ss^{\frac{\alpha}{4\beta}}(\phis)^{-\frac{\alpha\varepsilon}{4\beta}}\int_{r'}^R\int_{\gamma(r_0)}^\pi (\theta_0-\phi(r_0))^{\frac{\alpha\varepsilon}{4\beta}-\frac{\alpha}{\beta}}d\theta_0e^{-\alpha(R-r_0)}dr_0.\end{equation}
Observe that the condition $\alpha>4\beta$ implies that $\phis=e^{-\beta R}\sqrt{\ss}$. Assume first that $\frac{\alpha\varepsilon}{4\beta}-\frac{\alpha}{\beta}<-1$ in which case~\eqref{eq:mix:termfinal2} is of order 
\[ne^{-\alpha R}\ss^{\frac{\alpha}{4\beta}}(\phis)^{-\frac{\alpha\varepsilon}{4\beta}}\int_{r'}^R (\gamma(r_0)-\phi(r_0))^{1-\frac{\alpha}{\beta}+\frac{\alpha\varepsilon}{4\beta}}e^{-\alpha(R-r_0)}dr_0=O\big(ne^{-\alpha R}\ss^{\frac{\alpha}{4\beta}}(\phis)^{1-\frac{\alpha}{\beta}}\big),\]
and from the particular form of $\phis$ we have $e^{-\alpha R}\ss^{\frac{\alpha}{4\beta}}(\phis)^{-\frac{\alpha}{\beta}}=\ss^{-\frac{\alpha}{4\beta}}=O(1)$ and hence the bound is $O(n\phis)$. Suppose next that $\frac{\alpha\varepsilon}{4\beta}-\frac{\alpha}{\beta}>-1$, in which case~\eqref{eq:mix:termfinal2} is of order 
\[ne^{-\alpha R}\ss^{\frac{\alpha}{4\beta}}(\phis)^{-\frac{\alpha\varepsilon}{4\beta}}\int_{r'}^R e^{-\alpha(R-r_0)}dr_0=O\big(ne^{-\alpha R}\ss^{\frac{\alpha}{4\beta}}(\phis)^{-\frac{\alpha\varepsilon}{4\beta}}\big),\]
and in order to show that the bound is $O(n\phis)$ it will be enough to prove that $e^{-\alpha R}\ss^{\frac{\alpha}{4\beta}}(\phis)^{-1-\frac{\alpha\varepsilon}{4\beta}}=O(1)$. It can be checked that using the particular form of $\phis$, we have $e^{-\alpha R}\ss^{\frac{\alpha}{4\beta}}(\phis)^{-1-\frac{\alpha\varepsilon}{4\beta}}=e^{-\alpha\frac{R}{2}}(\phis)^{\frac{\alpha}{2\beta}-1-\frac{\alpha\varepsilon}{4\beta}}$, and the result will follow as soon as the exponent of $\phis$ is positive. However, this is equivalent to $2-\varepsilon>\frac{4\beta}{\alpha}$ which holds since by definition $\varepsilon<1$, and we are assuming $\frac{4\beta}{\alpha}<1$. Finally, suppose that $\frac{\alpha\varepsilon}{4\beta}-\frac{\alpha}{\beta}=-1$, in which case~\eqref{eq:mix:termfinal2} is of order
$ne^{-\alpha R}\ss^{\frac{\alpha}{4\beta}}(\phis)^{-\frac{\alpha\varepsilon}{4\beta}}\log(1/\phis)$
and using the same analysis as in the case $\frac{\alpha\varepsilon}{4\beta}-\frac{\alpha}{\beta}>-1$ (with strict inequalities) we deduce that it is $O(n\phis)$.

\medskip
We now bound the integral of the term $\mathfrak{t}_5$ in the special case $\alpha=2\beta$ where
\begin{align*}\int_{\ndt}\mathfrak{t}_5(x_0)d\mu(x_0)&=\,O\Big(n\int_{r'}^R \int_{\gamma(r_0)}^\infty\int_{0}^{1}\sqrt{\tfrac{w\hatt^2}{\sigma_0}}e^{-\frac{w\hatt^2}{2\sigma_0}}(\auxl'\log(v_0\ss))^{w-1}dwd\theta_0e^{-\alpha(R-r_0)}dr_0\Big)\\[2pt]&=O\Big(n\int_{r'}^R \int_{0}^{1}\int_{\gamma(r_0)}^\infty\sqrt{\tfrac{w\hatt^2}{\sigma_0}}e^{-\frac{w\hatt^2}{2\sigma_0}}d\theta_0dwe^{-\alpha(R-r_0)}dr_0\Big).\end{align*}
Now, for any fixed $w\in(0,1)$, we can use the change of variables $u:=\frac{\theta_0-\phi(r_0)}{\gamma(r_0)-\phi(r_0)}$ in the inner integral, which gives
\[\int_{\gamma(r_0)}^\infty\sqrt{\tfrac{w\hatt^2}{\sigma_0}}e^{-\frac{w\hatt^2}{2\sigma_0}}d\theta_0=(\gamma(r_0)-\phi(r_0))\int_{1}^\infty\sqrt{\frac{wuy_0\log\ss}{\log(u\ss)}}\exp\Big({-}\Omega\left(\frac{wuy_0\log(\ss)}{2\log(u\ss)}\right)\Big)du\]
for $y_0:=\frac{(\gamma(r_0)-\phi(r_0))^2}{e^{-2\beta R}\ss\log(\ss)}$ which is $\Omega(1)$ since $\gamma(r_0)-\phi(r_0)=\Omega(\phis)$. On the other hand $\frac{\log(\ss)}{\log(u\ss)}=\Omega(\frac{1}{\log(u)})$ which allows us to conclude that
\[\int_0^1\int_{1}^\infty\sqrt{\frac{wuy_0\log\ss}{\log(u\ss)}}\exp\Big({-}\Omega\left(\frac{wuy_0\log(\ss)}{2\log(u\ss)}\right)\Big)dudw=O(1)\]
Therefore,
\[\int_{\ndt}\mathfrak{t}_5(x_0)d\mu(x_0)=O\Big(n\int_{r'}^R(\gamma(r_0)-\phi(r_0))e^{-\alpha(R-r_0)}dr_0\Big)=O(n\phis)\]
where the last equality follows directly from Part~\eqref{itm:upp5} in Fact~\ref{fact:uppermix}.
\end{itemize}
\end{proof}






\section{Conclusion and outlook}\label{sec:conclusion}
We studied a movement of particles in the random hyperbolic graph model introduced by~\cite{KPKVB10} and analyzed the tail of detection times of a fixed target in this model. To the best of our knowledge, the current paper is the first one in analyzing dynamic hyperbolic graphs. It is natural to ask similar questions to the ones addressed in~\cite{Peres2010} in the model of random geometric graphs: how long does it take to detect a target that is mobile itself? How long does it take to detect all (fixed or mobile) targets in a certain sector? How long does it take in order for a given vertex initially outside the giant component at some fixed location needs to connect to some vertex of the giant component (percolation time)? How long does it take to broadcast a message in a mobile graph? We think that the detection of a mobile target, using the same proof ideas as presented here, should accelerate the detection time by a constant factor (presumably by a factor $2$ in the angular case and perhaps by a different factor in general), and also the deletion time of all targets in a certain sector can presumably be done with similar techniques to the ones used in this paper. On the other hand, new proof ingredients might be needed for analyzing the percolation time and the broadcast time. We think that the same observation as the one made by Moreau et al.~\cite{Moreau} which show that for the continuous random walk on the square lattice the best strategy to detect (to end up right at the same location) is to stay put, holds in our model as well.

\appendix
\section{Appendix: on the sums of Pareto variables}

Here we prove Lemma \ref{lem:cotapower} that gives a large deviation result for sums of independent heavy tailed random variables. The lemma is a more precise version of the results in \cite{Omelchenko2019} (although here we treat only with positive random variables which are bounded away from zero), and the proof follows closely what was done there.

\begin{lemma}
Let $S_m=\sum_{i=1}^m Z_i$ where $\{Z_i\}_{i\in\NN}$ is a sequence of i.i.d. absolutely continuous random variables taking values in $[1,\infty)$ such that there are $V,\gamma>0$ with
\[1-F_Z(x)=\PP(Z_i\geq x)\leq Vx^{-\gamma}\] 
for all $x>0$. Then there are $\auxc,\auxl>0$ depending on $V$, $\gamma$ and $\EE(Z_1)$ (if it exists) alone such that:
\begin{itemize}
    \item If $\gamma<1$, then for all $L>\auxl$,
    $\displaystyle\PP(S_m\geq Lm^{\frac{1}{\gamma}})\leq \auxc L^{-\gamma}$.
    \item If $\gamma=1$, then for all $L>\auxl$, $\displaystyle\PP(S_m\geq Lm\log(m))\leq \left(\frac{\auxc}{L\log(m)}\right)^{1-\frac{\auxl}{L}}$.
    \item If $\gamma>1$, then for all $L>\auxl$,
    $\displaystyle\PP(S_m\geq Lm)\,\leq\,\auxc L^{-\gamma}m^{-((\gamma-1)\wedge\frac{\gamma}{2})}$.
\end{itemize}
\end{lemma}
\begin{remark}
We remark that tighter results can be obtained for $\gamma \ge 1$ below, but the current result is sufficient for our purposes.
\end{remark}
\begin{proof}
We proceed as in \cite{Omelchenko2019} by assuming first that $\gamma\leq 1$. Fix $x>0$ and define the event $B=\{\forall 1\leq i< m,\,Z_i\leq x\}$ so that 
\[
    \PP(S_m\geq x)\;\leq\;\PP(\overline{B})+\PP(S_m\geq x\,|\,B)\PP(B)\;\leq\;mVx^{-\gamma}+\PP(S^{(x)}_m\geq x)\PP(B)
\]
where $S^{(x)}$ is the sum of $m$ i.i.d. random variables with c.d.f. $\frac{F_Z(y)}{F_Z(x)}$ for any $y \in [0,x]$. We can thus use a Chernoff bound to deduce
\[
    \PP(S_m\geq x)\;\leq\;mVx^{-\gamma}+e^{-\lambda x}\EE\big(e^{\lambda S^{(x)}_m}\big)\PP(B)\;=\;mVx^{-\gamma}+e^{-\lambda x}\left(\int_1^xe^{\lambda y}dF_Z(y)\right)^m,
\]
where we used independence of the random variables to conclude that $\PP(B)=(F_Z(x))^m$. Define now $M=\frac{2\gamma}{\lambda}$ for some $\lambda$ to be chosen below so that $M<x$ holds. Hence,
\[R(\lambda,x)\,:=\,\int_1^xe^{\lambda y}dF_Z(y)\,=\,\int_1^Me^{\lambda y}dF_Z(y)\,+\,\int_M^xe^{\lambda y}dF_Z(y),\]
which we bound separately. First notice that for some constant $C_1=C_1(\gamma, V)$ we have
\begin{align*}
    \int_1^Me^{\lambda y}dF_Z(y)&\leq\,e^{\lambda M}F_Z(M)-\lambda\int_1^Me^{\lambda y}F_Z(y)dy\\[3pt]
    &\leq\,e^{\lambda M}F_Z(M)-e^{\lambda M}+1+\lambda\int_1^Me^{\lambda y}(1-F_Z(y))dy\\[3pt]
    &\leq\,e^{\lambda M}F_Z(M)-e^{\lambda M}+1+\lambda Ve^{\lambda M}\int_1^My^{-\gamma}dy\;\leq\,1+C_1Q(\lambda),
\end{align*}
where $Q(\lambda)=\lambda^{\gamma}$ if $\gamma<1$ and $Q(\lambda)=-\lambda\log(\frac{\lambda}{2\gamma})$ if $\gamma=1$. For the integral between $M$ and $x$ observe that
\begin{align*}
    \int_M^xe^{\lambda y}dF_Z(y)&\leq\,e^{\lambda M}(1-F_Z(M))+\lambda\int_M^xe^{\lambda y}(1-F_Z(y))dy\\[3pt]
    &\leq\,Ve^{\lambda M}M^{-\gamma}+\lambda V\int_M^xe^{\lambda y}y^{-\gamma}dy\\[3pt]
    &=\,Ve^{2\gamma}\left(\frac{\lambda}{2\gamma}\right)^{\gamma}+ Ve^{\lambda x}x^{-\gamma}\int_0^{\lambda(x-M)}e^{-w}\left(1-\frac{w}{\lambda x}\right)^{-\gamma}dw,
\end{align*}
where in the last line we used the change of variables $w=\lambda(x-y)$. Now, since $M = 2\gamma/\lambda$, the function $f(w)=e^{\frac{w}{2}}(1-\frac{w}{\lambda x})^{\gamma}$ has a positive derivative for $w \in [0, \lambda(x-M)]$. Hence for $w\in[0,\lambda(x-M)]$ we have  $(1-\frac{w}{\lambda x})^{-\gamma}\leq e^{w/2}$  and hence the last integral is therefore smaller than $\int_0^\infty e^{-w/2}dw=2$, giving
\[\int_M^xe^{\lambda y}dF_Z(y)\,\leq\,Ve^{2\gamma}\left(\frac{\lambda}{2\gamma}\right)^{\gamma}+ 2Ve^{\lambda x}x^{-\gamma}\,=\,C_2\lambda^{\gamma}+C_3e^{\lambda x}x^{-\gamma}\]
for some constants $C_2, C_3$ depending only on $V$ and $\gamma$. Putting together both bounds for $R(\lambda,x)$ we arrive at
\begin{align*}\PP(S_m\geq x)&\leq\;mVx^{-\gamma}+e^{-\lambda x}\left(1+C_1Q(\lambda)+C_2\lambda^{\gamma}+ C_3e^{\lambda x}x^{-\gamma}\right)^m\\[3pt]
&\leq\;mVx^{-\gamma}+\exp\left(-\lambda x+mC_1Q(\lambda)+mC_2\lambda^{\gamma}+ mC_3e^{\lambda x}x^{-\gamma}\right).
\end{align*}
Our aim at this point to choose $\lambda$ such that the term on the right is small, which is achieved when taking
\[\lambda=\frac{1}{x}\log\left(\frac{x^{\gamma}}{m}\right),\]
so that $me^{\lambda x}x^{-\gamma}=1$. Assume first that $\gamma<1$ so $x=Lm^{\frac{1}{\gamma}}$ for $L$ large, for which $\lambda=\frac{\gamma\log(L)}{Lm^{\frac{1}{\gamma}}}$ is small, while $\lambda x=\gamma\log(L)$ is large so the assumption $M<x$ is justified. Now, since $\gamma<1$, $Q(\lambda)=\lambda^{\gamma}$ and hence we have
\[-\lambda x+mC_1Q(\lambda)+mC_2\lambda^{\gamma}+ mC_3e^{\lambda x}x^{-\gamma}\,=\,-\gamma\log(L)+(C_1+C_2)\left(\frac{\gamma\log(L)}{L}\right)^{\gamma}+C_3,\]
and since we are assuming $L$ large, we have $(\frac{\gamma\log(L)}{L})^{\gamma}\leq 1$ which finally gives
\begin{align*}\PP(S_m\geq x)&\leq\;mVx^{-\gamma}+\exp\left(-\lambda x+mC_1Q(\lambda)+mC_2\lambda^{\gamma}+ mC_3e^{\lambda x}x^{-\gamma}\right)\\[3pt]
&\leq\;VL^{-\gamma}+\exp\left(-\gamma\log(L)+C_1+C_2+C_3\right)\;=\;\auxc L^{-\gamma},
\end{align*}
which proves the first point of the theorem. Suppose now that $\gamma=1$ so that $x=Lm\log(m)$ for $L\geq \auxl$ for some $\auxl$ large, and hence $\lambda=\frac{1}{Lm\log(m)}\log(L\log(m))$ is small, while $\lambda x=\log(L\log(m))$ is large, so again the assumption $M<x$ is justified. For this choice of $\gamma$ we have $mQ(\lambda)=-m\lambda\log(\frac{\lambda}{2})$ which we can bound as 
\begin{align*}
-m\lambda\log(\tfrac{\lambda}{2})&=\frac{\log(L\log(m))}{L\log(m)}\log\left(\frac{2Lm\log(m)}{\log(L\log(m))}\right)\\[3pt]&\leq\frac{(\log(2L\log(m))^2}{L\log(m)}+\frac{\log(L\log(m))}{L}\leq C_4+\frac{\log(L\log(m))}{L}
\end{align*}
for some constant $C_4$, and hence we arrive at
\begin{align*}\PP(S_m\geq x)&\leq\;mVx^{-1}+\exp\left(-\lambda x+mC_1Q(\lambda)+mC_2\lambda+ mC_3e^{\lambda x}x^{-1}\right)\\[3pt]
&\leq\;\frac{V}{L\log(m)}+C_5\exp\Big({-}\log(L\log(m))+\frac{C_1}{L}\log(L\log(m))\Big)\;\leq\;\Big(\frac{\auxc}{L\log(m)}\Big)^{1-\frac{\auxl}{L}}
\end{align*}
for some constant $C_5$, and where the last inequality holds by choosing $\auxl$ larger than $C_1$ and also by choosing $\auxc$ large enough.

Suppose now that $\gamma>1$ so that $E_0:=\EE(Z_1)$ exists. In this case we can perform a similar computation to the one before to deduce that 
\[
    \PP(S_m-mE_0\geq x)\;\leq\;mVx^{-\gamma}+e^{-\lambda x}\Big(e^{-\lambda E_0}\int_1^xe^{\lambda y}dF_Z(y)\Big)^m,
\]
and we can divide the integral $\int_1^xe^{\lambda y}dF_Z(y)$ as before so that
\[\int_1^xe^{\lambda y}dF_Z(y)\,=\,\int_1^Me^{\lambda y}dF_Z(y)+\int_M^xe^{\lambda y}dF_Z(y)\]
where again $M=\frac{2\gamma}{\lambda}$ (and for our choice of small $\lambda$ below again we have $M < x$). Now, the main difference in this case is the treatment of the first term, for which we have
\begin{align*}
    \int_1^Me^{\lambda y}dF_Z(y)&=\,\int_1^MdF_Z(y)+\lambda\int_1^MydF_Z(y)+\int_1^M\left(e^{\lambda y}-1-\lambda y\right)dF_Z(y)\\[3pt]&\leq\,1+\lambda E_0-\left(e^{\lambda y}-1-\lambda y\right)(1-F_Z(y))\bigg|^M_1+\lambda\int_1^M\left(e^{\lambda y}-1\right)(1-F_Z(y))dy\\[3pt]&\leq 1+\lambda E_0+\left(e^{\lambda }-1-\lambda\right)+\lambda V\int_1^M\left(e^{\lambda y}-1\right)y^{-\gamma}dy\\[3pt]&\leq 1+\lambda E_0+\left(e^{\lambda }-1-\lambda\right)+\frac{\lambda V}{\gamma-1}\left(e^{\lambda}-1\right)+\frac{\lambda^2 V}{\gamma-1}e^{\lambda M}\int_1^My^{1-\gamma}dy\\[3pt]&\leq 1+\lambda E_0+\lambda^2+\frac{2\lambda^2 V}{\gamma-1}+C_1 W(\lambda),
\end{align*}
for some value $C_1$ depending on $\gamma$ alone, where we used that $\gamma>1$, that $\lambda$ is small, but also $\lambda M=2\gamma$, and where 
\[W(\lambda)=\left\{\begin{array}{cl}\lambda^{\gamma}&\text{ if }\gamma<2\\[3pt]-\lambda^2\log(\lambda)&\text{ if }\gamma=2\\[3pt]\lambda^2&\text{ if }\gamma>2\end{array}\right.\]
Since $\lambda$ is small we conclude that the $W(\lambda)$ is at least of the same order as the terms containing $\lambda^2$ and hence
\[\int_1^Me^{\lambda y}dF_Z(y)\;\leq\;1+\lambda E_0+3W(\lambda).\]
Treating the integral $\int_M^xe^{\lambda y}dF_Z(y)$ as in the case $\gamma\leq 1$ we finally obtain 
\begin{align*}\PP(S_m-mE_0\geq x)&\leq\;mVx^{-\gamma}+e^{-\lambda x-\lambda mE_0}\left(1+\lambda E_0+3W(\lambda)+C_2\lambda^{\gamma}+ C_3e^{\lambda x}x^{-\gamma}\right)^m\\[4pt]
&\leq\;mVx^{-\gamma}+\exp\left(-\lambda x+mC_4W(\lambda)+ mC_3e^{\lambda x}x^{-\gamma}\right).
\end{align*}
Now, since we are interested in the probability $\PP(S_m\geq Lm)$ for $L$ larger than some $\auxl$ which we can take larger than $2E_0$ we have
\[\PP(S_m\geq Lm)\,\leq\,\PP(S_m-mE_0\geq Lm/2),\]
and hence we can take $x=Lm/2$. For $\gamma<2$ we choose $\lambda=\frac{1}{x}\log(\frac{x^\gamma}{m})$ as before (which is small) for which $mC_3e^{\lambda x}x^{-\gamma}=C_3$ and hence
\[\PP(S_m-mE_0\geq x)\;\leq\;2^\gamma VL^{-\gamma}m^{1-\gamma}+\exp\left(-\log(L^\gamma m^{\gamma-1}/2^{\gamma})+3C_5m\lambda^\gamma+ C_3\right),  \]
but $m\lambda^\gamma=\frac{2^\gamma\log^\gamma(L^\gamma m^{\gamma-1}2^{-\gamma})}{L^\gamma m^{\gamma-1}}\leq 1$ for $L^\gamma m^{\gamma-1}$ large enough, and hence
\[\PP(S_m-mE_0\geq x)\;\leq\;\auxc L^{-\gamma}m^{1-\gamma}.\]
Suppose now that $\gamma\geq 2$ and choose $\lambda=\frac{\gamma}{x}\log(\frac{x}{\sqrt{m}})$ for which we have $mC_3e^{\lambda x}x^{-\gamma}=C_3m^{1-\frac{\gamma}{2}}\leq C_3$, giving 
\[\PP(S_m-mE_0\geq x)\;\leq\;2^\gamma VL^{-\gamma}m^{1-\gamma}+\exp\left(-\log(L^\gamma m^{\frac{\gamma}{2}}/2^{\gamma})+C_6mW(\lambda)+ C_3\right).\]
Now, if $\gamma=2$, then $W(\lambda)=\lambda^2\log(1/\lambda)$ so for some constant $C_7$
\[mW(\lambda)=\frac{16}{L^2m}\log^2(\tfrac{L \sqrt{m}}{2})\log\left(\frac{Lm}{4\log(L\sqrt{m}/2)}\right)\leq\frac{C_1}{L^2m}\log^3(L^2m)\leq 1\]
for large $L^2m$, while if $\gamma>2$, $W(\lambda)=\lambda^2$, and so 
\[mW(\lambda)=\frac{2\gamma^2}{L^2m}\log^2(\tfrac{L \sqrt{m}}{2})\leq 1\]
for large $L^2m$. In any case scenario, we obtain
\[\PP(S_m-mE_0\geq x)\;\leq\;2^\gamma VL^{-\gamma}m^{1-\gamma}+\auxc L^{-\gamma}m^{-\frac{\gamma}{2}},\]
but for $\gamma\geq 2$ we have $\frac{\gamma}{2}\leq\gamma-1$ and hence the second term dominates the first, giving the result.
\end{proof}

\small
\printbibliography

\end{document}